\documentclass[a4paper,11pt]{amsart}

\usepackage{amssymb,amsmath,amsthm,mathrsfs,enumerate,graphicx, color}
\usepackage[pdfpagelabels,colorlinks,linkcolor=blue,citecolor=black,urlcolor=blue]{hyperref}
\usepackage{stackengine,scalerel}
\usepackage{stix}
\usepackage{esint}
\usepackage{citeref}
\newtheorem{thm}{Theorem}[section]

\newtheorem{cor}[thm]{Corollary}
\newtheorem{lem}[thm]{Lemma}
\newtheorem{prop}[thm]{Proposition}
\newtheorem{defn}[thm]{Definition}
\newtheorem{rem}[thm]{Remark}

\newcommand{\lesi}{\lesssim}

\def\supp{{\mathop\mathrm{\,supp\,}}}

\newcommand{\N}{\mathfrak N}

\newcommand{\vc}{\infty}
\newcommand{\f}{\frac}

\newcommand{\B}{\dot{B}}
\newcommand{\FF}{\dot{F}}

\def\dfrac{\displaystyle\frac}

\newcommand{\RN}{{\mathbb R^N}}

\newcommand{\Z}{{\mathbb Z}}
\newcommand{\F}{{\mathcal F}}
\newcommand{\x}{{\rm x}}
\newcommand{\y}{{\rm y}}
\newcommand{\z}{{\rm z}}

\begin{document}

\title[Harmonic analysis in Dunkl settings]
{Harmonic analysis in Dunkl settings}

\authors

\author[The Anh Bui]{The Anh Bui}
\address{School of Mathematical and Physical Sciences, Macquarie University, NSW 2109,
	Australia}
\email{the.bui@mq.edu.au}

%\author[P.~D'Ancona]{Piero D'Ancona}
%\address{Piero D'Ancona: 
%Dipartimento di Matematica,
%Sapienza Universit\`{a} di Roma,
%Piazzale A.~Moro 2, 00185 Roma, Italy}
%\email{dancona@mat.uniroma1.it}

%\author[Xuan Thinh Duong]{Xuan Thinh Duong}
%\address{School of Mathematical and Physical Sciences, Macquarie University, %NSW 2109,
%	Australia}
%\email{xuan.duong@mq.edu.au}


\arraycolsep=1pt

\begin{abstract}
Let $L$ be the Dunkl Laplacian on the Euclidean space $\RN$ associated with a normalized root $R$ and a multiplicity function $k(\nu)\ge 0, \nu\in R$. In this paper, we first prove that the Besov and Triebel-Lizorkin spaces associated with the Dunkl Laplacian $L$ are identical to the Besov and Triebel-Lizorkin spaces defined in the space of homogeneous type $(\RN, \|\cdot\|, dw)$, where $dw(\x)=\prod_{\nu\in R}\langle \nu,\x\rangle^{k(\nu)}d\x$. Next, consider the Dunkl transform denoted by $\mathcal{F}$. We introduce the multiplier operator $T_m$, defined as $T_mf = \mathcal{F}^{-1}(m\mathcal{F}f)$, where $m$ is a bounded function defined on $\mathbb{R}^N$. Our second aim is to prove multiplier theorems, including the  H\"ormander multiplier theorem, for $T_m$ on the Besov and Tribel-Lizorkin spaces in the space of homogeneous type $(\RN, \|\cdot\|, dw)$.  Importantly, our findings present novel results, even in the specific case of the Hardy spaces.

\end{abstract}

\maketitle

\tableofcontents
\section{Introduction}

The Dunkl theory is an interesting subject in analysis and can be view as a  generalization of Fourier analysis in the Euclidean setting. This foundation can be tracked back by the influential paper of Dunkl  \cite{ADM, Dunkl2}. Since then, significant progress has been made in its development. See for example  \cite{Dunkl1, Dunkl3, Dunkl4,J, R, R2, R3, TX, TX2, JL} and the references therein. For further information, we refer the reader to \cite{R4, R5} for more information and references. Motivated by the extensive research conducted in the Dunkl settings, in this paper we will develop the harmonic analysis on the Dunkl setting. Firstly we would like to introduce the framework of the Dunkl setting in \cite{R4, R5} (see also \cite{A etal, ADH, DH, DH2}). \\

\medskip

On the Euclidean space $\mathbb R^N$ we consider a normalized root system $R$ and a multiplicity function $k(\nu)\ge 0,\nu\in R$. Define
\[
dw(\x) =\prod_{\nu \in R} |\langle \x,\nu\rangle|^{k(\nu)} d\x.
\]
In this paper, for $p\in (0,\vc)$ we denote by $L^p(dw)$ the set of all measurable functions $f$ such that 
\[
\|f\|_{L^p(dw)}: =\Big(\int_{\RN}|f(\x)|^pdw(\x)\Big)^{1/p}<\vc.
\]

We consider the Euclidean space $\RN$ with the usual scalar product and the induced norm. For a nonzero vector $\nu\in \RN$ the reflection $\sigma_\nu$ with respect to the hyperplane $\nu^{\perp}$ orthogonal to $\nu$ is given by
\begin{equation}
	\label{eq-reflection}
	\sigma_\nu\x =\x-2\f{\langle \x,\nu\rangle}{\|\nu\|^2}\nu.
\end{equation}
We denote by $G$ the Weyl group generated by the reflections $\sigma_\nu, \nu\in R$. 

A finite set $R\subset \RN\backslash \{0\}$ is called a root system if $\sigma_\nu(R)=R$ for every $\nu\in R$. We shall consider normalized reduced root systems, that is, $\|\nu\|^2=2$ for every $\nu\in R$. The finite group $G$ generated by the reflections $\sigma_\nu$ is called the Weyl group (reflection group) of the root system. A multiplicity function is a $G$-invariant function $k: R \to [0,\vc)$ which will be fixed and be nonnegative throughout this paper. We will denote by $\mathcal O(\x)=\{\sigma(\x): \sigma \in R\}$ the $G$-orbit of a point $\x\in\RN$, and $\mathcal O(E)=\bigcup_{\x \in E} \mathcal O(\x)$ the $G$-orbit of $E\subset \RN$.

The number $\mathfrak{N}:= N + \sum_{\nu\in R} k(\nu)$ is called the homogeneous dimension of the system, since
\begin{equation}
	\label{eq-homgoenous measure}
	w(B(t\x,tr)) = t^{\mathfrak{N}} w(B(\x,r)) \ \ \text{for} \ \ \x \in \RN, t,r>0,
\end{equation}
where  $B(\x, r) =\{\y\in \RN: \|\x-\y\|\le r\}$ denotes the (closed)
Euclidean ball centered at $\x$ with radius $r>0$. Observe that
\begin{equation}
	\label{eq-volume formula}
	w(B(\x,r)) \simeq r^{N}\prod_{\nu\in R} (|\langle\x,\nu\rangle|+r)^{k(\nu)}\gtrsim r^{\N},
\end{equation}
so $dw$  is a doubling measure, that is, there exists a constant $C>0$ such that
\begin{equation}
	\label{eq-doubling}
	w(B(\x,2r)) \le Cw(B(\x,r)) \ \ \text{for} \ \ \x \in \RN, r>0.
\end{equation}
Moreover, by \eqref{eq-volume formula},
\begin{equation}\label{eq-ratios on volumes of balls}
	\Big(\f{R}{r}\Big)^N\lesi \f{w(B(\x,R))}{w(B(\x,r))}\lesi \Big(\f{R}{r}\Big)^{\mathfrak{N}}\ \ \text{for} \ \ \x \in \RN, 0<r<R.
\end{equation}

Finally, define 
\begin{equation}\label{eq- d distance}
	d(\x,\y) =\min_{\sigma\in G}\|\x-\sigma(\y)\|
\end{equation}
as the distance between two $G$-orbits $\mathcal O(\x)$ and $\mathcal O(\y)$. Obviously, $\mathcal O(B(\x,r)) = \{\y\in \RN: d(\x,\y)<r\}$; moreover,
\[
w(B(\x,r)) \le w(\mathcal O(B(\x,r))) \le |G|w(B(\x,r))
\] 
for all $\x \in \RN$ and $r>0$. We also denote $B^d(\x,r):= \{\y\in \RN: d(\x,\y)<r\}$, i.e., $B^d(\x,r)=\mathcal O(B(\x,r))$.\\

Let $E(\x,\y)$ be the associated Dunkl kernel, which was introduced in \cite{Dunkl3}. It is known that $E(\x,\y)$ has a unique extension to a holomorphic function in $\mathbb C^N\times \mathbb C^N$. The Dunkl transform is defined by
\begin{equation}
	\label{eq-Dunkl transform}
	\mathcal Ff(\xi) = c_k^{-1}\int_{\mathbb R^N}E(-i\xi,\x)f(\x) dw(\x),
\end{equation}
where
\[
c_k =\int_{\mathbb R^N} e^{-\f{\|\x\|^2}{2}}dw(\x)>0.
\]
It is know that the Dunkl transform, which is initially defined for $f\in L^1(dw)$, is an isometry on $L^2(dw)$ and preserves the Schwartz class of functions $\mathcal S(\mathbb R^N)$. The inverse of the Dunkl transform is defined by
\[
\F^{-1}g(\x) =c_k^{-1}\int_{\RN} E(i\xi,\x)g(\xi) dw(\xi).
\] 

For $\x\in \RN$, the Dunkl translation $\tau_{\x}$ is defined by
\begin{equation}\label{eq-translation}
	\tau_{\x}f(\y) =c_k^{-1}\int_{\RN} E(i\xi, \x)E(i\xi, \y) \F f(\xi)dw(\xi)
\end{equation}
is bounded on $L^2(dw)$.

In what follows, we use the notation
\[
g(\x,\y) =\tau_{\x}g(-\y) =\tau_{-\y}g(x). 
\]
The Dunkl convolution of two appropriate functions is defined by
\begin{equation}\label{eq-convolution defn}
	f\ast g(\x) =c_k\F^{-1}[\F f \F g](\x)=\int_{\RN}\F f(\xi) \F g(\xi) E(\x,i\xi)dw(\xi),
\end{equation}
or equivalently,
\[
f\ast g(\x) =\int_{\RN}f(\y)\tau_{\x}g(-\y)dw(\y) =\int_{\RN}f(\y) g(\x,\y)dw(\y).
\]

We now consider the Dunkl operators $T_\xi$ defined by
\[
T_\xi f(\x)=\partial_\xi f(\x)+\sum_{\nu\in R}\f{k(\nu)}{2}\langle \nu,\xi\rangle\f{f(\x)-f(\sigma_\nu(\x))}{\langle \nu,\x\rangle},
\]
where   $\partial_\xi f$ denotes the directional derivative of $f$ in the direction $\xi$.

Set $T_j=T_{e_j}$, where $\{e_1,\ldots,e_N\}$ is the canonical basis for $\RN$. The Dunkl Laplacian is defined by
\[
\begin{aligned}
	L :&= -\sum_{j=1}^N T_{j}^2\\
	&= -\Delta_{\RN} -\sum_{\nu\in R}k(\nu)\delta_\nu f(\x),
\end{aligned}
\]
where $\Delta_{\RN}$ is the Laplacian on $\RN$ and $\delta_\nu$ is defined by
\[
\delta_\nu f(\x) =\f{\partial_\nu f(\x)}{\langle\alpha,\x\rangle} -\f{f(\x)-f(\sigma_\nu(\x))}{\langle\alpha,\x\rangle^2}.
\]

It is well-known  that the operator $L$ is a non-negative and  self-adjoint operator on $L^2(dw)$ and generates a semigroup $e^{-tL}$ whose kernel $h_t(\x,\y)$ is defined by
\[
h_t(\x,\y) = c_k^{-1}(2t)^{-\N/2} \exp\Big(-\f{\|\x\|^2+\|\y\|^2}{4t}\Big)E\Big(\f{\x}{\sqrt{2t}},\f{\y}{\sqrt{2t}}\Big), \ \ \ \x,\y\in \RN, t>0.
\]
See for example \cite{AH, R}.

The paper focuses on two primary objectives. Firstly, we will investigate some characterizations of  the Besov spaces and Triebel-Lizorkin spaces in $(\RN, \|\cdot\|, dw)$ via the heat semigroup $e^{-tL}$. Secondly, we will prove the boundedness of the multiplier of the Dunkl transform on these function spaces.\\

\medskip

For the first aim, we would like to go back the theory of Hardy spaces in $(\RN, \|\cdot\|, dw)$. Since $(\RN, \|\cdot\|, dw)$ is a homogneous type in the sense of Coifman and Weiss, one can define the Hardy spaces $H^p_{\rm CW}(dw)$ with $\f{\N}{\N+1}<p\le 1$ via the atomic decomposition. For convenience, we provide the definition of the Hardy spaces $H^p_{\rm CW}(dw)$ with $\f{\N}{\N+1}<p\le 1$ in \cite{CW}, which can be viewed as extensions to the Hardy spaces on $\mathbb R^N$ in the influential paper of Fefferman and Stein \cite{FS} (see also \cite{SW, St}). For $0<p\le 1$, we  say that a function $a$ is a $(p,2)$ atom if there exists a ball $B$ such that
\begin{enumerate}[(i)]
	\item supp $a\subset B$;
	\item $\|a\|_{L^2(dw)}\leq w(B)^{1/2-1/p}$;
	\item $\displaystyle \int a(\x)dw(\x)=0$.
\end{enumerate}
For $p=1$ the atomic Hardy space $H^1_{\rm CW}(dw)$ is defined as follows. We say that a function $f\in H^1_{\rm CW} (dw)$, if $f\in L^1(dw)$ and there exist a sequence $(\lambda_j)_{j\in \mathbb{N}}\in \ell^1$ and a sequence of $(p,2)$-atoms $(a_j)_{j\in \mathbb{N}}$ such that $f=\sum_{j}\lambda_ja_j$. We set
$$
\|f\|_{H^1_{\rm CW}(dw) }=\inf\Big\{\sum_{j}|\lambda_j|: f=\sum_{j}\lambda_ja_j\Big\},
$$
where the infimum is taken over all possible atomic decomposition of $f$.

For $0<p<1$, as in \cite{CW}, we need to introduce the Lipschitz space $\mathfrak{L}_\alpha$. We say that the function $f\in \mathfrak{L}_\alpha$ if there exists a constant $c>0$, such that
$$
|f(\x)-f(\y)|\leq cw(B)^{\alpha}
$$
for all ball $B$ and $\x, \y\in B$. The best constant $c$ above can be taken to be the norm of $f$ and is denoted by $\|f\|_{\mathfrak{L}_\alpha}$.

Now let $0<p<1$ and $\alpha=1/p-1$. We say that a function $f\in H^p_{\rm CW} (dw)$, if $f\in (\mathfrak{L}_{\alpha})^*$ and there is a sequence $(\lambda_j)_{j\in \mathbb{N}}\in \ell^p$ and a sequence of $(p,2)$-atoms $(a_j)_{j\in \mathbb{N}}$ such that $f=\sum_{j}\lambda_ja_j$. Furthermore, we set
$$
\|f\|_{H^p_{\rm CW}(dw) }=\inf\Big\{\Big(\sum_{j}|\lambda_j|^p\Big)^{1/p}: f=\sum_{j}\lambda_ja_j\Big\},
$$
where the infimum is taken over all possible atomic decomposition of $f$.

Apart from the Hardy spaces $H^p_{\rm CW} (dw)$, we can consider the Hardy spaces $H^p_{L}(dw)$ associated to the Dunkl Laplacian operator $L$.  
For $0<p\le 1$, the Hardy space $H^p_{L}(dw)$ is defined as the completion of the set
\begin{equation*}
	\left\{f\in L^2(dw): \mathcal{S}_Lf\in L^p(dw) \right\}
\end{equation*}
under the norm $\|f\|_{H^p_{L}(dw)}=\|\mathcal{S}_Lf\|_{L^p(dw)}$ where
\[
\mathcal{S}_Lf(x)=\Big[\int_0^\vc\int_{d(\x,\y)<t}|t^2Le^{-tL^2}f(\y)|^2\f{dw(\y) dt}{tw(B(\x,t))}\Big]^{1/2}.
\]
It is important to highlight  that  in the definition of the square function $\mathcal{S}_L$ we employ the distance $d$ defined by \eqref{eq- d distance} instead of the Euclidean distance $\|\cdot\|$. A good reason for this choice  is that while the heat kernel $h_t(\x,\y)$ does not satisfy the Gaussian upper bound with respect to the Euclidean distance $\|\cdot\|$, it satisfies the Gaussian upper bound with respect to the distance $d$. See Lemma \ref{lem-heat kernel estimate}. As a result, the theory of the Hardy space $H^p_{L}(dw)$ follows directly from existing results in \cite{HLMMY, SY, ADM, DY}. 

\bigskip

So it is natural to question if the Hardy spaces $H^p_{\rm CW}$ and $H^p_{L}(dw)$ are the same when $\f{\N}{\N+1}<p\le 1$. This question is not straightforward since the Hardy space $H^p_{\rm CW}$ corresponds to the Euclidean distance $\|\cdot\|$, while $L$ does not satisfy the Gaussian upper bound with respect to the Euclidean distance; moreover, the Hardy space $H^p_{L}(dw)$ corresponds to the Euclidean distance $d$. For the case $p=1$, a positive answer was carried out in \cite{DH2}. It remains  unclear if the approach in \cite{DH2} can be extended to the case $\f{\N}{\N+1}<p<1$. \\
\medskip

Motivated by the aforementioned research, our first primary objective is to establish a more comprehensive result that encompasses not only the Hardy spaces but also the Besov and Triebel-Lizorkin spaces in $(\RN, \|\cdot\|, dw)$. It is worth noting that the theory of Besov and Triebel-Lizorkin spaces holds significant interest in the field of harmonic analysis and has garnered considerable attention. See for example \cite{AWYY, HS, HMY, HY, HWY, WHYY} and the references therein. In the light of the fact that $(\RN, \|\cdot\|, dw)$ is a space of homogeneous type satisfying the reverse doubling property (i.e., the first inequality in \eqref{eq-ratios on volumes of balls}), we can define the Besov space $\B^{s}_{p,q}(dw)$ with $s\in (-1,1), p(s,1)<p<\vc$, $0<q<\vc$ and  the Triebel-Lizorkin space $\FF^{s}_{p,q}(dw)$ with $s\in (-1,1), p(s,1)<p,q<\vc$ (see Definition \ref{defn-Besov and TL spaces}),  following the framework established in \cite{HMY}. Here,  $p(s,1)=\max\{\f{\N}{\N+1}, \f{\N}{\N+s+1}\}$.  As far as we know, these ranges of the indices $s, p, q$ are the best in the setting of the space of homogeneous type except from the case $p=\vc$ or $q=\vc$. 

\medskip

Moreover, based on the general findings in \cite{BBD}, we can also define the Besov space $\B^{s,L}_{p,q}(dw)$  and the Triebel-Lizorkin space $\FF^{s,L}_{p,q}(dw)$ associated to the Dunkl Laplacian operator $L$. See Definition \ref{defn 1}.  It is intriguing to observe that these function spaces exhibit similar properties to their counterparts in the Euclidean setting. For example, $\FF^{0,L}_{p,2}(dw)$ coincide with the $L^p(dw)$ space if $1<p<\vc$ and $\FF^{0,L}_{p,2}(dw)$ coincide with the Hardy space $H^{p}_L(dw)$ for $p\in (0,1]$. For further properties such as atomic decomposition and interpolation we refer to Section 3.2 (see also \cite{BBD}). Drawing inspiration from these two research streams, our primary result is as follows.
\begin{thm}
	\label{thm-coincidence Besov and TL spaces}
	\begin{enumerate}[{\rm (a)}]
		\item For $s\in (-1,1)$, $p(s,1)<p< \vc$ and $0<q< \vc$, the Besov spaces $\dot{B}^s_{p,q}(dw)$ and $\dot{B}^{s, L}_{p,q}(dw)$ are identical with the equivalent norms.
		
		\item For $s\in (-1,1)$ and $p(s,1)<p, q< \vc$, the Triebel-Lizorkin   spaces $\dot{F}^s_{p,q}(dw)$ and $\dot{F}^{s, L}_{p,q}(dw)$ are identical with the equivalent norms.
	\end{enumerate}
\end{thm}
To evaluate the impact of Theorem  \ref{thm-coincidence Besov and TL spaces}, we only consider the particular case when $s=0, q=2$ and $\f{\N}{\N+1}<p\le 1$. In this case, Theorem \ref{thm-coincidence Besov and TL spaces} (b) reads $\dot{F}^0_{p,2}(dw)\equiv \dot{F}^{0,L}_{p,2}(dw)$. This, along with the fact that $\dot{F}^0_{p,2}(dw)\equiv H^p_{\rm CW}(dw)$ (see \cite[Remark 5.5]{GLY2})  and $\dot{F}^{0,L}_{p,2}(dw)\equiv H^p_L(dw)$ (see Theorem \ref{equiv-Hardy}), implies that $H^p_{\rm CW}(dw)\equiv H^p_L(dw)$ for $\f{\N}{\N+1}<p\le 1$. This, along with the maximal function characterizations of the Hardy space $H^p_L(\RN)$ via heat kernels of $L$ in \cite{SY}, extends the result in \cite{DH2, WLL} to the best range of $p$ in the theory of Hardy space on the space of homogeneous type.  In addition, Theorem \ref{thm-coincidence Besov and TL spaces}, together with known results in \cite{BBD}, deduces the heat kernel characterizations for the Besov spaces $\dot{F}^{s, L}_{p,q}(dw)$ and Triebel-Lizorkin spaces. See Corollary \ref{cor - heat kernel characterization}.\\

One of the main difficulties of the theory of the Besov and Triebel-Lizorkin spaces is the constructions of distributions. See for example \cite{KP, BBD}. In the Dunkl setting, with the ranges of indices in Theorem \ref{thm-coincidence Besov and TL spaces} we can use the distributions in \cite{HMY, HS} to define our function spaces. See Theorem \ref{thm 1}. This leads to some advantages to prove the coincidence between the function spaces associated to $L$ and the function spaces in the space $(\RN, \|\cdot\|, dw)$. Moreover, our approach relies on the atomic decomposition of the Besov and Triebel-Lizorkin spaces. See Section 3. We note that even the atomic decomposition results for the spaces $\dot{B}^s_{p,q}(dw)$ and $\dot{F}^s_{p,q}(dw)$ are new in the literature. See Theorems \ref{thm1-Besov atomic} and \ref{thm1-TL atomic}.

\medskip

The second main result focuses on multiplier theorems for the Dunkl transform. Let us provide a brief overview of the problem. Let $m$ be a bounded function on $\RN$. We define
\[
T_mf = \F^{-1}(m\F f).
\]
In the particular case when the multiplicity function $k\equiv 0$, the Dunkl transform turns out to be the Fourier transform, and the Dunkl multiplier $T_m$ boils down to the Fourier multiplier. Regarding the boundedness of the Dunkl multiplier $T_m$ we recall the main result in \cite{DH}. In what follows, we define $\|f\|_{W^q_s(\RN)}=\|(I-\Delta)^{s/2}f\|_{L^q(\RN,d\x)}$ for $s>0$ and $q\in (1,\vc]$, and $a\vee b =\max\{a,b\}$ and $a\wedge b = \min \{a,b\}$.

\begin{thm}[\cite{DH}]\label{mainthm-spectralmultipliers DH} 	Let $\psi$ be a smooth radial function on $\RN$ such that supp $\psi\subset \{\xi: 1/4\le \|\xi\|\le 4\}$ and $\psi\equiv 1$ on $\{\xi: 1/2\le \|\xi\|\le 2\}$. If $m$ is a function on $\RN$ satisfying the following condition
	\begin{equation}
		\label{smoothness condition - DH}
		\sup_{t>0}\|\psi(\cdot)m(t\cdot)\|_{W^{2}_\alpha(\RN)}<\infty 
	\end{equation} 
	for some $\alpha>\N$. Then we have:
	\begin{enumerate}[\rm (a)]
		\item the  multiplier $T_m$ is weak type of $(1,1)$;
		\item the  multiplier $T_m$ is bounded on $L^p(dw)$ for $1<p<\vc$;
		\item the  multiplier $T_m$ is bounded on the Hardy space $H^1_{\rm CW}(dw)$.
	\end{enumerate}
	Moreover, if we assume the additional condition \begin{equation}\label{eq-L1 uniform}
		\sup_{\y\in \RN}\|\tau_{\y}f\|_{L^1(dw)}\le C\|f\|_{L^1(dw)},
	\end{equation} 
	then the conclusions (a), (b) and (c) above hold true provided that \eqref{smoothness condition} is valid for some $\alpha>\N/2$.
\end{thm}
The condition \eqref{smoothness condition - DH} with $\alpha>\N$ is unexpected but is reasonable since we need an extra order of $\N/2$ in the proof of the boundedness of the Dunkl translation on $L^p(dw), p\ne 2$. The extra number of derivatives of $\N/2$ can be removed under the condition \eqref{eq-L1 uniform}. This condition holds true in the rank-one case and  the product case as in \cite{R}.

\medskip

The second main aim in this paper is as follows. Firstly, we aim to reduce the number of the derivative imposed in the function $m$ to $\N/2$ without the extra condition \eqref{eq-L1 uniform}; however, we have to use the norm $\|\cdot\|_{W^\vc_\alpha(\RN)}$ rather than $\|\cdot\|_{W^2_\alpha(\RN)}$. Secondly, we extend the boundedness to the Besov space $\B^{s}_{p,q}(dw)$ and Triebel-Lizorkin space $\FF^{s}_{p,q}(dw)$. More precisely, we have the following result.

\begin{thm}\label{mainthm-spectralmultipliers-homogeneous spaces} 	Let $\alpha>\f{\N}{2}$. Let $\psi$ be a smooth radial function  defined on $\RN$  such that supp $\psi\subset \{\xi: 1/4\le \|\xi\|\le 4\}$ and $\psi\equiv 1$ on $\{\xi: 1/2\le \|\xi\|\le 2\}$. If $m$ is a function on $\RN$ satisfying the following condition
	\begin{equation}
		\label{smoothness condition}
		\sup_{t>0}\|\psi(\cdot)m(t\cdot)\|_{W^{\vc}_\alpha(\RN)}<\infty.
	\end{equation} 
	Then the following hold true.
	\begin{enumerate}[\rm (a)]
		\item The  multiplier $T_m$ is bounded on $\FF^{s}_{p,q}(dw)$ provided that $s\in (-1,1)$, $p(s,1)<p,q <\vc$ and $\alpha> \N(\f{1}{1\wedge p\wedge  q}-\f{1}{2}) $, i.e.,
		\[
		\|T_m\|_{\FF^{s}_{p,q}(dw)\to \FF^{s}_{p,q}(dw)}\lesi |m(0)|+\sup_{t>0}\|\psi(\cdot)m(t\cdot)\|_{W^{\vc}_\alpha(\RN)}.
		\]
		\item The multiplier $T_m$ is bounded on $\B^{s}_{p,q}(dw)$ provided that $s\in (-1,1)$, $p(s,1)<p, q<\vc$ and $\alpha> \N(\f{1}{1\wedge p\wedge  q}-\f{1}{2}) $, i.e.,
		$$
		\|T_m\|_{\B^{s}_{p,q}(dw)\to \B^{s}_{p,q}(dw)}\lesi |m(0)| +\sup_{t>0}\|\psi(\cdot)m(t\cdot)\|_{W^{\vc}_\alpha(\RN)}.
		$$		
	\end{enumerate}
	Moreover, if  \eqref{eq-L1 uniform} is assumed,  the estimates (a) and (b)  hold true with $W^{2}_\alpha(\RN)$ taking place of $W^{\vc}_\alpha(\RN)$ in \eqref{smoothness condition}.
\end{thm}
Note that our approach can prove that (a) and (b) hold true under the condition \eqref{smoothness condition - DH} (without the additional condition \eqref{eq-L1 uniform}) with $\alpha> \N(\f{1}{1\wedge p\wedge  q}-\f{1}{2})+\f{\N}{2}$. 

In the particular case when $s=0, q=2$ and $p\in (0,1]$, the estimate (a) in Theorem \ref{mainthm-spectralmultipliers-homogeneous spaces} boils down to
\[
\|T_m\|_{\FF^{0}_{p,2}(dw)\to \FF^{0}_{p,2}(dw)}\lesi |m(0)|+\sup_{t>0}\|\psi(\cdot)m(t\cdot)\|_{W^{\vc}_\alpha(\RN)},
\]
as long as $\alpha > \N (\f{1}{p}-\f{1}{2} )$. This, along with the fact that $\FF^{0}_{p,2}(dw)\equiv H^p_{\rm CW}(dw)$ for $\f{\N}{\N+1}<p\le 1$ (see \cite[Remark 5.5]{GLY2}), implies that 
\[
\|T_m\|_{H^p_{\rm CW}(dw)\to H^p_{\rm CW}(dw)}\lesi |m(0)|+\sup_{t>0}\|\psi(\cdot)m(t\cdot)\|_{W^{\vc}_\alpha(\RN)},
\]
as long as $\alpha > \N (\f{1}{p}-\f{1}{2} )$  and $\f{\N}{\N+1}<p\le 1$. It is worth noting that this estimate improves (c) in Theorem \ref{mainthm-spectralmultipliers DH} for $\f{\N}{\N+1}<p<1$. It is important to emphasize that it seems that the approach in \cite{DH} can be applied to prove the estimate for $p<1$. \\

In addition, since $\FF^{0}_{p,2}(dw)\equiv L^p(dw)$ (see \cite{HMY}), the estimate (a) reads
\[
\|T_m\|_{H^p_{\rm CW}(dw)\to H^p_{\rm CW}(dw)}\lesi |m(0)|+\sup_{t>0}\|\psi(\cdot)m(t\cdot)\|_{W^{\vc}_\alpha(\RN)},
\]
as long as $\alpha > \N/2$. This improves the main result in \cite[Theorem 3]{Sol} even when $N=1$.\\

\medskip

Our techniques regarding the proof of Theorem \ref{mainthm-spectralmultipliers-homogeneous spaces} are built upon the framework of function spaces associated to the operator $L$ in \cite{BBD, BD}. We begin by proving the sharp estimates for  the multiplier $T_m$ on the Hardy spaces $H^p_L(dw)$, while obtaining the (un-sharp) estimates for  the multiplier $T_m$ on the Besov/Triebel-Lizorkin spaces. The sharpness will follow by utilizing the duality and interpolation arguments as in \cite{BD}. 

Once we have established the boundedness of $T_m$ on the Besov and Triebel-Lizorkin spaces associated with $L$, as stated in Theorem \ref{mainthm-spectralmultipliers-space adapted to L}, the conclusion of Theorem \ref{mainthm-spectralmultipliers-homogeneous spaces} follows directly from this result, in conjunction with Theorem \ref{thm-coincidence Besov and TL spaces}. It is worth noting that the boundedness of the multiplier operator $T_m$ on the Besov and Triebel-Lizorkin spaces associated with $L$, with the full range of indices as outlined in Theorem \ref{mainthm-spectralmultipliers-space adapted to L}, is of considerable significance in its own right.
\medskip

We  now consider the Riesz transform define by $\mathcal R_j:=T_jL^{-1/2}, j=1,\ldots, N$. It is clear that for each $j=1,\ldots, N$ the Riesz transform can be viewed as the Dunkl multiplier $T_{m_j}$ with $m_j(\xi) = -i\f{\xi_j}{\|\xi\|}$. See for example \cite{TX, A etal}.  As a direct consequence of Theorem \ref{thm-coincidence Besov and TL spaces}, we have:
\begin{cor}
	For each $j=1,\ldots, N$, the Riesz transform $\mathcal R_j:=T_jL^{-1/2}$ is bounded on  the Triebel-Lizorkin space $\FF^{s}_{p,q}(dw)$ with $s\in (-1,1)$, $p(s,1)<p,q <\vc$, and is bounded on the Besov space with  $s\in (-1,1)$, $p(s,1)<p<\vc, 0<q<\vc$.
\end{cor}
By using Theorem \ref{mainthm-spectralmultipliers-space adapted to L} we also obtain the boundedness of the Riesz transforms $\mathcal R_j$ on the Besov and Triebel-Lizorkin associated to the operator $L$ with full ranges of indices $s\in \mathbb R$ and $0<p,q<\vc$.\\

\medskip

The organization of the paper is as follows. In Section 2, we first recall results in \cite{HK} on the dyadic systems in the Dunkl setting. Then  some kernel estimates of the functional calculus of $L$ will be establish. In addition, we also prove some Calder\'on reproducing formulae related to the functional calculus of $L$ under the distributions in \cite{HMY, HS}. Section 3 will provide the proof of Theorem \ref{thm-coincidence Besov and TL spaces}. The proof of Theorem \ref{mainthm-spectralmultipliers-homogeneous spaces} will be given in Section 4.\\

\medskip

Throughout this paper, we use $C$ to denote positive constants, which are independent of the main
parameters involved and whose values may vary at every occurrence.
By writing $f \lesssim g$, we mean that $f \leq Cg$. We also use
$f \simeq g$ to denote that $C^{-1}g \leq f \leq C g$. We also use the following notations through the paper, $a\vee b =\max\{a,b\}$ and $a\wedge b = \min \{a,b\}$.

\section{Preliminaries}

\subsection{Dyadic cube systems in the Dunkl setting}
Let $X$ be a metric space, with metric $d$ and
$\mu$ is a nonnegative Borel measure on $X$, which satisfies the doubling property below. For $\x\in X$ and $r>0$ we set $B(\x, r)=\{\y\in X:d(\x,\y)<r\}$ to be the open ball with radius $r >0$ and center $\x\in
X$. The doubling
property of $\mu$ provides a constant $C>0$ such that
\begin{equation}\label{doubling1}
	\mu(B(\x,2r))\leq C\mu(B(\x,r))
\end{equation}
for all $\x\in X$ and $r>0$.\\
%The doubling property (\ref{doubling1}) yields a constant $n>0$ so that
%\begin{equation}\label{doubling2}
%	V(x,\lambda r)\leq C\lambda^nV(x,r),
%\end{equation}
%for all $\lambda\geq 1, x\in X$ and $r>0$.\\

Let  $0<r<\vc$. The  Hardy-Littlewood maximal function $\mathcal{M}_{r}$ is defined by
\begin{equation}\label{eq-maximal function}
	\mathcal{M}_{r} f(\x)=\sup_{\x\in B}\Big(\f{1}{\mu(B)}\int_B|f(\y)|^r d\mu(\y)\Big)^{1/r}
\end{equation}
where the sup is taken over all balls $B$ containing $\x$. We will drop the subscripts $r$ when $r=1$.

Let  $0<r<\vc$. It is well-known that 
\begin{equation}
	\label{boundedness maximal function}
	\|\mathcal{M}_{r} f\|_{p}\lesi \|f\|_{p}
\end{equation}
for all $p>r$.

We recall the Fefferman-Stein vector-valued maximal inequality  in \cite{GLY}. For $0<p<\vc$, $0<q\leq \vc$ and $0<r<\min \{p,q\}$, we then have for any sequence of measurable functions $\{f_\nu\}$,
\begin{equation}\label{FSIn}
	\Big\|\Big(\sum_{\nu}|\mathcal{M}_rf_\nu|^q\Big)^{1/q}\Big\|_{p}\lesi \Big\|\Big(\sum_{\nu}|f_\nu|^q\Big)^{1/q}\Big\|_{p}.
\end{equation}

The Young's inequality and \eqref{FSIn} imply the following  inequality: If $\{a_\nu\} \in \ell^{q}\cap \ell^{1}$, then 
\begin{equation}\label{YFSIn}
	\Big\|\sum_{j}\Big(\sum_\nu|a_{j-\nu}\mathcal{M}_r f_\nu|^q\Big)^{1/q}\Big\|_{p}\lesi \Big\|\Big(\sum_{\nu}|f_\nu|^q\Big)^{1/q}\Big\|_{p}.
\end{equation}
The following result  on the existence of a system of dyadic cubes in a space of homogeneous type is just a combination of \cite[Theorem 2.2]{HK} and \cite[Proposition 2.5]{CKP} .
\begin{lem}\label{lem-dyadic cube}
	Fix $\delta \in (0,1)$ such that $\delta \le 1/24$. Then there exist  a family of set $\{Q_\alpha^k: k\in \Z, \alpha\in I_k\}$ and a set of points $\{\x_{Q_\alpha^k}: k\in \Z, \alpha\in I_k\}$ and satisfying
	\begin{enumerate}[{\rm (i)}]
		\item $d(\x_{Q_\alpha^k},\x_{Q_\beta^k}) \ge \delta^k$ for all $k\in \mathbb Z$ and $\alpha,\beta \in I_k$ with $\alpha\ne \beta$;
		\item $\min_{\alpha\in I_k}d(\x,\x_{Q_\alpha^k}) \le \delta^k$ for all $\x\in X$;
		\item for any $k\in \Z$, $\bigcup_{\alpha\in I_k} Q^k_\alpha =\RN$ and $\{Q_\alpha^k:  \alpha\in I_k\}$ is disjoint;
		\item if $k,\ell\in \Z$ and $k\ge \ell$, then either $Q_\alpha^k\subset Q_\beta^\ell$ or $Q_\alpha^k\cap  Q_\beta^\ell=\emptyset$ for every $\alpha\in I_k$ and $\beta\in I_\ell$;
		\item for any $k\in \Z$ and $\alpha\in I_k$, $B(\x_{Q_\alpha^k},\delta^k/6)\subset Q_\alpha^k\subset B(\x_{Q_\alpha^k},2\delta^k)$.
	\end{enumerate}
\end{lem}

\bigskip

We now fix $\delta \in (0,1/24)$ through the paper. Return to the Dunkl setting, we have two metrics defined in $\RN$: the Euclidean metric induced by the Euclidean norm $\|\cdot\|$ and the orbit distance $d$ defined by \eqref{eq- d distance}. 

\medskip

Since both $(\RN, \|\cdot\|, dw)$ is a space of homogeneous type, we can apply Lemma \ref{lem-dyadic cube} to extract a system of dyadic cubes corresponding to each metric. For the homogeneous space $(\RN, \|\cdot\|, dw)$, denote by $\mathscr D :=\{Q^k_\alpha: k\in \Z, \alpha \in I_k\}$ the set constructed by Lemma \ref{lem-dyadic cube}. The set $\mathscr D$ is called the system of dyadic cubes in $(\RN, \|\cdot\|, dw)$. For each $k\in \Z$, we denote $\mathscr D_k = \{Q_\alpha^k: \alpha \in I_k\}$. For each $k\in \Z$ and $\alpha\in I_k$, we denote $\ell(Q_\alpha^k)=\delta^k$ and $\lambda Q_\alpha^k = B(\x_{Q_\alpha^k},\lambda \delta^k), \lambda>0$.

Although $(\RN, d, dw)$ is not a space of homogeneous type, $(\RN/G, d, d\mu)$ is a space of homogeneous type, where $\mu(E)=w(\cup_{\mathcal O(\x)\in E}\mathcal O(\x))$. Therefore, there exists a family $\mathscr D^d :=\{Q^{k,d}_\alpha: k\in \Z, \alpha \in I^d_k\}$ satisfying (i)-(v) in Lemma \ref{lem-dyadic cube}. The set $\mathscr D^d$ is called the system of dyadic cubes in $(\RN, d, dw)$. For each $k\in \Z$, we denote $\mathscr D^d_k = \{Q_\alpha^{k,d}: \alpha \in I^d_k\}$. For each $k\in \Z$ and $\alpha\in I^d_k$, we denote $\ell(Q_\alpha^{k,d})=\delta^k$ and $\lambda Q_\alpha^{k,d} = B^d(\x_{Q_\alpha^{k,d}},\lambda \delta^k):=\mathcal O(B(\x_{Q_\alpha^{k,d}},\lambda \delta^k)), \lambda>0$. Note that we use the subscript $d$ in dyadic cubes to indicate the metric $d$.

In what follows, for $r>0$ we still denote by $\mathcal M_r$ the maximal function defined by \eqref{eq-maximal function} in  $(\RN, \|\cdot\|, dw)$ and by $\mathcal M^d_r$ the maximal function defined by \eqref{eq-maximal function} in  $(\RN, d, dw)$. Obviously,
\[
\mathcal M^d_rf(\x) \lesi \sum_{y\in \mathcal O(\x)}\mathcal M_r f(\y)
\]
for all $\x\in \RN$. Hence, $\mathcal M^d_r$ also enjoys the inequalities \eqref{boundedness maximal function}, \eqref{FSIn} and \eqref{YFSIn}.

We have the following technical lemmas.
\begin{lem}[\cite{FJ2}]\label{lem1- JF lem}
	Let $\mathscr D^d = \{\mathscr D^d_k\}_{k\in \Z}$ be the system of dyadic cubes in $(\RN, d, dw)$. Let  $M>\mathfrak{N}$, $\kappa\in [0,1]$, and $\eta, k \in \mathbb{Z}$ and $k\geq \eta$. Assume that  $\{f_{Q^d}\}_{{Q^d}\in \mathscr{D}^d_k}$ is a  sequence of functions satisfying
	\begin{equation*} 
		|f_{Q^d}(\x)|\lesi \Big(\f{w(Q^d)}{w(B(\x_{Q^d},\delta^{\eta}))}\Big)^\kappa\Big(1+\f{d(\x,\x_{Q^d})}{\delta^{\eta}}\Big)^{-M}.
	\end{equation*}
	Then for $\f{\mathfrak{N}}{M}<r\leq 1$ and a sequence of numbers $\{s_{Q^d}\}_{Q^d\in \mathscr{D}^d_k}$, we have
	$$
	\sum_{Q^d\in \mathscr{D}^d_k}|s_{Q^d}|\,|f_{Q^d}(\x)|\lesi \delta^{-\mathfrak{N}(k-\eta)(1/r-\kappa)}\mathcal{M}^d_{r}\Big(\sum_{Q^d\in \mathscr{D}^d_k}|s_{Q^d}|\chi_{Q^d}\Big)(\x).
	$$
	The statement still holds true if we replace all dyadic cubes $\mathscr D^d = \{\mathscr D^d_k\}_{k\in \Z}$, the distance $d$ and the maximal function $\mathcal M_r^d$ by the dyadic cubes $\mathscr D = \{\mathscr D_k\}_{k\in \Z}$, the distance $\|\cdot\|$ and the maximal function $\mathcal M_r$, respectively. 
\end{lem}

\begin{lem}\label{lem1- thm2 atom Besov}
	Let $\mathscr D = \{\mathscr D_k\}_{k\in \Z}$ be the system of dyadic cubes in $(\RN, \|\cdot\|, dw)$. Let  $M>\mathfrak{N}$, $\kappa\in [0,1]$, and $\eta, k \in \mathbb{Z}$, $k\geq \eta$. Assume that  $\{f_{Q}\}_{{Q}\in \mathscr{D}_k}$ is a  sequence of functions satisfying
	\begin{equation}\label{eq-bound of fQ}
		|f_{Q}(\x)|\lesi \Big(\f{w(Q)}{w(B(\x_{Q},\delta^{\eta}))}\Big)^\kappa\Big(1+\f{d(\x,\x_{Q})}{\delta^{\eta}}\Big)^{-M}.
	\end{equation}
	where $d$ is the distance defined by \eqref{eq- d distance}. Then for $\f{\mathfrak{N}}{M}<r\leq 1$ and a sequence of numbers $\{s_{Q}\}_{Q\in \mathscr{D}_k}$, we have
	$$
	\sum_{Q\in \mathscr{D}_k}|s_{Q}|\,|f_{Q}(\x)|\lesi \delta^{-\mathfrak{N}(k-\eta)(1/r-\kappa)}\mathcal{M}^d_{r}\Big(\sum_{Q\in \mathscr{D}_k}|s_{Q}|\chi_{Q}\Big)(\x).
	$$
\end{lem}
\begin{proof}
	Before coming to the proof we would like to point out the main issue of the lemma lies on the fact that the dyadic system corresponds to the metric $\|\cdot\|$, while the estimate \eqref{eq-bound of fQ} corresponds to the distance $d$.

	Fix $\x\in \RN$. We set
	$$
	\mathcal{B}_0=\{Q\in \mathscr{D}_k: d(\x,\x_Q)\leq \delta^{\eta}\}, Q_0=\bigcup\limits_{Q\in \mathcal{B}_0} Q
	$$
	and
	$$
	\mathcal{B}_j=\{Q\in \mathscr{D}_k: \delta^{-j+\eta+1}<d(\x,\x_Q)\leq \delta^{-j+\eta}\}, Q_j= \bigcup\limits_{Q\in \mathcal{B}_j} Q, \ \ j\in \mathbb{N}_+.
	$$
	Then we write
	$$
	\begin{aligned}
		\sum_{Q\in \mathscr{D}_k}|s_Q|\,|f_Q(\x)|&=\sum_{j\in \mathbb{N}}\sum_{Q\in \mathcal{B}_j}|s_Q|\,|f_Q(\x)|\leq \sum_{j\in \mathbb{N}}\sum_{Q\in \mathcal{B}_j}|s_Q|\Big(\f{w(Q)}{w(B(\x_Q,\delta^{\eta}))}\Big)^{\kappa}\Big(1+\f{d(\x,x_Q)}{\delta^{\eta}}\Big)^{-M}\\
		&=:\sum_{j\in \mathbb{N}}E_j.
	\end{aligned}
	$$
	For each $j\in \mathbb{N}$, we have
	\begin{equation}\label{eq1-sec4}
		\begin{aligned}
			E_j&\lesi \sum_{Q\in \mathcal{B}_j} \delta^{jM}\Big(\f{w(Q)}{w(B(\x_Q,\delta^{\eta}))}\Big)^{\kappa}|s_Q| \leq \delta^{jM}\left[\sum_{Q\in \mathcal{B}_j} \Big(\f{w(Q)}{w(B(\x_Q,\delta^{\eta}))}\Big)^{\kappa r}|s_Q|^r\right]^{1/r}\\
			&\lesi \delta^{jM}\left\{\int_{Q_j}\left[\sum_{Q\in \mathcal{B}_j}\Big(\f{w(Q)}{w(B(x_Q,\delta^{\eta}))}\Big)^{\kappa} w(Q)^{-1/r}|s_Q|\chi_Q(y)\right]^rdw(y)\right\}^{1/r}\\
			&\lesi \delta^{jM}\left\{\f{1}{w(Q_j)}\int_{Q_j}\left[\sum_{Q\in \mathcal{B}_j} \Big(\f{w(Q_j)}{w(Q)}\Big)^{1/r}\Big(\f{w(Q)}{w(B(x_Q,\delta^{\eta}))}\Big)^{\kappa}|s_Q|\chi_Q(y)\right]^rdw(y)\right\}^{1/r}.
		\end{aligned}
	\end{equation}
	It is easy to see that $w(Q_j)\simeq w(B(\x,\delta^{\eta-j}))\simeq w(B(\x_Q,\delta^{\eta-j}))$, for each $Q\in \mathcal{B}_j$. Therefore,
	$$
	\begin{aligned}
		\Big(\f{w(Q_j)}{w(Q)}\Big)^{1/r}\Big(\f{w(Q)}{w(B(x_Q,\delta^{\eta}))}\Big)^{\kappa}&\lesi  \delta^{-j\mathfrak{N}/r}\delta^{-\mathfrak{N}(k-\eta)(1/r-\kappa)}.
	\end{aligned}
	$$
	Inserting this into (\ref{eq1-sec4}) gives
	$$
	\begin{aligned}
		E_j&\lesi \delta^{jM}\delta^{-j\mathfrak{N}/r}\delta^{-\mathfrak{N}(k-\eta)(1/r-\kappa)}\Big\{\f{1}{w(Q_k)}\int_{Q_k}\Big[\sum_{Q\in \mathcal{B}_k} |s_Q|\chi_Q(y)\Big]^rdw(\y)\Big\}^{1/r}\\
		&\lesi \delta^{j(M-\mathfrak{N}/r)}\delta^{-\mathfrak{N}(k-\eta)(1/r-\kappa)}\mathcal{M}^{d}\Big(\sum_{Q\in \mathcal{B}_k} |s_Q|\chi_Q\Big)(\x).
	\end{aligned}
	$$
	Since $r>\f{\mathfrak{N}}{M}$, we find that
	$$
	\sum_{j\in \mathbb{N}}E_j\lesi \delta^{-\mathfrak{N}(k-\eta)(1/r-\kappa)}\mathcal{M}^{d}\Big(\sum_{Q\in \mathscr{D}_k} |s_Q|\chi_Q\Big)(\x).
	$$
	
	This completes our proof.
\end{proof}

\subsection{Some kernel estimates}
In this section, we will prove some kernel estimates for the functional calculus of $L$. We now recall the following result in \cite{DH2}.

\begin{lem}[\cite{DH2}]\label{lem-heat kernel estimate for semigroups}
	\begin{enumerate}[\rm (a)]
		\item For $\x,\y\in \RN$ and $t>0$, $h_t(\x,\y)\ge 0$. In addition,
		\[
		\int_{\RN}h_t(\x,\y)dw(\y)=\int_{\RN}h_t(\y,\x)dw(\y)=1
		\]
		for all $\x\in \RN$ and $t>0$.
		
		\item For every nonnegative integer $m$ and every multi-indices $\alpha,\beta$ there are constants $C, c>0$ such that
		\[
		|\partial_t^m\partial^{\alpha}_{\x}\partial^{\beta}_{\y}h_t(\x,\y)|\le Ct^{-(m+|\alpha|/2+|\beta|/2)}\Big(1+\f{\|\x-\y\|}{\sqrt t}\Big)^{-2}\f{1}{w(B(\x,\sqrt t))}\exp\Big(\f{d(\x,\y)^2}{ct}\Big).
		\]
		Moreover, if $\|\y-\y'\|\le \sqrt t$, then
		\[
		|\partial_t^m h_t(\x,\y)-\partial_t^m h_t(\x,\y')|\le Ct^{-m}\f{\|\y-\y'\|}{\sqrt t}\Big(1+\f{\|\x-\y\|}{\sqrt t}\Big)^{-2}\f{1}{w(B(\x,\sqrt t))}\exp\Big(\f{d(\x,\y)^2}{ct}\Big).
		\]
	\end{enumerate}
\end{lem}
Note that the extra decay $\Big(1+\f{\|\x-\y\|}{\sqrt t}\Big)^{-2}$ in (b) plays an essential role in the proofs of our main results. 

\medskip

Since $L$ is a nonnegative self-adjoin operator on $L^2(dw)$, for each bounded Borel function $F: [0,\vc) \to \mathbb C$, we can defined 
\begin{align*}
	F(L) = \int_{0}^{\infty} F(\lambda)dE(\lambda)
\end{align*}
is bounded on $L^2 (dw)$, where $E(\lambda)$ is the spectral resolution of $L$.

The following lemma addresses some kernel estimates for the functional calculus of $L$.
\begin{lem}\label{lem-heat kernel estimate}
	Let $\varphi\in \mathscr{S}(\mathbb R)$ be an even function. Then the following estimates hold true. 
	\begin{enumerate}[\rm (a)]
		\item For every $M>0$ and every multi-indices $\alpha,\beta$, the kernel $\varphi(t\sqrt L)(\x,\y)$ of $\varphi(t\sqrt L)$ satisfies the following estimate
		\[
		|\partial^{\alpha}_{\x}\partial^{\beta}_{\y}\varphi(t\sqrt L)(\x,\y)|\lesi t^{-(|\alpha|+|\beta|)}\Big(1+\f{\|\x-\y\|}{ t}\Big)^{-2}\f{1}{w(B(\x, t+d(\x,\y)))}\Big(\f{t}{t+d(\x,\y)}\Big)^M
		\]
		for all $\x,\y\in \RN$ and $t>0$.
		
		\item  For every $M>0$, 
		\[
		|\varphi(t\sqrt L)(\x,\y)-\varphi(t\sqrt L)(\x,\y')|\lesi \f{\|\y-\y'\|}{t}\Big(1+\f{\|\x-\y\|}{t}\Big)^{-2}\f{1}{w(B(\x, t+d(\x,\y)))}\Big(\f{t}{t+d(\x,\y)}\Big)^M
		\]
		for all $t>0$ and  $\x,\y, \y'\in \RN$ with $\|\y-\y'\|\le t$.
		
		%\item For every $M>0$,  
		%\[
		%|\varphi(t\sqrt L)(\x,\y)-\varphi(t\sqrt L)(\x,\y')|\lesi %\f{\|\y-\y'\|}{t}\f{1}{w(B(\x, t+d(\x,\y)))}\Big(\f{t}{t+d(\x,\y)}\Big)^M
		%\]
		%for all $t>0$ and  $\x,\y, \y'\in \RN$ with $d(\y,\y') \le t$
		
		\item If $\varphi(0)=0$, then 
		\[
		\int_{\RN}\varphi(t\sqrt L)(\x,\y)dw(\y)=\int_{\RN}\varphi(t\sqrt L)(\y,\x)dw(\y)=0
		\]
		for all $\x\in \RN$ and $t>0$.
	\end{enumerate}
\end{lem}
We note that the estimates of Lemma \ref{lem-heat kernel estimate} without the extra decay $\displaystyle \Big(1+\f{\|\x-\y\|}{ t}\Big)^{-2}$ is standard. See  for example  \cite[Theorem 3.1]{KP}. The presence of  the extra decay needs special properties of the operator $L$.
\begin{proof}
	Item (c) follows from Lemma \ref{lem-heat kernel estimate for semigroups} (a) by applying \cite[Theorem 3.1]{KP}, while item (b) follows directly from (a).  Hence, it suffices to prove (a).
	
	It was proved in \cite{DH} that 
	\[
	Lh_t(\x,\y) = \f{\|\x-\y\|^2}{4t^2}h_t(\x,\y) -\f{N}{2t}h_t(\x,\y) -\f{1}{2t}\sum_{\nu\in R}k(\nu)h_t(\sigma_\nu(\x),\y),
	\]
	or equivalently,
	\[
	\f{\|\x-\y\|^2}{4t}h_t(\x,\y) =  -tLh_t(\x,\y)+\f{N}{2}h_t(\x,\y) +\f{1}{2}\sum_{\nu\in R}k(\nu)h_t(\sigma_\nu(\x),\y).
	\]
	Since $h_t(\x,\y)$ is analytic on $\mathbb C^+ =\{z: \Re z \ge 0\}$, we have for $z\in \mathbb C^+$,
	\[
	\f{\|\x-\y\|^2}{4z}h_z(\x,\y) =  -zLh_z(\x,\y)+\f{N}{2}h_z(\x,\y) +\f{1}{2}\sum_{\nu\in R}k(\nu)h_z(\sigma_\nu(\x),\y).
	\]
	It follows that 
	\[
	\f{\|\x-\y\|^2}{4|z|}|h_z(\x,\y)| \lesi  |zLh_z(\x,\y)|+|h_z(\x,\y)| + \sum_{\nu\in R}k(\nu)|h_z(\sigma_\nu(\x),\y)|.
	\]
	Arguing similarly to the proof of \cite[Lemma 4.1]{CCO}, we have
	\[
	|h_z(\sigma_\nu(\x),y)|\lesi \f{1}{\left[w(B(\x,\sqrt{\f{|z|}{\cos\theta}}))w(B(\y,\sqrt{\f{|z|}{\cos\theta}}))\right]^{1/2}}\exp\Big(-c\f{d(\x,\y)^2}{|z|}\cos\theta\Big)\f{1}{(\cos\theta)^\N}, \ \ \nu\in R,
	\]
	and
	\[
	|zLh_z(\x,\y)|\lesi \f{1}{\left[w(B(\x,\sqrt{\f{|z|}{\cos\theta}}))w(B(\y,\sqrt{\f{|z|}{\cos\theta}}))\right]^{1/2}}\exp\Big(-c\f{d(\x,\y)^2}{|z|}\cos\theta\Big)\f{1}{(\cos\theta)^{\N+1}}
	\]
	for all $\x,\y\in \RN$ and $z\in \mathbb{C}_+=\{z\in \mathbb{C}:\Re z>0\}$ where $\theta=\arg z$. 
	
	Consequently, 
	\[
	\f{\|\x-\y\|^2}{4|z|}|h_z(\x,\y)| \lesi \f{1}{\left[w(B(\x,\sqrt{\f{|z|}{\cos\theta}}))w(B(\y,\sqrt{\f{|z|}{\cos\theta}}))\right]^{1/2}}\exp\Big(-c\f{d(\x,\y)^2}{|z|}\cos\theta\Big)\f{1}{(\cos\theta)^{\N+1}},
	\]
	which implies that 
	\[
	|h_z(\x,\y)|\lesi \Big(1+ \f{\|\x-\y\|^2}{|z|}\Big)^{-1}\f{1}{\left[w(B(\x,\sqrt{\f{|z|}{\cos\theta}}))w(B(\y,\sqrt{\f{|z|}{\cos\theta}}))\right]^{1/2}}\exp\Big(-c\f{d(\x,\y)^2}{|z|}\cos\theta\Big)\f{1}{(\cos\theta)^{\N+1}}
	\]
	for all $\x,\y\in \RN$ and $z\in \mathbb{C}_+=\{z\in \mathbb{C}:\Re z>0\}$ where $\theta=\arg z$.
	
	At this stage, by the argument use in the proof of \cite[Lemma 2.1]{BDDM}, we obtain that for an even function $\varphi\in \mathscr{S}(\mathbb R)$ and every $M>0$,
	\begin{equation}\label{eq1-proof of varphi t L x y}
		|\varphi(t\sqrt L)(\x,\y)|\lesi \Big(1+\f{\|\x-\y\|}{ t}\Big)^{-2}\f{1}{w(B(\x, t+d(\x,\y)))}\Big(\f{t}{t+d(\x,\y)}\Big)^M
	\end{equation}
	for all $\x,\y\in \RN$ and $t>0$.
	
	Recall that for any $m\in \mathbb N$, 
	\begin{equation*} 
		\begin{aligned}
			(I + tL)^{-m} &= \f{1}{m!} \int_0^\vc s^{m-1}e^{-s} e^{-stL} ds.
		\end{aligned}
	\end{equation*}
	From this formula and Lemma \ref{lem-heat kernel estimate for semigroups}, by a simple calculation it can be verified that for any $M>0$ and multi-indices $\alpha,\beta$, there exists $m$ such that
	\begin{equation}\label{eq2-proof of varphi t L x y}
		|\partial^{\alpha}_{\x}\partial^{\beta}_{\y}K_{(I + tL)^{-m}}(\x,\y)|\lesi t^{-(|\alpha|+|\beta|)}\Big(1+\f{\|\x-\y\|}{ t}\Big)^{-2}\f{1}{w(B(\x, t+d(\x,\y)))}\Big(\f{t}{t+d(\x,\y)}\Big)^M,
	\end{equation}
	where $K_{(I + tL)^{-m}}(\x,\y)$ denotes the kernel of $(I + tL)^{-m}$.
	
	For any even function $\varphi \in \mathscr S(\mathbb R)$, we write
	\begin{equation*}
		\varphi(t\sqrt L) = (I + t^2L)^{-m} \widetilde \varphi(t\sqrt L),  
	\end{equation*}
	where $\widetilde \varphi(\xi) = (1 + \xi^2)^m\varphi(\xi)$ and $m$ is a sufficiently large number which will be fixed later.
	
	Then we have
	\[
	\varphi(t\sqrt L)(\x,\y)=\int_{\RN}K_{(I + t^2L)^{-m}}(\x,\z)\widetilde \varphi(t\sqrt L)(\z,\y)d w(\z).
	\]
	From this, \eqref{eq1-proof of varphi t L x y} and \eqref{eq2-proof of varphi t L x y}, for
	any $M>0$ and multi-indices $\alpha$ by taking $m$ sufficiently large, we obtain 
	\begin{equation}\label{eq3-proof of varphi t L x y}
		|\partial^{\alpha}_{\x}\varphi(t\sqrt L)(\x,\y)|\lesi t^{-|\alpha| }\Big(1+\f{\|\x-\y\|}{ t}\Big)^{-2}\f{1}{w(B(\x, t+d(\x,\y)))}\Big(\f{t}{t+d(\x,\y)}\Big)^M.
	\end{equation}
	We now write
	\[
	\varphi(t\sqrt L) = \widetilde \varphi(t\sqrt L)(I + t^2L)^{-m}
	\]
	so that 
	\[
	\varphi(t\sqrt L)(\x,\y)=\int_{\RN}\widetilde \varphi(t\sqrt L)(\x,\z)K_{(I + t^2L)^{-m}}(\z,\y)d w(\z).
	\]
	Then we apply \eqref{eq2-proof of varphi t L x y} and \eqref{eq3-proof of varphi t L x y} to obtain
	\[
	|\partial^{\alpha}_{\x}\partial^{\beta}_{\y}\varphi(t\sqrt L)(\x,\y)|\lesi t^{-(|\alpha|+|\beta|) }\Big(1+\f{\|\x-\y\|}{ t}\Big)^{-2}\f{1}{w(B(\x, t+d(\x,\y)))}\Big(\f{t}{t+d(\x,\y)}\Big)^M.
	\]
	
	This completes our proof.
\end{proof}
Since $L$ is non-negative self-adjoint and satisfies the Gaussian upper bound with respect to the distance $d$, arguing similarly to the proof of \cite[Theorem 2]{S} (see also \cite{CS}) we obtain that  the kernel $K_{\cos(t\sqrt{L})}$ of $\cos(t\sqrt{L})$ satisfies 
\begin{equation}\label{finitepropagation}
	{\rm supp}\,K_{\cos(t\sqrt{L})}\subset \{(\x,\y)\in \RN\times \RN:
	d(\x,\y)\leq t\}.
\end{equation}

We have the following useful lemma.
\begin{lem}\label{lem:finite propagation}
	Let $\varphi\in C^\vc_0(\mathbb{R})$ be an even function with {\rm supp}\,$\varphi\subset (-1, 1)$ and $\int \varphi =2\pi$. Denote by $\Phi$ the Fourier transform of $\varphi$.  Then the kernel $ t^2L\Phi(t\sqrt{L})(\x,\y)$ of $t^2L\Phi(t\sqrt{L})$ satisfies 
	\begin{equation}\label{eq0-lemPsiL}
		\displaystyle
		\int_{\RN}t^2L\Phi(t\sqrt{L})(\x,\y)dw(\y)=\int_{\RN}t^2L\Phi(t\sqrt{L})(\y,\x)dw(\y)=0
	\end{equation}
	\begin{equation}\label{eq1-lemPsiL}
		\displaystyle
		{\rm supp}\, t^2L\Phi(t\sqrt{L})(\cdot,\cdot)\subset \{(\x,\y)\in \RN\times \RN:
		d(\x,\y)\leq t\},
	\end{equation}
	\begin{equation}\label{eq2-lemPsiL}
		| t^2L\Phi(t\sqrt{L})(\x,\y)|\lesi \Big(1+\f{\|\x-\y\|}{t}\Big)^{-2}\f{1}{w(B(\y,t))}
	\end{equation}
	for all $\x,\y \in \RN$ and $t>0$; moreover, 
	\begin{equation}\label{eq3-lemPsiL}
		|t^2L\Phi(t\sqrt{L})(\x,\y)- t^2L\Phi(t\sqrt{L}) (\x',\y)|\lesi \f{\|\x-\x'\|}{ t}\Big(1+\f{\|\x-\y\|}{t}\Big)^{-2}\f{1}{w(B(\y,t))},
	\end{equation}
	for all $\x,\x',\y'\in \RN$ and $t>0$.
\end{lem}
\begin{proof}
	The proof of \eqref{eq1-lemPsiL} is similar to that of \cite[Lemma 3]{S}. The remaining properties \eqref{eq0-lemPsiL}, \eqref{eq2-lemPsiL} and \eqref{eq3-lemPsiL} follows directly from Lemma \ref{lem-heat kernel estimate} and \eqref{eq1-lemPsiL}.
	
	This completes our proof.
\end{proof}

\subsection{Spaces of test functions and Calder\'on reproducing formulae}
We first begin with concepts on test functions and distributions on the space of homogeneous type in \cite{HMY}.

For $\x,\y\in \RN$ and $r, \gamma>0$, we set
\[
P_\gamma(\x,\y;r)=\f{1}{w(B(\x,r+\|\x-\y\|))}\Big(\f{r}{r+\|\x-\y\|}\Big)^\gamma.
\]
\begin{defn}\label{defn-Galphabeta}
	Let $\x_1\in \RN, r>0, \beta\in (0,1]$ and $\gamma>0$. A function defined on $\RN$ is called a test function type $(x_1,r,\beta,\gamma)$, denoted by $f\in \mathcal G(\x_1,r,\beta,\gamma)$, if there exists a constant $C>0$ such that 
	\begin{enumerate}[{\rm (i)}]
		\item for any $\x\in \RN$, $|f(\x)|\le CP_\gamma(\x_1,\x;r)$;
		\item for any $\x,\y\in \RN$ satisfying $\|\x-\y\|\le \f{1}{2}(r+\|\x-\x_1\|)$,
		\[
		|f(\x)-f(\y)|\le C\Big[\f{\|\x-\y\|}{r+\|\x-\x_1 \|}\Big]^\beta P_\gamma(\x_1,\x;r).
		\]
	\end{enumerate}
	We also define
	\[
	\|f\|_{\mathcal G(\x_1,r,\beta,\gamma)} =\inf\{ C\in (0,\vc): \text{(i) and (ii) hold true}\}.
	\]
	We also define
	\[
	\mathring{\mathcal G}(\x_1,r,\beta,\gamma)=\Big\{f\in {\mathcal G}(\x_1,r,\beta,\gamma): \int_X f(\x)dw(\x)=0\Big\}
	\]
	under the norm $\|\cdot\|_{\mathring{\mathcal G}(\x_1,r,\beta,\gamma)}:=\|\cdot\|_{\mathcal G(\x_1,r,\beta,\gamma)}$.
\end{defn}
It is known that both ${\mathcal G}(\x_1,r,\beta,\gamma)$ and $\mathring{\mathcal G}(\x_1,r,\beta,\gamma)$ are Banach spaces for every $\beta\in (0,1]$ and $\gamma>0$. Denote ${\mathcal G}(\beta,\gamma):={\mathcal G}(0,1,\beta,\gamma)$ and $\mathring{\mathcal G}(\beta,\gamma):=\mathring{\mathcal G}(0,1,\beta,\gamma)$. Then it was proved that ${\mathcal G}(\beta,\gamma)={\mathcal G}(\x_1,r,\beta,\gamma)$ and $\mathring{\mathcal G}(\beta,\gamma)=\mathring{\mathcal G}(\x_1,r,\beta,\gamma)$ for any $\x_1\in \RN$ and $r>0$. See for example \cite{HMY}.

\bigskip

For $\epsilon \in (0,1]$ and $\beta,\gamma\in (0,\epsilon]$, we define $\mathcal G^\epsilon(\beta,\gamma)$ (resp. $\mathring{\mathcal G}^\epsilon(\beta,\gamma)$) to be the closure of $\mathcal G(\epsilon,\epsilon)$ (resp. $\mathring{\mathcal G}(\epsilon,\epsilon)$) in the space $\mathcal G(\beta,\gamma)$ (resp. $\mathring{\mathcal G}(\beta,\gamma)$) under the norm $\|\cdot\|_{\mathcal G^\epsilon(\beta,\gamma)} =\|\cdot\|_{\mathcal G(\beta,\gamma)}$ (resp. $\|\cdot\|_{\mathring{\mathcal G}^\epsilon_0(\beta,\gamma)} =\|\cdot\|_{\mathring{\mathcal G}(\beta,\gamma)}$). Denote by $(\mathcal G^\epsilon(\beta,\gamma))'$ (resp. $(\mathring{\mathcal G}^\epsilon(\beta,\gamma))'$) the dual space of $(\mathcal G^\epsilon(\beta,\gamma))$ (resp. $(\mathring{\mathcal G}^\epsilon(\beta,\gamma))$) under the weak-$*$ topology. The spaces $(\mathcal G^\epsilon(\beta,\gamma))$ and $(\mathring{\mathcal G}^\epsilon(\beta,\gamma))$ are called the spaces of test functions on $\RN$, and $(\mathcal G^\epsilon(\beta,\gamma))'$  and $(\mathring{\mathcal G}^\epsilon(\beta,\gamma))'$ are called the spaces of distributions on $\RN$.

\begin{prop}\label{prop-prop1 kernel is a test function}
	Let $\varphi \in \mathscr S(\mathbb R)$ be an even function satisfying $\varphi(0)=0$. Then, for every $\y\in \RN$ and $t>0$, $\varphi(t\sqrt L)(\cdot, \y), \varphi(t\sqrt L)(\y,\cdot)\in \mathring{\mathcal G}(1,1)$.
\end{prop}
\begin{proof}
	From Lemma \ref{lem-heat kernel estimate},
	$$
	\int \varphi(t\sqrt L)(\x,\y)dw(\x)=\int \varphi(t\sqrt L)(\y,\x)dw(\x)=0
	$$
	for all $\y\in \RN$ and $t>0$.
	
	Hence, it suffices to prove that for a fixed $\x \in \RN$ and $t>0$,  
	\[
	\varphi(t\sqrt L)(\cdot, \y)\in \mathcal G(\y, r,1,1), \ \ \ \ r=4\|\y\|.
	\]
	We first show that 
	\begin{equation}\label{eq-cond i for varphi t}
		|\varphi(t\sqrt L)(\x,\y)|\lesi P_1(\y,\x; r):=\f{1}{w(B(\x, r+\|\x-\y\|))} \f{r}{r+\|\x-\y\|}, \ \ \ \x\in \RN.
	\end{equation}
	If $r\ge \|\x-\y\|$, then from Lemma \ref{lem-heat kernel estimate},
	\[
	|\varphi(t\sqrt L)(\y,\x)|\lesi \f{1}{w(B(\x,t))} \simeq_{t,\|\y\|} \f{1}{w(B(\x,4\|\y\|))}= \f{1}{w(B(\x,r))}\simeq P_1(\y,\x; r).
	\]
	If $\|\x-\y\|>r:=4\|\y\|$, then $\|\x\|\ge 3\|\y\|$. In this case $d(\x,\y)\simeq \|\x-\y\|\simeq \|\x\|$. Hence,
	\[
	\begin{aligned}
		|\varphi(t\sqrt L)(\x,\y)|&\lesi \f{1}{w(B(\x,t))}\Big(\f{t}{t+\|\x-\y\|}\Big)^{1+\mathfrak{N}}\\
		&\lesi \f{1}{w(B(\x,\|\x-\y\|))} \f{t}{t+\|\x-\y\|} \\
		&\lesi_{t,\|\y\|} \f{1}{w(B(\x,\|\x-\y\|))} \f{\|\y\|}{\|\x-\y\|} \simeq P_1(\y,\x; r).		
	\end{aligned}
	\]
	We next show that for $\x,\x'\in \RN$ with $\|\x-\x'\|\le \f{1}{2}(r+\|\x-\y\|)$,
	\[
	|\varphi(t\sqrt L)(\x,\y)-\varphi(t\sqrt L)(\x',\y)|\lesi \f{\|\x-\x'\|}{r+\|\x-\y\|} P_1(\y,\x; r).
	\]
	If $\|\x-\x'\|<t$, then 
	\[
	|\varphi(t\sqrt L)(\x,\y)-\varphi(t\sqrt L)(\x',\y)|\lesi \f{\|\x-\x'\|}{t}\f{1}{w(B(\x,t))}\Big(1+\f{d(\x,\y)}{t}\Big)^{-2-\mathfrak{N}}.
	\]
	Similarly to the proof of \eqref{eq-cond i for varphi t}, if $r\ge \|\x-\y\|$, we have
	\[
	\begin{aligned}
		|\varphi(t\sqrt L)(\x,\y)-\varphi(t\sqrt L)(\x',\y)|&\lesi \f{\|\x-\x'\|}{t} \f{1}{w(B(\x,t))}\\
		&\lesi_{t,\|\y\|} \f{\|\x-\x'\|}{4\|\y\|}\f{1}{w(B(\x,4\|\y\|))}\\
		&\simeq \f{\|\x-\x'\|}{r+\|\x-\y\|} P_1(\y,\x; r). 
	\end{aligned}
	\]
	If $\|\x-\y\|>r:=4\|\y\|$, we have
	\[
	\begin{aligned}
		|\varphi(t\sqrt L)(\x,\y)-\varphi(t\sqrt L)(\x',\y)|&\lesi \f{\|\x-\x'\|}{t}\f{1}{w(B(\x,t))}\Big(\f{t}{t+\|\x-\y\|}\Big)^{2+\mathfrak{N}}\\
		&\lesi \f{\|\x-\x'\|}{ \|\x-\y\|}\f{1}{w(B(\x,\|\x-\y\|))} \f{t}{t+\|\x-\y\|} \\
		&\lesi_{t,\|\y\|} \f{\|\x-\x'\|}{r+ \|\x-\y\|}\f{1}{w(B(\x,\|\x-\y\|))} \f{\|\y\|}{\|\x-\y\|}\\
		&\lesi  \f{\|\x-\x'\|}{r+ \|\x-\y\|}P_1(\y,\x; r).		
	\end{aligned}
	\]
	If $\|\x-\x'\|\ge t$, then by a careful examination of the proof of  \eqref{eq-cond i for varphi t}, we can show that 
	\[
	|\varphi(t\sqrt L)(\x,\y)|\lesi \f{t}{r+\|\x-\y\|}P_1(\y,\x; r) , \ \ \ \x\in \RN.
	\]
	Therefore,
	\[
	\begin{aligned}
		|\varphi(t\sqrt L)(\x,\y)-\varphi(t\sqrt L)(\x',\y)|&\lesi \f{t}{r+\|\x-\y\|}P_1(\y,\x; r) +\f{t}{r+\|\x-\y\|}P_1(\y,\x'; r) \\
		&\lesi \f{\|\x-\x'\|}{r+\|\x-\y\|}P_1(\y,\x; r) +\f{\|\x-\x'\|}{r+\|\x-\y\|}P_1(\y,\x'; r).
	\end{aligned}
	\]
	On the other hand, $r+\|\x-\y\|\simeq r+\|\x'-\y\|$, whenever $\|\x-\x'\|\le \f{1}{2}(r+\|\x-\y\|)$. Consequently,
	\[
	|\varphi(t\sqrt L)(\x,\y)-\varphi(t\sqrt L)(\x',\y)|
	\lesi \f{\|\x-\x'\|}{r+\|\x-\y\|}P_1(\y,\x; r).
	\]
	This completes our proof.
\end{proof}

Due to Proposition \ref{prop-prop1 kernel is a test function}, for each even function $\varphi\in \mathscr{\mathbb R}$, $\x\in \RN$ and $t>0$, we can define
\[
\varphi(t\sqrt L)f(\x)=\langle \varphi(t\sqrt L)(\x,\cdot), f\rangle
\]
for any $f\in (\mathring{\mathcal G}^1(\beta,\gamma))'$ with $0<\beta, \gamma<1$.

In what follows, we will use the following elementary inequality (see for example \cite[Lemma 2.2]{BBD}) frequently without any explanation
\[
\int_{\RN} \f{1}{w(B(\x,s))+w(B(\y,s))}\Big(1+\f{d(\x,\y)}{s}\Big)^{-\N-\epsilon}dw(\y)\le C(\epsilon).
\]
\begin{lem}\label{lem - t small}
	
	Let $\varphi \in \mathscr S(\mathbb R)$ be an even function satisfying $\varphi(0)=0$. Let $\epsilon\in (0,1]$ and $0<\beta,\gamma<\epsilon$. For $t\in (0,1)$ and $f\in \mathring{\mathcal G}(\epsilon,\epsilon)$, then we have
	\begin{equation}\label{eq- inequality small t}
		\|\varphi(t\sqrt L)f\|_{\mathring{\mathcal G}^\epsilon(\beta,\gamma)}\lesi t^{\epsilon-\beta}\|f\|_{\mathring{\mathcal G}(\epsilon,\epsilon)}.
	\end{equation}
\end{lem}
\begin{proof}
	We fix a constant $M>\N$.  Without loss of generality, assume that $\|f\|_{\mathring{\mathcal G}(\epsilon,\epsilon)}=1$.  
	We have
	\[
	\begin{aligned}
		\varphi(t\sqrt L)f(\x) &=  \int_{\RN} \varphi(t\sqrt L)(\x,\y)f(\y)dw(\y).
	\end{aligned}
	\]
	By using Lemma \ref{lem-heat kernel estimate}, it follows that 
	\[
	\int_{\RN}\varphi(t\sqrt L)f(\x) dw(\x) = 0.
	\]
	
	For $\x\in \RN$, set $S(\x)=\{\y: \|\x-\y\|\le \f{1}{2}(1+\|\x\|)\}$. We have
	\[
	\begin{aligned}
		|\varphi(t\sqrt L)f(\x)| &= \Big|\int_{\RN} \varphi(t\sqrt L)(\x,\y)f(\y)dw(\y)\Big|\\
		&= \Big|\int_{\RN} \varphi(t\sqrt L)(\x,\y)[f(\y)-f(\x)]dw(\y)\Big|\\
		&\le \Big|\int_{S(\x)} \varphi(t\sqrt L)(\x,\y)[f(\y)-f(\x)]dw(\y)\Big| +\Big|\int_{\RN\backslash S(\x)} \varphi(t\sqrt L)(\x,\y)[f(\y)-f(\x)]dw(\y)\Big|\\
		&=: E_1 + E_2.
	\end{aligned}
	\]
	For the term $E_1$, by (ii) in Definition \ref{defn-Galphabeta},
	\[
	|f(\y)-f(\x)|\le \Big[\f{\|\x-\y\|}{1+\|\x\|}\Big]^\epsilon P_\epsilon(0,\x;1).
	\]
	This, along with Lemma \ref{lem-heat kernel estimate}, implies, for $M>\mathfrak{N}$,
	\[
	\begin{aligned}
		E_1&\lesi P_\epsilon(0,\x;1)\int_{S(\x)}\Big[\f{\|\x-\y\|}{1+\|\x\|}\Big]^\epsilon \Big(1+\f{\|\x-\y\|}{t}\Big)^{-2}\f{1}{w(B(\x, t))}\Big(\f{t}{t+d(\x,\y)}\Big)^Mdw(\y)\\
		&\lesi P_\epsilon(0,\x;1)\int_{S(\x)}\Big[\f{\|\x-\y\|}{1+\|\x\|}\Big]^\epsilon \Big[\f{ t}{t +\|\x-\y\|}\Big]^\epsilon\f{1}{w(B(\x,  t))}\Big(\f{t}{t+d(\x,\y)}\Big)^Mdw(\y)\\
		&\lesi  \Big[\f{t}{1+\|\x\|}\Big]^\epsilon P_\epsilon(0,\x;1)\int_{\RN} \f{1}{w(B(\x,  t))}\Big(\f{t}{t+d(\x,\y)}\Big)^Mdw(\y)\\
		&\lesi  \Big[\f{t}{1+\|\x\|}\Big]^\epsilon P_\gamma(0,\x;1).
	\end{aligned}
	\]
	For $E_2$, we write
	\[
	\begin{aligned}
		E_2&\le \int_{\RN\backslash S(\x)} |\varphi(t\sqrt L)(\x,\y)||f(\y)|dw(\y)+\int_{\RN\backslash S(\x)} |\varphi(t\sqrt L)(\x,\y)||f(\x)|dw(\y)\\
		&=: E_{21} + E_{22}.
	\end{aligned}
	\]
	By (i) in Definition \ref{defn-Galphabeta} and Lemma \ref{lem-heat kernel estimate},
	\[
	\begin{aligned}
		E_{22}&\lesi P_\epsilon(0,\x;1)\int_{_{\RN\backslash S(\x)}} \Big(1+\f{\|\x-\y\|}{  t}\Big)^{-2}\f{1}{w(B(\x,  t))}\Big(\f{t}{t+d(\x,\y)}\Big)^Mdw(\y)\\
		&\lesi \Big[\f{t}{1+\|\x\|}\Big]^2 P_\gamma(0,\x;1)\int_{\RN}  \f{1}{w(B(\x,  t))}\Big(\f{t}{t+d(\x,\y)}\Big)^Mdw(\y)\\
		&\lesi \Big[\f{t}{1+\|\x\|}\Big]^\epsilon P_\gamma(0,\x;1).
	\end{aligned}
	\]
	Applying (i) in Definition \ref{defn-Galphabeta} and Lemma \ref{lem-heat kernel estimate}, we have
	\[
	\begin{aligned}
		E_{21}&\lesi \int_{\RN\backslash S(\x)}\Big(1+\f{\|\x-\y\|}{  t}\Big)^{-2}\f{1}{w(B(\x,  t))}\Big(\f{t}{t+d(\x,\y)}\Big)^{2M+\gamma}P_\epsilon(0,\y;1) dw(\y)\\
		&\lesi \Big[\f{t}{1+\|\x\|}\Big]^2\int_{\RN\backslash S(\x)} \f{1}{w(B(\x,  t))}\Big(\f{t}{t+d(\x,\y)}\Big)^{2M+\gamma}P_{\gamma}(0,\y;1) dw(\y).
	\end{aligned}
	\]
	This, along with the fact that 
	\[
	\Big(\f{1}{1+\|\y\|}\Big)\Big(\f{t}{t+d(\x,\y)}\Big)\le \Big(\f{1}{1+\|\y\|}\Big)\Big(\f{1}{1+d(\x,\y)}\Big)\lesi \f{1}{1+\|\x\|} , \ \ \ \x,\y\in \RN, t\in (0,1),
	\]
	implies that 
	\begin{equation}\label{eq-estimate E22}
		\begin{aligned}
			E_{21}&\lesi \Big[\f{t}{1+\|\x\|}\Big]^2\Big(\f{1}{1+\|\x\|}\Big)^\gamma \int_{\RN\backslash S(\x)} \f{1}{w(B(\x,  t))}\Big(\f{t}{t+d(\x,\y)}\Big)^{2M}\f{1}{w(B(0,1+\|\y\|))} dw(\y).
		\end{aligned}
	\end{equation}
	If $\|\y\|\ge \|\x\|/2$, then 
	\[
	\f{1}{w(B(0,1+\|\y\|))}\lesi \f{1}{w(B(0,1+\|\x\|))}.
	\]
	If $\|\y\|\le \|\x\|/2$, then $d(\x,\y)\simeq \|\x\|$. This, together with \eqref{eq-doubling}, yields
	\[
	\begin{aligned}
		\Big(\f{t}{t+d(\x,\y)}\Big)^{M}\f{1}{w(B(0,1+\|\y\|))}&\simeq \Big(\f{t}{t+\|\x\|}\Big)^{M}\f{1}{w(B(0,1+\|\y\|))}\\
		&\lesi \Big(\f{t}{t+\|\x\|}\Big)^{M}\Big(\f{1+\|\x\|}{1+\|\y\|}\Big)^{\mathfrak{N}}\f{1}{w(B(0,1+\|\x\|))}\\
		&\lesi \f{1}{w(B(0,1+\|\x\|))}.		
	\end{aligned}
	\]
	In both cases, we have
	\[
	\Big(\f{t}{t+\|\x\|}\Big)^{M}\f{1}{w(B(0,1+\|\y\|))}\lesi \f{1}{w(B(0,1+\|\x\|))}.
	\]
	Plugging this into \eqref{eq-estimate E22},
	\[
	\begin{aligned}
		E_{22}&\lesi \Big[\f{t}{1+\|\x\|}\Big]^2\f{1}{w(B(0,1+\|\x\|))}\Big(\f{1}{1+\|\x\|}\Big)^\gamma \int_{\RN} \f{1}{w(B(\x,  t))}\Big(\f{t}{t+\|\x\|}\Big)^{M}dw(\y)\\
		&\lesi \Big[\f{t}{1+\|\x\|}\Big]^\epsilon P_\gamma(0,\x;1).
	\end{aligned}
	\]
	Taking all the estimates of $E_1, E_{21}$ and $E_{22}$ into account we obtain 
	\begin{equation}\label{eq- estimate of tLetL f}
		|\varphi(t\sqrt L)f(\x)|\lesi \Big[\f{t}{1+\|\x\|}\Big]^\epsilon P_\gamma(0,\x;1),
	\end{equation}
	which implies
	\begin{equation}\label{eq- condition i estimate of tLetL f}
		|\varphi(t\sqrt L)f(\x)|\lesi t^\epsilon P_\gamma(0,\x;1).
	\end{equation}
	\bigskip

	We next prove that 
	For $\x,\x'\in \RN$ with $\|\x-\x'\|\le \f{1}{2}(1+\|\x\|)$ and $t\in (0,1)$, we have
	\begin{equation}\label{eq-condition ii of tLetL}
		|\varphi(t\sqrt L)f(\x)-\varphi(t\sqrt L)f(\x')| \lesi t^{\epsilon-\beta} \Big(\f{\|\x-\x'\|}{1+\|\x\|}\Big)^\beta P_\gamma(0,\x;1).
	\end{equation}
	
	Indeed, we first observe that  
	\[
	\begin{aligned}
		\varphi(t\sqrt L)f(\x)-\varphi(t\sqrt L)f(\x') &= \int_{\RN} [\varphi(t\sqrt L)(\x,\y)-\varphi(t\sqrt L)(\x',\y)]f(\y)dw(\y).
	\end{aligned}
	\]
	
	If $\|\x-\x'\|\le t$, then we write
	\[
	\begin{aligned}
		|\varphi(t\sqrt L)f(\x)&-\varphi(t\sqrt L)f(\x')|\\ &= \Big|\int_{\RN} [\varphi(t\sqrt L)(\x,\y)-\varphi(t\sqrt L)(\x',\y)][f(\y)-f(\x)]dw(\y)\Big|\\
		&\lesi \int_{\RN} \f{\|\x-\x'\|}{t}\Big(1+\f{\|\x-\y\|}{  t}\Big)^{-2}\f{1}{w(B(\x,  t))}\Big(\f{t}{t+d(\x,\y)}\Big)^{M}|f(\x)-f(\y)|dw(\y)\\
		&\lesi\int_{\{\y: \|\x-\y\|\le \f{1}{2}(1+\|\x\|)\}}\ldots +\int_{\{\y: \|\x-\y\|> \f{1}{2}(1+\|\x\|)\}}\ldots\\
		&=:F_1 + F_2. 
	\end{aligned}
	\]
	For $F_1$, using (ii) in Definition \ref{defn-Galphabeta},
	\[
	\begin{aligned}
		F_1 & \lesi P_\gamma(0,\x;1)\int_{\RN} \Big(\f{\|\x-\x'\|}{t}\Big)^\epsilon\Big( \f{  t}{t+\|\x-\y\|}\Big)^\epsilon\f{1}{w(B(\x,  t))}\Big(\f{t}{t+d(\x,\y)}\Big)^{M}\Big[\f{\|\x-\y\|}{1+\|\x\|}\Big]^\epsilon dw(\y)\\
		& \lesi P_\gamma(0,\x;1)\int_{\RN}  \Big(\f{\|\x-\x'\|}{1+\|\x\|}\Big)^\epsilon\f{1}{w(B(\x,  t))}\Big(\f{t}{t+\|\x\|}\Big)^{M} dw(\y).
	\end{aligned}
	\]
	Since 
	\[
	\begin{aligned}
		\Big(\f{\|\x-\x'\|}{1+\|\x\|}\Big)^\epsilon&= \Big(\f{\|\x-\x'\|}{1+\|\x\|}\Big)^\beta \Big(\f{\|\x-\x'\|}{1+\|\x\|}\Big)^{\epsilon-\beta}\\
		&\le \Big(\f{\|\x-\x'\|}{1+\|\x\|}\Big)^\beta \|\x-\x'\| ^{\epsilon-\beta}\\
		&\le t^{\epsilon-\beta}\Big(\f{\|\x-\x'\|}{1+\|\x\|}\Big)^\beta,
	\end{aligned}	
	\]
	as long as $\|\x-\x'\|\le t$, we have
	\[
	F_1\lesi t^{\epsilon-\beta}\Big(\f{\|\x-\x'\|}{1+\|\x\|}\Big)^\beta P_\gamma(0,\x;1).
	\]
	
	For the term $F_2$, similarly we have
	\[
	\begin{aligned}
		F_2&\lesi \int_{\{\y: \|\x-\y\|> \f{1}{2}(1+\|\x\|)\}}\f{\|\x-\x'\|}{t} \f{t}{\|\x-\y\|}\f{1}{w(B(\x,  t))}\Big(\f{t}{t+d(\x,\y)}\Big)^{M}|f(\x)-f(\y)|dw(\y)\\
		&\lesi \f{\|\x-\x'\|}{1+\|\x\|}  \int_{\{\y: \|\x-\y\|> \f{1}{2}(1+\|\x\|)\}}\f{1}{w(B(\x,  t))}\Big(\f{t}{t+d(\x,\y)}\Big)^{M}|f(\x)-f(\y)|dw(\y).
	\end{aligned}
	\]
	Arguing similarly to the estimate of $E_{2}$,
	\[
	\int_{\{\y: \|\x-\y\|> \f{1}{2}(1+\|\x\|)\}}\f{1}{w(B(\x,  t))}\Big(\f{t}{t+d(\x,\y)}\Big)^{M}|f(\x)-f(\y)|dw(\y)\lesi P_\gamma(0,\x;1).
	\]
	Therefore,
	\[
	\begin{aligned}
		F_2&\lesi  \f{\|\x-\x'\|}{1+\|\x\|} P_\gamma(0,\x;1)\\
		&\lesi \|\x-\x'\|^{1-\beta}\Big(\f{\|\x-\x'\|}{1+\|\x\|}\Big)^{\beta} P_\gamma(0,\x;1)\\
		&\lesi t^{1-\beta}\Big(\f{\|\x-\x'\|}{1+\|\x\|}\Big)^{\beta} P_\gamma(0,\x;1) \\
		&\lesi t^{\epsilon-\beta}\Big(\f{\|\x-\x'\|}{1+\|\x\|}\Big)^{\beta} P_\gamma(0,\x;1).
	\end{aligned}
	\]
	as long as $\|\x-\x'\|\le t$.

	\textbf{Case 2:} $t<\|\x-\x'\|<\f{1}{2}(1+\|\x\|)$. Using \eqref{eq- estimate of tLetL f},
	\[
	\begin{aligned}
		|\varphi(t\sqrt L)f(\x)-\varphi(t\sqrt L)f(\x')|&\le |\varphi(t\sqrt L)f(\x)|+|\varphi(t\sqrt L)f(\x')|\\
		&\lesi \Big[\f{t}{1+\|\x\|}\Big]^\epsilon P_\gamma(0,\x;1)  + \Big[\f{t}{1+\|\x'\|}\Big]^\epsilon P_\gamma(0,\x',1).
	\end{aligned}
	\]
	This, along with the fact that $1+\|\x\|\simeq 1+ \|\x'\|$, whenever $\|\x-\x'\|\le \f{1}{2}(1+\|\x\|)$, implies
	\[
	\begin{aligned}
		|\varphi(t\sqrt L)f(\x)-\varphi(t\sqrt L)f(\x')|&\le |\varphi(t\sqrt L)f(\x)|+|\varphi(t\sqrt L)f(\x')|\\
		&\lesi \Big[\f{t}{1+\|\x\|}\Big]^\epsilon P_\gamma(0,\x;1)\\
		&\lesi t^{\epsilon-\beta} \Big[\f{t}{1+\|\x\|}\Big]^\beta P_\gamma(0,\x;1)\\
		&\lesi t^{\epsilon-\beta} \Big[\f{\|\x-\x'\|}{1+\|\x\|}\Big]^\beta P_\gamma(0,\x;1),
	\end{aligned}
	\]
	as long as $\|\x-\x'\|>t$.
	
	This ensures \eqref{eq-condition ii of tLetL}. Taking the estimates \eqref{eq- condition i estimate of tLetL f} and \eqref{eq-condition ii of tLetL} into account, we imply \eqref{eq- inequality small t}. 
	
	This completes our proof.
\end{proof}

\begin{lem}\label{lem - integral t small}

	Let $\varphi \in \mathscr S(\mathbb R)$ be an even function satisfying $\varphi(0)=0$. Let  $0<\beta,\gamma<1$. For $f\in \mathring{\mathcal G}(\beta,\gamma)$, then we have
	\begin{equation}\label{eq- inequality integral small t}
		\Big\|\int_0^2\varphi(t\sqrt L)f\f{dt}{t}\Big\|_{\mathring{\mathcal G}(\beta,\gamma)}\lesi \|f\|_{\mathring{\mathcal G}(\beta,\gamma)}.
	\end{equation}
\end{lem}
\begin{proof}
	Without loss of generality, assume that $\|f\|_{\mathring{\mathcal G}(\beta,\gamma)}=1$.  
	
	Arguing similarly to the proof of \eqref{eq- estimate of tLetL f}, we also obtain 
	\begin{equation}\label{eq1- proof of integral t small}
		|\varphi(t\sqrt L)f(\x)|\lesi \Big[\f{t}{1+\|\x\|}\Big]^\beta P_\gamma(0,\x;1),
	\end{equation}
	which implies
	\begin{equation*}
		|\varphi(t\sqrt L)f(\x)|\lesi t^\beta P_\gamma(0,\x;1).
	\end{equation*}
	It follows that
	\begin{equation}\label{eq1- proof of integral t small s}
		\Big|\int_0^2\varphi(t\sqrt L)f(\x)\f{dt}{t}\Big|\lesi P_\gamma(0,\x;1).
	\end{equation}
	\bigskip

	It suffices to prove that for $\x,\x'\in \RN$ with $\|\x-\x'\|\le \f{1}{2}(1+\|\x\|)$ and $t\in (0,1)$, we have
	\begin{equation}\label{eq-condition ii of tLetL proof for integral}
		|\varphi(t\sqrt L)f(\x)-\varphi(t\sqrt L)f(\x')| \lesi \min\Big\{\Big(\f{\|\x-\x'\|}{t}\Big)^{1-\beta}, \Big(\f{t}{\|\x-\x'\|}\Big)^\beta\Big\} \Big(\f{\|\x-\x'\|}{1+\|\x\|}\Big)^\beta P_\gamma(0,\x;1).
	\end{equation}
	Indeed, once this inequality has been proved, it follows that  
	for $\x,\x'\in \RN$ with $\|\x-\x'\|\le \f{1}{2}(1+\|\x\|)$, we have
	\begin{equation*}%\label{eq-condition ii of tLetL proof for integral}
		\Big|\int_0^2\varphi(t\sqrt L)f(\x)\f{dt}{t}-\int_0^2\varphi(t\sqrt L)f(\x')\f{dt}{t}\Big| \lesi \Big(\f{\|\x-\x'\|}{1+\|\x\|}\Big)^\beta P_\gamma(0,\x;1).
	\end{equation*}
	which, along with \eqref{eq1- proof of integral t small s}, implies \eqref{eq- inequality integral small t}.
	
	We now take care of \eqref{eq-condition ii of tLetL proof for integral}. The proof of this is similar to the proof of Lemma \ref{lem - t small}. We first observe that  
	\[
	\begin{aligned}
		\varphi(t\sqrt L)f(\x)-\varphi(t\sqrt L)f(\x') &= \int_{\RN} [\varphi(t\sqrt L)(\x,\y)-\varphi(t\sqrt L)(\x',\y)]f(\y)dw(\y).
	\end{aligned}
	\]
	
	If $\|\x-\x'\|\le t$, then we write
	\[
	\begin{aligned}
		|\varphi(t\sqrt L)f(\x)&-\varphi(t\sqrt L)f(\x')|\\ &= \Big|\int_{\RN} [\varphi(t\sqrt L)(\x,\y)-\varphi(t\sqrt L)(\x',\y)][f(\y)-f(\x)]dw(\y)\Big|\\
		&\lesi \int_{\RN} \f{\|\x-\x'\|}{t}\Big(1+\f{\|\x-\y\|}{  t}\Big)^{-2}\f{1}{w(B(\x,  t))}\Big(\f{t}{t+d(\x,\y)}\Big)^{M}|f(\x)-f(\y)|dw(\y)\\
		&\lesi\int_{\{\y: \|\x-\y\|\le \f{1}{2}(1+\|\x\|)\}}\ldots +\int_{\{\y: \|\x-\y\|> \f{1}{2}(1+\|\x\|)\}}\ldots\\
		&=:E_1 + E_2. 
	\end{aligned}
	\]
	For $E_1$, using (ii) in Definition \ref{defn-Galphabeta},
	\[
	\begin{aligned}
		E_1 & \lesi P_\gamma(0,\x;1)\int_{\RN}  \f{\|\x-\x'\|}{t} \Big( \f{  t}{t+\|\x-\y\|}\Big)^\beta\f{1}{w(B(\x,  t))}\Big(\f{t}{t+d(\x,\y)}\Big)^{M}\Big[\f{\|\x-\y\|}{1+\|\x\|}\Big]^\beta dw(\y)\\
		& \lesi \Big(\f{\|\x-\x'\|}{t}\Big)^{1-\beta}P_\gamma(0,\x;1)\int_{\RN}  \Big(\f{\|\x-\x'\|}{t}\Big)^\beta \Big( \f{  t}{1+\|\x\|}\Big)^\beta\f{1}{w(B(\x,  t))}\Big(\f{t}{t+d(\x,\y)}\Big)^{M}  dw(\y)\\
		& \lesi \Big(\f{\|\x-\x'\|}{t}\Big)^{1-\beta}P_\gamma(0,\x;1)\int_{\RN}  \Big(\f{\|\x-\x'\|}{1+\|\x\|}\Big)^\beta\f{1}{w(B(\x,  t))}\Big(\f{t}{t+\|\x\|}\Big)^{M} dw(\y)\\
		& \lesi \Big(\f{\|\x-\x'\|}{t}\Big)^{1-\beta}\Big(\f{\|\x-\x'\|}{1+\|\x\|}\Big)^\beta P_\gamma(0,\x;1).
	\end{aligned}
	\]
	
	For the term $E_2$, similarly we have
	\[
	\begin{aligned}
		E_2&\lesi \int_{\{\y: \|\x-\y\|> \f{1}{2}(1+\|\x\|)\}}\f{\|\x-\x'\|}{t} \Big(1+\f{\|\x-\y\|}{t}\Big)^{-2} \f{1}{w(B(\x,  t))}\Big(\f{t}{t+d(\x,\y)}\Big)^{M}|f(\x)-f(\y)|dw(\y)\\
		&\lesi \Big(\f{\|\x-\x'\|}{t}\Big)^{1-\beta}\int_{\{\y: \|\x-\y\|> \f{1}{2}(1+\|\x\|)\}}\Big(\f{\|\x-\x'\|}{t}\Big)^\beta \Big(\f{t}{\|\x-\y\|}\Big)^\beta\f{1}{w(B(\x,  t))}\Big(\f{t}{t+d(\x,\y)}\Big)^{M}|f(\x)-f(\y)|dw(\y)\\
		&\lesi \Big(\f{\|\x-\x'\|}{t}\Big)^{1-\beta}\int_{\{\y: \|\x-\y\|> \f{1}{2}(1+\|\x\|)\}}\Big(\f{\|\x-\x'\|}{\|\x-\y\|}\Big)^\beta \f{1}{w(B(\x,  t))}\Big(\f{t}{t+d(\x,\y)}\Big)^{M}|f(\x)-f(\y)|dw(\y)\\
		&\lesi \Big(\f{\|\x-\x'\|}{t}\Big)^{1-\beta} \Big(\f{\|\x-\x'\|}{\|\x-\y\|}\Big)^\beta  \int_{\{\y: \|\x-\y\|> \f{1}{2}(1+\|\x\|)\}}\f{1}{w(B(\x,  t))}\Big(\f{t}{t+d(\x,\y)}\Big)^{M}|f(\x)-f(\y)|dw(\y).
	\end{aligned}
	\]
	Arguing similarly to the estimate of the term $E_{2}$ in the proof of Lemma \ref{lem - t small}, we also obtain
	\[
	\int_{\{\y: \|\x-\y\|> \f{1}{2}(1+\|\x\|)\}}\f{1}{w(B(\x,  t))}\Big(\f{t}{t+d(\x,\y)}\Big)^{M}|f(\x)-f(\y)|dw(\y)\lesi P_\gamma(0,\x;1).
	\]
	Therefore,
	\[
	\begin{aligned}
		E_2&\lesi    \Big(\f{\|\x-\x'\|}{t}\Big)^{1-\beta} \Big(\f{\|\x-\x'\|}{\|\x-\y\|}\Big)^\beta P_\gamma(0,\x;1),
	\end{aligned}
	\]
	as long as $\|\x-\x'\|\le t$.
	
	\bigskip	
	
	\textbf{Case 2:} $t<\|\x-\x'\|<\f{1}{2}(1+\|\x\|)$. Using \eqref{eq1- proof of integral t small},
	\[
	\begin{aligned}
		|\varphi(t\sqrt L)f(\x)-\varphi(t\sqrt L)f(\x')|&\le |\varphi(t\sqrt L)f(\x)|+|\varphi(t\sqrt L)f(\x')|\\
		&\lesi \Big[\f{t}{1+\|\x\|}\Big]^\beta P_\gamma(0,\x;1)  + \Big[\f{t}{1+\|\x'\|}\Big]^\beta P_\gamma(0,\x',1).
	\end{aligned}
	\]
	This, along with the fact that $1+\|\x\|\simeq 1+ \|\x'\|$, whenever $\|\x-\x'\|\le \f{1}{2}(1+\|\x\|)$, implies
	\[
	\begin{aligned}
		|\varphi(t\sqrt L)f(\x)-\varphi(t\sqrt L)f(\x')|&\le |\varphi(t\sqrt L)f(\x)|+|\varphi(t\sqrt L)f(\x')|\\
		&\lesi \Big[\f{t}{1+\|\x\|}\Big]^\beta P_\gamma(0,\x;1)\\
		&\lesi  \Big[\f{t}{\|\x-\x'\|}\Big]^\beta \Big[\f{\|\x-\x'\|}{1+\|\x\|}\Big]^\beta P_\gamma(0,\x;1),
	\end{aligned}
	\]
	as long as $t<\|\x-\x'\|\le \f{1}{2}(1+\|\x\|)$, which ensures \eqref{eq-condition ii of tLetL proof for integral}.
	
	This completes our proof.
\end{proof}

As a counter part of Lemma \ref{lem - t small} and Lemma \ref{lem - integral t small} we have the following two lemmas corresponding to large value of $t$.
\begin{lem}
	\label{lem - t large}  	Let $\varphi \in \mathscr S(\mathbb R)$ be an even function satisfying $\varphi(0)=0$. Let $\epsilon\in (0,1]$ and $0<\beta,\gamma<\epsilon$. For $t\ge 1$ and $f\in \mathring{\mathcal G}(\epsilon,\epsilon)$, then we have 
	\begin{equation}\label{eq- inequality t large}
		\|\varphi(t\sqrt L)f\|_{\mathring{\mathcal G}^\epsilon(\beta,\gamma)}\lesi t^{-(\epsilon-\gamma)}\|f\|_{\mathring{\mathcal G}(\epsilon,\epsilon)}.
	\end{equation}
\end{lem}
\begin{proof}
	In this proof, $M>\mathfrak{N}$ will be a fixed number.  Without loss of generality, assume that $\|f\|_{\mathring{\mathcal G}(\beta,\gamma)}=1$. We have
	\[
	\begin{aligned}
		\varphi(t\sqrt L)f(\x) &=  \int_{\RN} \varphi(t\sqrt L)(\x,\y)f(\y)dw(\y).
	\end{aligned}
	\]
	By using Lemma \ref{lem-heat kernel estimate}, it follows that 
	\[
	\int_{\RN}\varphi(t\sqrt L)f(\x) dw(\x) = 0.
	\]
	We now consider two cases: $1+\|\x\|\ge t$ and $1+\|\x\|< t$.
	
	\bigskip
	
	\textbf{Case 1: $1+\|\x\|\ge t$}
	
	In this case, we write
	\[
	\begin{aligned}
		|\varphi(t\sqrt L)f(\x)|&\le \Big|\int_{\|\y\|\ge \|\x\|/2} \varphi(t\sqrt L)(\x,\y)f(\y)dw(\y)\Big| + \Big|\int_{\|\y\|< \|\x\|/2} \varphi(t\sqrt L)(\x,\y)f(\y)dw(\y)\Big|\\
		&=:E_1 + E_2.
	\end{aligned}
	\]
	We write
	\[
	\begin{aligned}
		E_1 &= \Big|\int_{\|\y\|\ge \|\x\|/2} \varphi(t\sqrt L)(\x,\y)f(\y)dw(\y)\Big|\\
		&\le \Big|\int_{\|\y\|\ge \|\x\|/2\atop \|\x-\y\|\ge \f{1}{2}(1+\|\x\|)} \varphi(t\sqrt L)(\x,\y)f(\y)dw(\y)\Big|+\Big|\int_{\|\y\|\ge \|\x\|/2\atop \|\x-\y\|< \f{1}{2}(1+\|\x\|)} \varphi(t\sqrt L)(\x,\y)f(\y)dw(\y)\Big|\\
		&\le \Big|\int_{\|\y\|\ge \|\x\|/2\atop \|\x-\y\|\ge \f{1}{2}(1+\|\x\|)} \varphi(t\sqrt L)(\x,\y)f(\y)dw(\y)\Big|+\Big|\int_{\|\y\|\ge \|\x\|/2\atop \|\x-\y\|< \f{1}{2}(1+\|\x\|)} \varphi(t\sqrt L)(\x,\y)[f(\y)-f(\x)]dw(\y)\Big|\\
		&\ \ +\Big|\int_{\|\y\|\ge \|\x\|/2\atop \|\x-\y\|< \f{1}{2}(1+\|\x\|)} \varphi(t\sqrt L)(\x,\y)f(\x)dw(\y)\Big|\\
		&=:E_{11}+E_{12}+E_{13}.
	\end{aligned}
	\]
	Using the fact that 
	\[
	|f(\y)|\lesi P_\epsilon(0,\y;1)\lesi P_\epsilon(0,\x;1) \ \text{whenever $\|\y\|\ge \|\x\|/2$},
	\]
	and Lemma \ref{lem-heat kernel estimate}, we have
	\[
	\begin{aligned}
		E_{11}&\lesi P_\epsilon(0,\x;1)\int_{\|\x-\y\|\ge \f{1}{2}(1+\|\x\|)} \Big(\f{t}{\|\x-\y\|}\Big)^2 \f{1}{w(B(\x,  t))}\Big(\f{t}{t+d(\x,\y)}\Big)^{M}dw(\y)\\
		&\lesi \Big(\f{t}{1+\|\x\|}\Big)^2P_\epsilon(0,\x;1)\\
		&\lesi \Big(\f{t}{1+\|\x\|}\Big)^2\Big(\f{1}{1+\|\x\|}\Big)^{\epsilon-\gamma}P_\gamma(0,\x;1)\\
		&\lesi t^{-(\epsilon -\gamma)}\Big(\f{t}{1+\|\x\|}\Big)^\beta P_\gamma(0,\x;1).
	\end{aligned}
	\]
	By Lemma \ref{lem-heat kernel estimate} and Definition \ref{defn-Galphabeta} (ii),
	\[
	\begin{aligned}
		E_{12}&\lesi P_\epsilon(0,\x;1)\int_{\|\x-\y\|< \f{1}{2}(1+\|\x\|)} \Big(\f{t}{\|\x-\y\|}\Big)^\beta \Big(\f{\|\x-\y\|}{1+\|\x\|}\Big)^\beta \f{1}{w(B(\x,  t))}\Big(\f{t}{t+d(\x,\y)}\Big)^{M}dw(\y)\\
		&\lesi \Big(\f{t}{1+\|\x\|}\Big)^\beta P_\epsilon(0,\x;1)\\
		&\lesi t^{-(\epsilon -\gamma)}\Big(\f{t}{1+\|\x\|}\Big)^\beta P_\gamma(0,\x;1).
	\end{aligned}
	\]
	For the term $E_{13}$, since $\displaystyle \int \varphi(t\sqrt L)(\x,\y)dw(\y)=0$, we have
	\[
	\begin{aligned}
		E_{13}& = \Big|\int_{\|\y\|< \|\x\|/2\atop \text{or} \ \|\x-\y\|\ge  \f{1}{2}(1+\|\x\|)} \varphi(t\sqrt L)(\x,\y)f(\x)dw(\y)\Big|\\
		&\le \int_{\|\y\|< \|\x\|/2 } |\varphi(t\sqrt L)(\x,\y)||f(\x)|dw(\y) +\int_{\|\x-\y\|\ge  \f{1}{2}(1+\|\x\|)} |\varphi(t\sqrt L)(\x,\y)||f(\x)|dw(\y).
	\end{aligned}
	\]
	Similarly to $E_{11}$,
	\[
	\begin{aligned}
		\int_{\|\x-\y\|\ge  \f{1}{2}(1+\|\x\|)} |\varphi(t\sqrt L)(\x,\y)||f(\x)|dw(\y) &\lesi \Big(\f{t}{1+\|\x\|}\Big)^2P_\epsilon(0,\x;1)\\
		&\lesi t^{-(\epsilon -\gamma)}\Big(\f{t}{1+\|\x\|}\Big)^\beta P_\gamma(0,\x;1).
	\end{aligned}
	\]
	Since $d(\x,\y)\simeq \|\x\|$ whenever $\|\y\|\le \|\x\|/2$ and the fact that $w(B(\x,  \|\x\|))\simeq w(B(0,  \|\x\|))$, we have
	\[
	\begin{aligned}
		\int_{\|\y\|< \|\x\|/2 } |\varphi(t\sqrt L)(\x,\y)||f(\x)|dw(\y)&\lesi P_\epsilon(0,\x;1)\int_{\|\y\|\le \|\x\|/2}\f{1}{w(B(\x,  \|\x\|))}\Big(\f{t}{t+\|\x\|}\Big)^{\beta}dw(\y)\\
		&\lesi \Big(\f{t}{t+\|\x\|}\Big)^\beta P_\epsilon(0,\x;1)\int_{\|\y\|\le \|\x\|/2}\f{1}{w(B(0,  \|\x\|))}dw(\y)\\
		&\lesi \Big(\f{t}{1+\|\x\|}\Big)^\beta P_\epsilon(0,\x;1)\\
		&\lesi t^{-(\epsilon -\gamma)}\Big(\f{t}{1+\|\x\|}\Big)^\beta P_\gamma(0,\x;1).
	\end{aligned}
	\]
	Consequently,
	\begin{equation}\label{eq- E1 t large}
		E_1\lesi 	 t^{-(\epsilon -\gamma)}\Big(\f{t}{1+\|\x\|}\Big)^\beta P_\gamma(0,\x;1).
	\end{equation}
	For the term $E_{2}$, if $\|\x\|\le 1$, then $
	|f(\y)|\lesi P_\epsilon(0,\y;1)\simeq P_\epsilon(0,\x;1)$. Therefore, by Lemma \ref{lem-heat kernel estimate},
	\[
	\begin{aligned}
		E_{2}&\lesi P_\epsilon(0,\x;1)\int_{\RN} |\varphi(t\sqrt L)(\x,\y)| dw(\y)\\
		&\lesi  \Big(\f{1}{1+\|\x\|}\Big)^{\epsilon-\gamma}P_\gamma(0,\x;1)\\
		&\lesi t^{-(\epsilon -\gamma)}  P_\gamma(0,\x;1)\\
		&\simeq t^{-(\epsilon -\gamma)}\Big(\f{t}{1+\|\x\|}\Big)^\beta P_\gamma(0,\x;1).
	\end{aligned}
	\]
	
	We now consider the case $\|\x\|\ge 1$. In this case, we have
	\[
	\begin{aligned}
		E_2&\lesi \Big|\int_{t\le \|\y\|< \|\x\|/2} \varphi(t\sqrt L)(\x,\y)f(\y)dw(\y)\Big|+ \Big|\int_{\|\y\|< \min\{\|\x\|/2,  t\}} \varphi(t\sqrt L)(\x,\y)f(\y)dw(\y)\Big|\\
		&\lesi \Big|\int_{t\le \|\y\|< \|\x\|/2} \varphi(t\sqrt L)(\x,\y)f(\y)dw(\y)\Big|+ \Big|\int_{\|\y\|< \min\{\|\x\|/2,  t\}} [\varphi(t\sqrt L)(\x,\y)-\varphi(t\sqrt L)(\x,0)]f(\y)dw(\y)\Big|\\
		& \ \ \ +\Big|\int_{\|\y\|< \min\{\|\x\|/2,  t\}} \varphi(t\sqrt L)(\x,0)f(\y)dw(\y)\Big|\\
		&=:E_{21}+E_{22}+E_{23}.
	\end{aligned}
	\]
	For the term $E_{21}$, note that in this situation $d(\x,\y)\simeq \|\x\|$. Hence, by Lemma \ref{lem-heat kernel estimate} and (i) in Definition \ref{defn-Galphabeta},
	\[
	\begin{aligned}
		E_{21}&\lesi \int_{t\le \|\y\|< \|\x\|/2} \f{1}{w(B(\x,\|\x\|))}\Big(\f{t}{t+\|\x\|}\Big)^{M+\beta} P_\epsilon(0,\y;1)dw(\y).
		%&\lesi \int_{t\le \|\y\|< \|\x\|/2} \f{1}{w(B(\x,\|\x\|))}\exp\Big(-\f{\|\x\|^2}{c't^2}\Big) \f{1}{w(B(0,t))} \Big(\f{1}{1+\|\y\|}\Big)^{ \epsilon}dw(\y).
	\end{aligned}
	\]
	Since $w(B(\x,\|\x\|))\simeq w(B(0,\|\x\|))$ and the following inequality
	\[
	P_\epsilon(0,\y;1)\le t^{-(\epsilon-\gamma)}\f{1}{w(B(0,t))}\Big(\f{1}{1+\|\y\|}\Big)^{ \gamma}, \ \ \ \|\y\|\ge t,
	\]
	we have
	\[
	\begin{aligned}
		E_{21}
		&\lesi t^{-(\epsilon-\gamma)}\int_{t\le \|\y\|< \|\x\|/2} \f{1}{w(B(0,\|\x\|))}\Big(\f{t}{t+\|\x\|}\Big)^{\mathfrak{N}+\beta+\gamma} \f{1}{w(B(0,t))} \Big(\f{1}{1+\|\y\|}\Big)^{ \gamma}dw(\y)\\
		&\lesi t^{-(\epsilon-\gamma)}\Big(\f{t}{t+\|\x\|}\Big)^{\mathfrak{N}+\beta}P_\gamma(0,\x;1) \f{w(B(0,\|\x\|))}{w(B(0,t))}\\
		&\lesi t^{-(\epsilon-\gamma)}\Big(\f{t}{1+\|\x\|}\Big)^{\beta}P_\gamma(0,\x;1),
	\end{aligned}
	\]
	where in the last inequality we used \eqref{eq-ratios on volumes of balls}.

	We now take care of the term $E_{22}$. By Lemma \ref{lem-heat kernel estimate} and the fact that $d(\x,\y)\simeq \|\x\|$ whenever $\|\y\|\le \|\x\|/2$,
	\[
	\begin{aligned}
		E_{22}&\lesi \int_{\|\y\|< \min\{\|\x\|/2,  t\}} \f{\|\y\|}{t} \f{1}{w(B(\x,\|\x\|))}\Big(\f{t}{t+\|\x\|}\Big)^{\gamma+\beta} P_\epsilon(0,\y;1)dw(\y)\\
		&\lesi \int_{\|\y\|< t} \f{\|\y\|^{1-\epsilon}}{t}\f{1}{w(B(0,\|\y\|))} \f{1}{w(B(\x,\|\x\|))}\Big(\f{t}{t+\|\x\|}\Big)^{\gamma+\beta} dw(\y).
	\end{aligned}
	\]
	Since 
	\[
	\int_{\|\y\|< t}  \|\y\|^{1-\epsilon} \f{1}{w(B(0,\|\y\|))}  dw(\y)\lesi t^{1-\epsilon},
	\]	
	we have
	\[
	\begin{aligned}
		E_{22}&\lesi t^{-\epsilon}  \f{1}{w(B(\x,\|\x\|))}\Big(\f{t}{t+\|\x\|}\Big)^{\gamma+\beta}\\
		&\lesi  t^{-\epsilon} \f{1}{w(B(0,\|\x\|))}\Big(\f{t}{t+\|\x\|}\Big)^{\gamma+\beta}\\
		&\lesi t^{-(\epsilon-\gamma)}\Big(\f{t}{1+\|\x\|}\Big)^{\beta}P_\gamma(0,\x;1).
	\end{aligned}
	\]
	For the term $E_{23}$, note that since $\displaystyle \int f(\y)dw(\y)=0$, we have
	\[
	\begin{aligned}
		E_{23}=\Big|\int_{\|\y\|\ge  \min\{\|\x\|/2,  t\}} \varphi(t\sqrt L)(\x,0)f(\y)dw(\y)\Big|.
	\end{aligned}
	\]
	Since $t< 1+\|\x\|\le 2\|\x\|$, we have
	\[
	\begin{aligned}
		E_{23}\le \int_{\|\y\|\ge  t/4} |\varphi(t\sqrt L)(\x,0)||f(\y)|dw(\y).
	\end{aligned}
	\]
	Using Lemma \ref{lem-heat kernel estimate} and (i) in Definition \ref{defn-Galphabeta},
	\[
	\begin{aligned}
		E_{23}%&\lesi \int_{\|\y\|\ge  t/4} \f{1}{w(B(0,t))}\Big(\f{t}{t+\|\x\|}\Big)^{\gamma+\beta} P_\epsilon(0,\y;1)dw(\y)\\
		&\lesi \int_{\|\y\|\ge  t/4} \f{1}{w(B(0,\|\x\|))}\Big(\f{t}{t+\|\x\|}\Big)^{\gamma+\beta} P_\epsilon(0,\y;1)dw(\y).
	\end{aligned}
	\]
	Since 
	\[
	\int_{\|\y\|\ge  t/4}  P_\epsilon(0,\y;1)dw(\y)\lesi t^{-\epsilon},
	\]
	we have
	\[
	\begin{aligned}
		E_{23} &\lesi  t^{-\epsilon}\f{1}{w(B(0,\|\x\|))}\Big(\f{t}{t+\|\x\|}\Big)^{\gamma+\beta}\\
		&\lesi t^{-(\epsilon-\gamma)}\Big(\f{t}{1+\|\x\|}\Big)^{\beta}P_\gamma(0,\x;1).
	\end{aligned}
	\]
	To sum up, we have
	\begin{equation}\label{eq-E}
		|\varphi(t\sqrt L)f(\x)|\le t^{-(\epsilon-\gamma)}\Big(\f{t}{1+\|\x\|}\Big)^{\beta}P_\gamma(0,\x;1), \ \ \x\in \RN, 1+\|\x\|\ge t.
	\end{equation}
	\bigskip
	
	\textbf{Case 2: $1+\|\x\|< t$}

	%For $x\in \RN$, set $S(\x)=\{\y: \|\x-\y\|\le \f{1}{2}(1+\|\x\|)\}$. 
	We have
	\[
	\begin{aligned}
		|\varphi(t\sqrt L)f(\x)| &= \Big|\int_{\RN} \varphi(t\sqrt L)(\x,\y)f(\y)dw(\y)\\
		&\le \Big|\int_{\|\y\|>\|\x\|/2} \varphi(t\sqrt L)(\x,\y)f(\y)dw(\y)\Big|+\Big|\int_{\|\y\|\le \|\x\|/2} \varphi(t\sqrt L)(\x,\y)f(\y)dw(\y)\Big|\\
		&=: F_1+F_2.
	\end{aligned}
	\]
	
	For the term $F_1$, %by (ii) in Definition \ref{defn-Galphabeta} and $\|\y\|>\|\x\|/2$,
	%\[
	%\begin{aligned}
	%	|f(\y)|&\le  P_\epsilon(0,\y;1)\\
	%	& \lesi  P_\delta(0,\y;1) \Big(\f{1}{1+\|\y\|}\Big)^{\gamma+(\epsilon-\gamma -\delta )}\\
	%	& \lesi  P_\delta(0,\y;1) \Big(\f{1}{1+\|\x\|}\Big)^{\gamma+(\epsilon-\gamma -\delta )}.
	%\end{aligned}
	%\]
	using  Lemma \ref{lem-heat kernel estimate},
	\[
	\begin{aligned}
		F_1&\lesi   \int_{\|\y\|>\|\x\|/2} \f{1}{w(B(\x,t))} P_\epsilon(0,\y;1) dw(\y)\\
		&\lesi \f{1}{w(B(\x,t))}\Big(\f{1}{1+\|\x\|}\Big)^{\epsilon}.
	\end{aligned}
	\]
	Since $t>1+\|\x\|$, using \eqref{eq-ratios on volumes of balls} and \eqref{eq-doubling} we obtain
	\[
	\begin{aligned}
		\f{1}{w(B(\x,t))}&\lesi \f{1}{w(B(\x,1+\|\x\|))}\Big(\f{1+\|\x\|}{t}\Big)^N\\
		&\lesi \f{1}{w(B(\x,1+\|\x\|))}\Big(\f{1+\|\x\|}{t}\Big)^{\epsilon-\gamma}\\
		&\lesi \f{1}{w(B(0,1+\|\x\|))}\Big(\f{1+\|\x\|}{t}\Big)^{\epsilon-\gamma}.
	\end{aligned}
	\]
	Therefore,
	\[
	\begin{aligned}
		F_1&\lesi \f{1}{w(B(0,1+\|\x\|))}\Big(\f{1+\|\x\|}{t}\Big)^{\epsilon-\gamma}\Big(\f{1}{1+\|\x\|}\Big)^{\epsilon}\\
		&\lesi t^{-(\epsilon-\gamma)} P_\gamma(0,\x;1).
	\end{aligned}
	\]
	
	For $F_2$, we write
	\[
	\begin{aligned}
		F_{2}&\le \Big|\int_{\|\y\|\le \|\x\|/2} [\varphi(t\sqrt L)(\x,\y)-\varphi(t\sqrt L)(\x,0)]f(\y)dw(\y)\Big|+\Big|\int_{\|\y\|\le \|\x\|/2} \varphi(t\sqrt L)(\x,0)f(\y)dw(\y)\Big|.
	\end{aligned}
	\]
	On the other hand, since $\displaystyle \int f(\y)dw(\y)=0$, 
	\[
	\Big|\int_{\|\y\|\le \|\x\|/2} \varphi(t\sqrt L)(\x,0)f(\y)dw(\y)\Big|=\Big|\int_{\|\y\|> \|\x\|/2} \varphi(t\sqrt L)(\x,0)f(\y)dw(\y)\Big|.
	\]
	Therefore,
	\[
	\begin{aligned}
		F_{2}&\le \Big|\int_{\|\y\|\le \|\x\|/2} [\varphi(t\sqrt L)(\x,\y)-\varphi(t\sqrt L)(\x,0)]f(\y)dw(\y)\Big|+\Big|\int_{\|\y\|> \|\x\|/2} \varphi(t\sqrt L)(\x,0)f(\y)dw(\y)\Big|\\
		&=: F_{21} + F_{22}.
	\end{aligned}
	\]
	Arguing similarly to the estimate of $F_1$, we have
	\[
	F_{22}\lesi t^{-(\epsilon-\gamma)} P_\gamma(0,\x;1).
	\]
	
	For the term $F_{21}$, since $\|\y\|\le \|\x\|/2 <t$, using Lemma  \ref{lem-heat kernel estimate} and (i) in Definition \ref{defn-Galphabeta},
	\[
	\begin{aligned}
		F_{21}&\lesi  \int_{\|\y\|\le \|\x\|}\f{\|\y\|}{t} \f{1}{w(B(0,t+\|\x\|))} \Big(\f{t}{t+\|\x\|}\Big)^{\gamma}P_\epsilon(0,\y;1) dw(\y)\\
		&\lesi  \int_{\|\y\|\le \|\x\|}\f{\|\y\|^{1-\epsilon}}{t} \f{1}{w(B(0,1+\|\x\|))}\Big(\f{t}{t+\|\x\|}\Big)^{\gamma} \f{1}{w(B(0,\|\y\|))} dw(\y).
	\end{aligned}
	\]
	Since 
	\[
	\int_{\|\y\|\le \|\x\|} \|\y\|^{1-\epsilon}\f{1}{w(B(0,\|\y\|))} dw(\y)\lesi \|\x\|^{1-\epsilon}
	\]	
	and $\|\x\|\le 2t$, we further obtain
	\[
	\begin{aligned}
		F_{21}	
		&\lesi   \f{\|\x\|^{1-\epsilon}}{t} \Big(\f{t}{t+\|\x\|}\Big)^{\gamma}\f{1}{w(B(0,1+\|\x\|))}\\
		&\lesi   \f{1}{t^\epsilon} \Big(\f{t}{t+\|\x\|}\Big)^{\gamma}\f{1}{w(B(0,1+\|\x\|))}\\
		&\lesi  t^{-(\epsilon-\gamma)} P_\gamma(0,\x;1). 
	\end{aligned}
	\]
	Therefore, in this case 
	\begin{equation}\label{eq-E t large}
		|\varphi(t\sqrt L)f(\x)|\le t^{-(\epsilon-\gamma)}P_\gamma(0,\x;1), \ \ x\in \RN, t\ge 1+\|\x\|.
	\end{equation}
	In both case, we have proved that 
	\[
	|\varphi(t\sqrt L)f(\x)|\le t^{-(\epsilon-\gamma)}P_\gamma(0,\x;1), \ \ x\in \RN, t\ge .
	\]
	\bigskip
	
	Hence, 	it suffices to prove that for $\x,\x'\in \RN$ with $\|\x-\x'\|\le \f{1}{2}(1+\|\x\|)$,  
	\begin{equation}
		\label{eq-condition ii for tLetL t large}
		|\varphi(t\sqrt L)f(\x)-\varphi(t\sqrt L)f(\x')|\lesi t^{-(\epsilon-\gamma)}\Big(\f{\|\x-\x'\|}{1+\|\x\|}\Big)^{\beta}P_\gamma(0,\x;1).
	\end{equation}
	
	Indeed, observe that 
	\[
	\begin{aligned}
		\varphi(t\sqrt L)f(\x)-\varphi(t\sqrt L)f(\x') &= \int_{\RN} [\varphi(t\sqrt L)(\x,\y)-\varphi(t\sqrt L)(\x',\y)]f(\y)dw(\y)\\
	\end{aligned}
	\]
	For $t>0$ and $\x,\x',\y\in \RN$, set 
	\[
	q_{t}(\x,\x',\y)= \varphi(t\sqrt L)(\x,\y)-\varphi(t\sqrt L)(\x',\y).
	\]
	Hence, we can write
	\[
	\varphi(t\sqrt L)f(\x)-\varphi(t\sqrt L)f(\x') = \int_{\RN} q_{t}(\x,\x',\y)f(\y)dw(\y).
	\]
	In addition,  from Lemma \ref{lem-heat kernel estimate}, we have, for $\|\x-\x'\|\le t$ and $\|\y-\y'\|\le t$,
	\[
	|q_{t}(\x,\x',\y)|\lesi \f{\|\x-\x'\|}{t}\Big(1+\f{\|\x-\y\|}{  t}\Big)^{-2}\f{1}{w(B(\x,  t))}\exp\Big(-\f{d(\x,\y)^2}{ct^2}\Big),
	\]
	\[
	|q_{t}(\x,\x',\y)-q_{t}(\x,\x',\y')|\lesi \f{\|\x-\x'\|}{t}\f{\|\y-\y'\|}{t}\Big(1+\f{\|\x-\y\|}{  t}\Big)^{-2}\f{1}{w(B(\x,  t))}\exp\Big(-\f{d(\x,\y)^2}{ct^2}\Big),
	\]
	and
	\[
	\int_{\RN}q_{t}(\x,\x',\y)dw(\y)=0 \ \ \ \ \x,\x'\in \RN.
	\]
	
	\textbf{Case 1: $\|\x-\x'\|<t$}
	
	In this situation, we also consider two subcases: $1+\|\x\|\ge  t$ and $1+\|\x\|< t$. For the first subcase $1+\|\x\|\ge  t$, arguing similarly to when  At this stage, arguing similarly to the proof of \eqref{eq-E}, we obtain
	\[
	\begin{aligned}
		\varphi(t\sqrt L)f(\x)-\varphi(t\sqrt L)f(\x')&\lesi t^{-(\epsilon-\gamma)}\f{\|\x-\x'\|}{t}\Big(\f{t}{1+\|\x\|}\Big)^{\beta}P_\gamma(0,\x;1)\\
		&\lesi t^{-(\epsilon-\gamma)}\Big(\f{\|\x-\x'\|}{t}\Big)^\beta\Big(\f{t}{1+\|\x\|}\Big)^{\beta}P_\gamma(0,\x;1)\\
		&\lesi t^{-(\epsilon-\gamma)}\Big(\f{\|\x-\x'\|}{1+\|\x\|}\Big)^{\beta}P_\gamma(0,\x;1),
	\end{aligned}
	\]
	as long as $\|\x-\x'\|\le \f{1}{2}(1+\|\x\|)$.
	
	For the second subcase $t>1+\|\x\|$, arguing similarly to \eqref{eq-E t large},
	\[
	\begin{aligned}
		\varphi(t\sqrt L)f(\x)-\varphi(t\sqrt L)f(\x')&\lesi t^{-(\epsilon-\gamma)}\f{\|\x-\x'\|}{t}\ P_\gamma(0,\x;1)\\
		&\lesi t^{-(\epsilon-\gamma)}\Big(\f{\|\x-\x'\|}{t}\Big)^\beta P_\gamma(0,\x;1)\\
		&\lesi t^{-(\epsilon-\gamma)}\Big(\f{\|\x-\x'\|}{1+\|\x\|}\Big)^{\beta}P_\gamma(0,\x;1),
	\end{aligned}
	\]
	as long as $\|\x-\x'\|\le \f{1}{2}(1+\|\x\|)$.
	
	\bigskip
	
	\textbf{Case 2: $t\le \|\x-\x'\|<\f{1}{2}(1+\|\x\|)$}
	
	We now write 
	\[
	|\varphi(t\sqrt L)f(\x)-\varphi(t\sqrt L)f(\x')|\le |\varphi(t\sqrt L)f(\x)|+|\varphi(t\sqrt L)f(\x')|.
	\]
	From the inequalities \eqref{eq-E} and \eqref{eq- E1 t large}, we have
	\[
	|\varphi(t\sqrt L)f(\x)|\lesi t^{-(\epsilon-\gamma)}\Big(\f{t}{1+\|\x\|}\Big)^{\beta}P_\gamma(0,\x;1)\lesi t^{-(\epsilon-\gamma)}\Big(\f{\|\x-\x'\|}{1+\|\x\|}\Big)^{\beta}P_\gamma(0,\x;1)
	\]
	and
	\[
	|\varphi(t\sqrt L)f(\x')|\lesi t^{-(\epsilon-\gamma)}\Big(\f{t}{1+\|\x\|}\Big)^{\beta}P_\gamma(0,\x;1)\lesi t^{-(\epsilon-\gamma)}\Big(\f{\|\x-\x'\|}{1+\|\x'\|}\Big)^{\beta}P_\gamma(0,\x';1).
	\]
	Therefore,
	\[
	|\varphi(t\sqrt L)f(\x)-\varphi(t\sqrt L)f(\x')|\lesi t^{-(\epsilon-\gamma)}\Big(\f{\|\x-\x'\|}{1+\|\x\|}\Big)^{\beta}P_\gamma(0,\x;1)+t^{-(\epsilon-\gamma)}\Big(\f{\|\x-\x'\|}{1+\|\x'\|}\Big)^{\beta}P_\gamma(0,\x';1).
	\]
	On the other hand,  it is easy to see that $1+\|\x\|\simeq 1+ \|\x'\|$, whenever $\|\x-\x'\|\le \f{1}{2}(1+\|\x\|)$. It follows $P_\gamma(0,\x;1)\simeq P_\gamma(0,\x';1)$. Consequently,
	\[
	|\varphi(t\sqrt L)f(\x)-\varphi(t\sqrt L)f(\x')|\lesi t^{-(\epsilon-\gamma)}\Big(\f{\|\x-\x'\|}{1+\|\x\|}\Big)^{\beta}P_\gamma(0,\x;1).
	\]
	This completes the inequality \eqref{eq-condition ii for tLetL t large} and hence  our proof is completed.
	
\end{proof}

\begin{lem}
	\label{lem - t large for integral}  	Let $\varphi \in \mathscr S(\mathbb R)$ be an even function satisfying $\varphi(0)=0$. Let  $0<\beta,\gamma<1$. For  $f\in \mathring{\mathcal G}(\beta,\gamma)$, then we have 
	\begin{equation}\label{eq- inequality t large for integral}
		\Big\|\int_2^\vc \varphi(t\sqrt L)f\f{dt}{t}\Big\|_{\mathring{\mathcal G}(\beta,\gamma)}\lesi  \|f\|_{\mathring{\mathcal G}(\beta,\gamma)}.
	\end{equation}
\end{lem}
\begin{proof}
	In this proof, $M>\mathfrak{N}$ will be a fixed number.
	
	Without loss of generality, assume that $\|f\|_{\mathring{\mathcal G}(\beta,\gamma)}=1$. We have
	\[
	\begin{aligned}
		\varphi(t\sqrt L)f(\x) &=  \int_{\RN} \varphi(t\sqrt L)(\x,\y)f(\y)dw(\y).
	\end{aligned}
	\]
	It suffices to show that following inequalities 
	\begin{equation}\label{eq- eq1 proof of integral}
		|\varphi(t\sqrt L)f(\x)|\lesi \min\Big\{\Big(\f{t}{1+\|\x\|}\Big)^\beta,\Big(\f{1+\|\x\|}{t}\Big)^{1-\gamma} 
		\Big\} P_\gamma(0,\x;1), \ \ t>2, \x\in \RN,
	\end{equation}
	and
	\begin{equation}
		\label{eq-condition ii for tLetL t large for integral}
		|\varphi(t\sqrt L)f(\x)-\varphi(t\sqrt L)f(\x')|\lesi \min\Big\{\Big(\f{\|\x-\x'\|}{t}\Big)^{1-\beta}, \Big(\f{t}{\|\x-\x'\|}\Big)^\beta\Big\}\Big(\f{\|\x-\x'\|}{1+\|\x\|}\Big)^{\beta}P_\gamma(0,\x;1)
	\end{equation}
	for $t>2$ and $\x,\x'\in \RN$ with $\|\x-\x'\|\le \f{1}{2}(1+\|\x\|)$.

	We first take care of \eqref{eq- eq1 proof of integral}. If  $1+\|\x\|\ge t$ arguing similarly to the proof of \eqref{eq-E}, we obtain 
	\begin{equation}\label{eq-E integral}
		|\varphi(t\sqrt L)f(\x)|\lesi  \Big(\f{t}{1+\|\x\|}\Big)^{\beta}P_\gamma(0,\x;1), \ \ \x\in \RN, t\le 1+\|\x\|.
	\end{equation}
	\bigskip
	
	If $1+\|\x\|< t$,  then we have
	\[
	\begin{aligned}
		|\varphi(t\sqrt L)f(\x)| &= \Big|\int_{\RN} \varphi(t\sqrt L)(\x,\y)f(\y)dw(\y)\Big|\\
		&\le \Big|\int_{\|\y\|>\|\x\|/2} \varphi(t\sqrt L)(\x,\y)f(\y)dw(\y)\Big|+\Big|\int_{\|\y\|\le \|\x\|/2} \varphi(t\sqrt L)(\x,\y)f(\y)dw(\y)\Big|\\
		&=: E_1+E_2.
	\end{aligned}
	\]

	For the term $E_1$, using  Lemma \ref{lem-heat kernel estimate},
	\[
	\begin{aligned}
		E_1&\lesi   \int_{\|\y\|>\|\x\|/2} \f{1}{w(B(\x,t))} P_\gamma(0,\y;1) dw(\y)\\
		&\lesi \f{1}{w(B(\x,t))}\Big(\f{1}{1+\|\x\|}\Big)^{\gamma}.
	\end{aligned}
	\]
	Since $t>1+\|\x\|$, using \eqref{eq-ratios on volumes of balls} and \eqref{eq-doubling} we obtain
	\[
	\begin{aligned}
		\f{1}{w(B(\x,t))}&\lesi \f{1}{w(B(\x,1+\|\x\|))}\Big(\f{1+\|\x\|}{t}\Big)^N\\
		&\lesi \f{1}{w(B(\x,1+\|\x\|))}\Big(\f{1+\|\x\|}{t}\Big)^{1-\gamma}\\
		&\lesi \f{1}{w(B(0,1+\|\x\|))}\Big(\f{1+\|\x\|}{t}\Big)^{1-\gamma}.
	\end{aligned}
	\]
	Therefore,
	\[
	\begin{aligned}
		E_1&\lesi \Big(\f{1+\|\x\|}{t}\Big)^{1-\gamma} P_\gamma(0,\x;1).
	\end{aligned}
	\]
	
	For $E_2$, we write
	\[
	\begin{aligned}
		E_{2}&\le \Big|\int_{\|\y\|\le \|\x\|/2} [\varphi(t\sqrt L)(\x,\y)-\varphi(t\sqrt L)(\x,0)]f(\y)dw(\y)\Big|+\Big|\int_{\|\y\|\le \|\x\|/2} \varphi(t\sqrt L)(\x,0)f(\y)dw(\y)\Big|.
	\end{aligned}
	\]
	On the other hand, since $\displaystyle \int f(\y)dw(\y)=0$, 
	\[
	\Big|\int_{\|\y\|\le \|\x\|/2} \varphi(t\sqrt L)(\x,0)f(\y)dw(\y)\Big|=\Big|\int_{\|\y\|> \|\x\|/2} \varphi(t\sqrt L)(\x,0)f(\y)dw(\y)\Big|.
	\]
	Therefore,
	\[
	\begin{aligned}
		E_2&\le \Big|\int_{\|\y\|\le \|\x\|/2} [\varphi(t\sqrt L)(\x,\y)-\varphi(t\sqrt L)(\x,0)]f(\y)dw(\y)\Big|+\Big|\int_{\|\y\|> \|\x\|/2} \varphi(t\sqrt L)(\x,0)f(\y)dw(\y)\Big|\\
		&=: E_{21} + E_{22}.
	\end{aligned}
	\]
	Arguing similarly to the estimate of $E_1$, we have
	\[
	E_{22}\lesi \Big(\f{1+\|\x\|}{t}\Big)^{1-\gamma} P_\gamma(0,\x;1).
	\]
	
	For the term $E_{21}$, since $\|\y\|\le \|\x\|/2 <t$, using Lemma  \ref{lem-heat kernel estimate} and (i) in Definition \ref{defn-Galphabeta},
	\[
	\begin{aligned}
		E_{21}&\lesi  \int_{\|\y\|\le \|\x\|}\f{\|\y\|}{t} \f{1}{w(B(0,t+\|\x\|))} \Big(\f{t}{t+\|\x\|}\Big)^{\gamma}P_\gamma(0,\y;1) dw(\y)\\
		&\lesi  \int_{\|\y\|\le \|\x\|}\f{\|\y\|^{1-\gamma}}{t} \f{1}{w(B(0,1+\|\x\|))}\Big(\f{t}{t+\|\x\|}\Big)^{\gamma} \f{1}{w(B(0,\|\y\|))} dw(\y).
	\end{aligned}
	\]
	Since 
	\[
	\int_{\|\y\|\le \|\x\|} \|\y\|^{1-\gamma}\f{1}{w(B(0,\|\y\|))} dw(\y)\lesi \|\x\|^{1-\gamma}
	\]	
	and $\|\x\|\le 2t$, we further obtain
	\[
	\begin{aligned}
		F_{21}	
		&\lesi   \f{\|\x\|^{1-\gamma}}{t} \Big(\f{t}{t+\|\x\|}\Big)^{\gamma}\f{1}{w(B(0,1+\|\x\|))}\\
		&\lesi   \f{(1+\|\x\|)^{1-\gamma}}{t} \Big(\f{t}{1+\|\x\|}\Big)^{\gamma}\f{1}{w(B(0,1+\|\x\|))}\\
		&\lesi     \Big(\f{1+\|\x\|}{t}\Big)^{1-\gamma} P_\gamma(0,\x;1). 
	\end{aligned}
	\]
	Therefore, in this case 
	\begin{equation}\label{eq-E t large integral}
		|\varphi(t\sqrt L)f(\x)|\le \Big(\f{1+\|\x\|}{t}\Big)^{1-\gamma}P_\gamma(0,\x;1), \ \ \x\in \RN, t\ge 1+\|\x\|.
	\end{equation}
	This ensures \eqref{eq- eq1 proof of integral}.
	
	\bigskip

	In order to prove \eqref{eq-condition ii for tLetL t large for integral}, we  observe that 
	\[
	\begin{aligned}
		\varphi(t\sqrt L)f(\x)-\varphi(t\sqrt L)f(\x') &= \int_{\RN} [\varphi(t\sqrt L)(\x,\y)-\varphi(t\sqrt L)(\x',\y)]f(\y)dw(\y)\\
	\end{aligned}
	\]
	For $t>0$ and $\x,\x',\y\in \RN$, set 
	\[
	q_{t}(\x,\x',\y)= \varphi(t\sqrt L)(\x,\y)-\varphi(t\sqrt L)(\x',\y).
	\]
	Hence, we can write
	\[
	\varphi(t\sqrt L)f(\x)-\varphi(t\sqrt L)f(\x') = \int_{\RN} q_{t}(\x,\x',\y)f(\y)dw(\y).
	\]
	In addition,  from Lemma \ref{lem-heat kernel estimate}, we have, for $|\x-\x'|\le t$ and $|\y-\y'|\le t$,
	\[
	|q_{t}(\x,\x',\y)|\lesi \f{\|\x-\x'\|}{t}\Big(1+\f{\|\x-\y\|}{  t}\Big)^{-2}\f{1}{w(B(\x,  t))}\exp\Big(-\f{d(\x,\y)^2}{ct^2}\Big),
	\]
	\[
	|q_{t}(\x,\x',\y)-q_{t}(\x,\x',\y')|\lesi \f{\|\x-\x'\|}{t}\f{\|\y-\y'\|}{t}\Big(1+\f{\|\x-\y\|}{  t}\Big)^{-2}\f{1}{w(B(\x,  t))}\exp\Big(-\f{d(\x,\y)^2}{ct^2}\Big),
	\]
	and
	\[
	\int_{\RN}q_{t}(\x,\x',\y)dw(\y)=0 \ \ \ \ \x,\x'\in \RN.
	\]
	
	\textbf{Case 1: $\|\x-\x'\|<t$}
	
	In this situation, arguing similarly to the proof of \eqref{eq- eq1 proof of integral}, if  $1+\|\x\|\ge  t$, we have
	\[
	\begin{aligned}
		|\varphi(t\sqrt L)f(\x)-\varphi(t\sqrt L)f(\x')|&\lesi  \f{\|\x-\x'\|}{t}\Big(\f{t}{1+\|\x\|}\Big)^{\beta}P_\gamma(0,\x;1)\\
		&\lesi  \Big(\f{\|\x-\x'\|}{t}\Big)^{1-\beta}\Big(\f{\|\x-\x'\|}{t}\Big)^\beta\Big(\f{t}{1+\|\x\|}\Big)^{\beta}P_\gamma(0,\x;1)\\
		&\lesi \Big(\f{\|\x-\x'\|}{t}\Big)^{1-\beta}\Big(\f{\|\x-\x'\|}{1+\|\x\|}\Big)^{\beta}P_\gamma(0,\x;1),
	\end{aligned}
	\]
	and if  $t>1+\|\x\|$, we have
	\[
	\begin{aligned}
		|\varphi(t\sqrt L)f(\x)-\varphi(t\sqrt L)f(\x')|&\lesi \f{\|\x-\x'\|}{t}\Big(\f{1+\|\x\|}{t}\Big)^{1-\gamma}P_\gamma(0,\x;1)\\
		&\lesi  \f{\|\x-\x'\|}{t}\ P_\gamma(0,\x;1)\\
		&\lesi \Big(\f{\|\x-\x'\|}{t}\Big)^{1-\beta}\Big(\f{\|\x-\x'\|}{t}\Big)^\beta\Big(\f{\|\x-\x'\|}{t}\Big)^\beta P_\gamma(0,\x;1)\\
		&\lesi \Big(\f{\|\x-\x'\|}{t}\Big)^{1-\beta}\Big(\f{\|\x-\x'\|}{1+\|\x\|}\Big)^{\beta}P_\gamma(0,\x;1).
	\end{aligned}
	\]

	\bigskip
	
	\textbf{Case 2: $t\le \|\x-\x'\|<\f{1}{2}(1+\|\x\|)$}
	
	We now write 
	\[
	|\varphi(t\sqrt L)f(\x)-\varphi(t\sqrt L)f(\x')|\le |\varphi(t\sqrt L)f(\x)|+|\varphi(t\sqrt L)f(\x')|.
	\]
	From \eqref{eq-E integral}, arguing similarly to \eqref{eq-condition ii for tLetL t large}, we come up with \eqref{eq-condition ii for tLetL t large for integral}.
	
	This completes our proof.
	
\end{proof}
As a byproduct, we also obtain the following estimate.
\begin{lem}
	\label{lem-estimate for varphi t f x}
	Let $\varphi \in \mathscr S(\mathbb R)$ be an even function satisfying $\varphi(0)=0$. Let $0<\beta,\gamma<1$. Then for every $f\in \mathring{\mathcal G}(\beta,\gamma)$ and $\x\in \RN$,
	\begin{equation}\label{eq- pointwise estimate t small}
		|\varphi(t\sqrt L)f(\x)|\lesi t^\beta P_\gamma(0,\x;1)\|f\|_{\mathring{\mathcal G}(\beta,\gamma)},  \ \ \ t\in(0,1),
	\end{equation}
	and
	\begin{equation}\label{eq- pointwise estimate t large}
		|\varphi(t\sqrt L)f(\x)|\lesi t^{-\gamma} P_{\gamma\wedge \beta}(0,\x;t)\|f\|_{\mathring{\mathcal G}(\beta,\gamma)},  \ \  t\ge 1.
	\end{equation}
\end{lem}
\begin{proof}
	The estimate \eqref{eq- pointwise estimate t small} follows directly from \eqref{eq1- proof of integral t small}.
	
	For the estimate \eqref{eq- pointwise estimate t large}, we consider two cases. If $t\le 1 + \|\x\|$, we proved in \eqref{eq-E integral} that \[
	\begin{aligned}
		|\varphi(t\sqrt L)f(\x)|&\lesi  \Big(\f{t}{1+\|\x\|}\Big)^{\beta}P_\gamma(0,\x;1)\\
		&\lesi t^{-\gamma} \Big(\f{t}{1+\|\x\|}\Big)^{\beta}\f{1}{w(B(0,1+\|\x\|))}\\
		&\simeq t^{-\gamma} P_\beta(0,\x;t).
	\end{aligned}
	\]
	If $t>1+\|\x\|$, then we write
	\[
	\begin{aligned}
		|\varphi(t\sqrt L)f(\x)| &= \Big|\int_{\RN} \varphi(t\sqrt L)(\x,\y)f(\y)dw(\y)\\
		&\le \Big|\int_{\|\y\|<t} \varphi(t\sqrt L)(\x,\y)f(\y)dw(\y)\Big|+\Big|\int_{\|\y\|\ge t} \varphi(t\sqrt L)(\x,\y)f(\y)dw(\y)\Big|\\
		&\le \Big|\int_{\|\y\|<t} [\varphi(t\sqrt L)(\x,\y)-\varphi(t\sqrt L)(\x,0)]f(\y)dw(\y)\Big|+\Big|\int_{\|\y\|<t} \varphi(t\sqrt L)(\x,0)f(\y)dw(\y)\Big|\\
		&+\Big|\int_{\|\y\|\ge t} \varphi(t\sqrt L)(\x,\y)f(\y)dw(\y)\Big|\\
		&=: F_1+F_2+F_3.
	\end{aligned}
	\]
	
	Using Lemma \ref{lem-heat kernel estimate},
	\[
	\begin{aligned}
		F_1&\lesi \int_{\|\y\|<t}  \f{\|\y\|}{t}  \f{1}{w(B(0,t+\|\x\|))}\Big(\f{t }{t+\|\x\|}\Big)^\gamma P_\gamma(0,\y;1)dw(\y)\\
		&\lesi \int_{\|\y\|<t}  \f{\|\y\|^{1-\gamma}}{t}  \f{1}{w(B(0,t+\|\x\|))}\Big(\f{t }{t+\|\x\|}\Big)^\gamma \f{1}{w(B(0,1+\|\y\|))}dw(\y)\\
		&\lesi t^{-1}P_\gamma(0,\x;t) \int_{\|y\|\le t}\f{\|\y\|^{1-\gamma}}{w(B(0,1+\|\y\|)}dw(\y)\\
		&\lesi t^{-1}t^{1-\gamma}P_\gamma(0,\x;t)\\
		&\lesi t^{-\gamma'}P_\gamma(0,\x;t).
	\end{aligned}
	\]
	For the term $F_2$, since $\displaystyle \int f(\y)dw(\y)=0$, we have
	\[
	\begin{aligned}
		F_2&=\Big|\int_{\|\y\|\ge t} \varphi(t\sqrt L)(\x,0)f(\y)dw(\y)\Big|\\
		&\lesi \int_{\|\y\|\ge t}  \f{1}{w(B(0,t+\|\x\|))}\Big(\f{t }{t+\|\x\|}\Big)^\gamma P_\gamma(0,\y;1)dw(\y)\\
		&\lesi P_\gamma(0,\x;t) \int_{\|\y\|\ge t}   P_\gamma(0,\y;1)dw(\y)\\
		&\lesi t^{-\gamma}P_\gamma(0,\x;t).
	\end{aligned}
	\] 	
	For the term $F_3$, we have
	\[
	\begin{aligned}
		F_3&\lesi \int_{\|\y\|\ge t} \f{1}{w(B(\x,t+d(\x,\y)))}\Big(\f{t}{t+d(\x,\y)}\Big)^{\gamma}P_\gamma(0,\y;1)dw(\y).
	\end{aligned}
	\]
	Since  $t>1+\|\x\|$, $P_\gamma(0,\x;t)\simeq \f{1}{w(B(\x,t))}$. Hence,
	\[
	\begin{aligned}
		F_3&\lesi \int_{\|\y\|\ge t}\f{1}{w(B(\x,t)} P_\gamma(0,\y;1)dw(\y)\\
		&\lesi  P_\gamma(0,\x;t) \int_{\|\y\|\ge t} P_\gamma(0,\y;1)dw(\y)\\
		&\lesi t^{-\gamma}P_\gamma(0,\x;t).	
	\end{aligned}
	\]
	This completes our proof.
\end{proof}

We are now ready to establish Calder\'on reproducing formulae.

\begin{thm}\label{Calderon reproducing}
	Let $\epsilon \in (0,1]$ and $\beta,\gamma\in \mathscr{S}(0,\epsilon)$ and $\varphi\in \mathscr(\mathbb R)$ be an even function satisfying $\varphi(0)=0$ and $\displaystyle \int_0^\vc \varphi(t)\f{dt}{t}\ne 0$. Then we have
	\[
	f= c_\varphi\int_0^\vc \varphi(t\sqrt{L})f\f{dt}{t} 
	\]
	in $(\mathring{\mathcal G}^\epsilon(\beta,\gamma))'$, where $\displaystyle c_\varphi =\Big[\int_0^\vc \varphi(t)\f{dt}{t}\Big]^{-1}$.
\end{thm}
\begin{proof}
	By duality, it suffices to prove that 
	\[
	f= c_\varphi\int_0^\vc \varphi(t\sqrt{L})f\f{dt}{t} 
	\]
	in $\mathring{\mathcal G}^\epsilon(\beta,\gamma)$.

	For $f\in \mathring{\mathcal G}^\epsilon(\beta,\gamma)$, there exists a sequence $\{f_n\}_n\subset \mathring{\mathcal G}(\epsilon,\epsilon)$ such that $\lim_{n\to \vc}\|f_n-f\|_{\mathring{\mathcal G}(\beta,\gamma)}=0$. We now write
	\begin{equation}\label{eq1-proof of Calderon}
		c_\varphi\int_0^\vc \varphi(t\sqrt{L})f\f{dt}{t} =c_\varphi\int_0^\vc \varphi(t\sqrt{L})(f-f_n)\f{dt}{t} +c_\varphi\int_0^\vc \varphi(t\sqrt{L})f_n\f{dt}{t}.
	\end{equation}
	From Lemma \ref{lem - integral t small} and Lemma \ref{lem - t large for integral},
	\[
	\Big\|\int_0^\vc \varphi(t\sqrt{L})(f-f_n)\f{dt}{t}\Big\|_{\mathring{\mathcal G}(\beta,\gamma)}\lesi \|f-f_n\|_{\mathring{\mathcal G}(\beta,\gamma)},
	\]
	which implies
	\begin{equation}\label{eq2-proof of Calderon}
		\lim_{n\to \vc}c_\varphi\int_0^\vc \varphi(t\sqrt{L})(f-f_n)\f{dt}{t} = 0 \ \ \ \text{in \ $\mathring{\mathcal G}(\beta,\gamma)$}.
	\end{equation}
	From Lemma \ref{lem - t small} and Lemma \ref{lem - t large},  there exists $h_n\in  \mathring{\mathcal G}^\epsilon(\beta,\gamma)$ such that
	\[
	h_n= c_\varphi\int_0^\vc \varphi(t\sqrt{L})f_n\f{dt}{t}. 
	\]
	On the other hand, since  $f_n, h_n\in \mathring{\mathcal G}^\epsilon(\beta,\gamma)$, $f_n, h_n\in L^2(dw)$. By the spectral theory,
	\[
	f_n = c_\varphi\int_0^\vc \varphi(t\sqrt{L})f_n\f{dt}{t} \ \ \text{in $L^2(dw)$}.
	\]
	Consequently, $h_n=f_n$ for almost everywhere. This, along with the fact that $f_n,h_n\in \mathring{\mathcal G}(\beta,\gamma)$, implies that $h_n(\x)=f_n(\x)$ for all $x\in \RN$ and for each $n$. Therefore, 
	\begin{equation}\label{eq-convergence in Gbeta gamma}
		f_n= c_\varphi\int_0^\vc \varphi(t\sqrt{L})f_n\f{dt}{t} 
	\end{equation}
	in $\mathring{\mathcal G}(\beta,\gamma)$.
	
	This, along with \eqref{eq1-proof of Calderon} and \eqref{eq2-proof of Calderon}, yields
	\[
	c_\varphi\int_0^\vc \varphi(t\sqrt{L})f\f{dt}{t}=\lim_{n} f_n,
	\]
	and hence
	\[
	c_\varphi\int_0^\vc \varphi(t\sqrt{L})f\f{dt}{t}=f
	\]
	in $\mathring{\mathcal G}(\beta,\gamma)$ since $\lim_{n\to \vc}\|f_n-f\|_{\mathring{\mathcal G}(\beta,\gamma)}=0$.
	
	This completes our proof.	
\end{proof}

Arguing similarly to the proof of Theorem \ref{Calderon reproducing}, we  have the following discrete version of the Calder\'on reproducing formula. In what follows, by a ``partition of unity'' we shall mean a function $\psi\in \mathcal{S}(\mathbb{R})$ such that $\supp\psi\subset[1/2,2]$, $\displaystyle \int\psi(\xi)\,\f{d\xi}{\xi}\neq 0$ and
$$\sum_{j\in \mathbb{Z}}\psi_j(\lambda)=1 \textup{ on } (0,\infty), $$
where $\psi_j(\lambda):=\psi(\delta^j\lambda)$ for each $j\in \mathbb{Z}$.

\begin{thm}
	\label{thm - Calderon discrete version}
	Let $\epsilon \in (0,1]$ and $\beta,\gamma \in (0,\epsilon)$ and $\psi$ be a partition of unity. Then we have
	\[
	f= \sum_{j\in \mathbb Z}\psi_j(\sqrt L)f
	\]
	in $(\mathring{\mathcal G}^\epsilon(\beta,\gamma))'$.
\end{thm}

\section{Besov and Triebel-Lizorkin spaces in the Dunkl setting}
The main aim of this section is to prove Theorem \ref{thm-coincidence Besov and TL spaces}. To accomplish this, our first task is to establish the atomic decomposition of the Besov and Triebel-Lizorkin spaces on the space of homogeneous type $(\RN, |\cdot|, dw)$. Subsequently, we recapitulate some fundamental properties of the Besov and Triebel-Lizorkin spaces associated with the operator $L$, including the atomic decomposition and interpolation theorems. This is interesting to emphasize that in the particular case when the indices are similar to those in the Besov and Triebel-Lizorkin spaces on the space of homogeneous type $(\RN, \|\cdot\|, dw)$, we can use the same space of distributions to defined these function spaces. Finally, in this particular case, we establish a new atomic decomposition which will play a crucial role in proving Theorem \ref{thm-coincidence Besov and TL spaces}.

\subsection{Atomic decomposition of Besov and Triebel-Lizorkin spaces on $(\RN, \|\cdot\|, dw)$}

In this section, we will prove the atomic decomposition for the Besov and Triebel-Lizorkin spaces on the space of homogeneous type $(\RN, \|\cdot\|, dw)$. We first recall the definition of an approximation of the identity  with bounded support in \cite{HMY} which plays an essential role in the definitions of Besov and Triebel-Lizorkin spaces.

\begin{defn}\label{defn- Sk}
	A sequence $\{S_k\}_{k\in \Z}$  of bounded linear integral operators on
	$L^2(dw)$ is said to be an approximation of the identity  with bounded support (for short, ATI with bounded support) if there exist constants $A_1, A_2>0$ such that for all $k\in \Z$  the integral kernel $S_k(\cdot, \cdot)$ of $S_k$ is a measurable function from $\RN\times \RN$
	satisfying the following properties
	\begin{enumerate}[{\rm (i)}]
		\item $S_k(\x,\y)=S_k(\y,\x)$;
		\item $S_k(\x,\y)=0$ whenever $\|\x-\y\|\ge A_1 \delta^{k}$;
		\item $\displaystyle |S_k(\x,\y)|\lesi \f{1}{w(B(\x,\delta^k))}$;
		\item $\displaystyle |S_k(\x,\y)-S_k(\x',\y)|+|S_k(\y,\x)-S_k(\y,\x')|\lesi \f{\|\x-\x'\|}{\delta^k} \f{1}{w(B(\x,\delta^k))+w(B(\y,\delta^k))}$ for $\|\x-\x'\|\le A_2\delta^k$;
		\item $\displaystyle \big|[S_k(\x,\y)-S_k(\x',\y)]- [S_k(\x,\y')-S_k(\x',\y')]\big|\lesi \f{\|\x-\x'\|}{\delta^k}\f{\|\y-\y'\|}{\delta^k} \f{1}{w(B(\x,\delta^k))+w(B(\y,\delta^k))}$ for $\max\{\|\x-\x'\|,\|\y-\y'\|\}\le A_2\delta^k$;
		\item $\displaystyle \int S_k(\x,\y)dw(\x)=  \int S_k(\y,\x)dw(\x)=1$ for all $\y\in \RN$.
	\end{enumerate}
\end{defn}
The existence of the family $\{S_k\}$ is guaranteed in Theorem 2.6 \cite{HMY} (see also \cite{THHLL}).

In what follows, for $s\in (-1,1)$ and $\epsilon\in (0,1]$, we set
\begin{equation}\label{eq- p s epsilon}
	p(s,\epsilon)=\max\Big\{\f{\N}{\N +\epsilon},\f{\N}{\N +s+\epsilon}\Big\}.
\end{equation}

We now recall the definitions of  the Besov and Triebel-Lizorkin spaces in \cite{HMY}.

\begin{defn}
	\label{defn-Besov and TL spaces}
	Let $S_k$ be an ATI with bounded support and let $D_k =S_k-S_{k-1}$ for $k\in \Z$.  For $s\in (-1,1)$, $p(s,1)<p< \vc$ and $0<q< \vc$, assume that $\epsilon \in (0,1)$ and $0<\beta,\gamma<\epsilon$ such that $|s|<\beta \wedge \gamma$ and $p>p(s,\epsilon)$.  The Besov space $\dot{B}^s_{p,q}(dw)$ is defined as the set of all $f\in (\mathring{\mathcal G}^\epsilon_0(\beta,\gamma))'$  such that 
	\begin{equation}
		\label{eq-Besov norm}
		\|f\|_{\dot{B}^s_{p,q}(dw)} = \Big[\sum_{k\in \Z}\delta^{-ksq}\|D_k f\|^q_{L^p(dw)}\Big]^{1/q}<\vc.
	\end{equation}
	For $s\in (-1,1)$ and $p(s,1)<p, q< \vc$, assume that $\epsilon \in (0,1)$ and $0<\beta,\gamma<\epsilon$ such that $|s|<\beta \wedge \gamma$ and $p\wedge q >p(s,\epsilon)$. The Triebel-Lizorkin space $\dot{F}^s_{p,q}(dw)$ is defined as the set of all $f\in (\mathring{\mathcal G}^\epsilon_0(\beta,\gamma))'$ such that
	\begin{equation}
		\label{eq-TL norm}
		\|f\|_{\dot{F}^s_{p,q}(dw)} = \Big\|\Big(\sum_{k\in \Z}\delta^{-ksq}\|D_k f\|^q\Big)^{1/q} \Big\|_{L^p(dw)}<\vc.
	\end{equation}

\end{defn}

We would like to highlight that the Besov and Triebel-Lizorkin spaces on a space of homogeneous type have been defined for infinite indices $p=\infty$ and/or $q=\infty$, as mentioned in \cite{HMY}. Nevertheless, when considering this scenario, it remains uncertain whether the function spaces are independent of  the choices of $\epsilon$. As a result, we have chosen to exclude this particular case from our theory.

%With $\delta$ as in Lemma \ref{lem-dyadic cube}, by Lemma \ref{lem-dyadic cube},  there exist a set of points $\{x_\alpha^k: k\in \Z, \alpha\in I_k\}$ and a family of set $\{Q_\alpha^k: k\in \Z, \alpha\in I_k\}$ satisfying
%\begin{enumerate}[{\rm (i)}]
%	\item $\|x_\alpha^k-x_\beta^k\|\ge \delta^k$ for all $k\in Z$ and $\alpha,\beta \in I_k$ with $\alpha\ne \beta$;
%	\item $\min_{\alpha\in I_k}\|x-z_\alpha^k\|\le \delta^k$;
%	\item for any $k\in \Z$, $\bigcup_{\alpha\in I_k} Q^k_\alpha =\RN$ and $\{Q_\alpha^k:  \alpha\in I_k\}$ is disjoint;
%	\item if $k,\ell\in \Z$ and $k\ge \ell$, then either $Q_\alpha^k\subset Q_\beta^\ell$ or $Q_\alpha^k\cap  Q_\beta^\ell=\emptyset$ for every $\alpha\in I_k$ and $\beta\in I_\ell$;
%	\item for any $k\in Z$ and $\alpha\in I_k$, $B(x_\alpha^k,\delta^k/6)\subset Q_\alpha^k\subset B(x_\alpha^k,2\delta^k)$.
%\end{enumerate}

Later on in Corollary \ref{cor2} below, by establishing the atomic decomposition for these function spaces, we will show that the definitions of the Besov and Triebel-Lizorkin spaces defined in Definition \ref{defn-Besov and TL spaces} are independent on the parameters $\epsilon, \beta$ and $\gamma$. In order to do this, we recall  from Section 2.1 that  $\mathscr D:= \{Q_\alpha^k: k\in \Z, \alpha\in I_k\}$ is the system of dyadic cubes in $(\RN, \|\cdot\|,dw)$ and  $\mathscr D_k =\{Q_\alpha^k: \alpha\in I_k\}$ for each $k\in \Z$. In what follows, we always choose a $j_0\in \mathbb N$ to be sufficiently large fixed number (such that the identity in Lemma \ref{lem-Calderon reproducing lemma} below holds true). For each $k\in \Z$ and $\alpha\in I_k$, from  Lemma \ref{lem-dyadic cube}, we can write
\begin{equation}\label{eq-dyadic cube 2}
	Q^k_\alpha =\bigcup_{\tau\in I_{k+j_0}: \ Q_\tau^{k+j_0}\subset Q^k_\alpha}Q_\tau^{k+j_0}.
\end{equation}
We also denote $N(\alpha,k)=\sharp\{\tau\in I_{k+j_0}: \ Q_\tau^{k+j_0}\subset Q^k_\alpha\}$. By the doubling condition \eqref{eq-doubling} and Lemma \ref{lem-dyadic cube}, $N(\alpha,k)\lesi 2^{j_0\mathfrak{N}}$. We can rearrange to write $\{Q_\tau^{k+j_0}\}_{\tau \in N(\alpha,k)}$ as $\{Q_{\alpha,m}^{k} \}_{m=1}^{N(\alpha,k)}$.

We have the following Calder\'on-Zygmund formula in \cite[Theorem 4.12]{HMY} (see also \cite{THHLL}).

\begin{lem}
	\label{lem-Calderon reproducing lemma}
	Let $S_k$ be an ATI with bounded support and $D_k = S_k-S_{k-1}$, and let $\epsilon\in (0,1)$ and $0<\beta,\gamma<\epsilon$. Then there exists a sequence $\{\widetilde D_k\}_{k\in \Z}$ of bounded linear operators on $L^2(\RN)$ such that
	\[
	f(\x) =\sum_{k\in \Z} \sum_{\alpha\in I_k}\sum_{m=1}^{N(k,\alpha)} w(Q_{\alpha,m}^k)D_k(\x,\y_{\alpha,m}^k) \widetilde D_k f(\y_{\alpha,m}^k)
	\]
	in both $L^p(dw), 1<p<\vc$ and $(\mathring{\mathcal G}^\epsilon_0(\beta,\gamma))'$ for any $\y_{\alpha,m}^k\in Q_{\alpha, m}^k$; moreover, the kernels $\widetilde D_k(\x,\y)$ of $\widetilde D_k$ satisfy the following conditions
	\begin{enumerate}[{\rm (i)}]
		\item $\displaystyle |\widetilde D_k(\x,\y)|\lesi \f{1}{w(B(\x,\delta^k+\|\x-\y\|))}\f{\delta^k}{\delta^k+\|\x-\y\|}$;
		\item $\displaystyle  |\widetilde D_k(\x,\y)-\widetilde D_k(\x,\y')|\lesi \f{\|\y-\y'\|}{\delta^k+\|\x-\y\|} \f{1}{w(B(x,\delta^k+\|\x-\y\|))}\f{\delta^k}{\delta^k+\|\x-\y\|}$ for $\|\y-\y'\|\le \f{1}{2}(\delta^k+\|\x-\y\|)$;
		\item $\displaystyle \int \widetilde D_k(\x,\y)dw(\x)=  \int \widetilde D_k(\y,\x)dw(\x)=0$ for all $y\in \RN$.
	\end{enumerate}
\end{lem}

\begin{defn}[\cite{HY}]\label{def:7.1}
	Let $ 0<p \leq \infty$. A function $a$ is said to be a   $p$-atom if there exists a dyadic cube $Q \in \mathscr{D}_{k}$ for some $k\in \Z$ such that \\
	\begin{enumerate}[\rm (i)]
		\item $\operatorname{supp}\,a \subset \Lambda Q$, where $\Lambda= 3\times 2^{j_0}A_1$ (see Section 2.1 for the definition of $\Lambda Q$);
		\item for any $\x \in \RN$,
		$$\displaystyle |a(\x)| \leq w(Q)^{-1/p};$$
		
		\item for any $\x,\x'\in\RN$,
		$$|a(\x)-a(\x')| \leq w(Q)^{-1/p} \dfrac{\|\x-\x'\|}{\delta^k} ;
		$$
		\item $\displaystyle \int a(\x)dw(\x)=0$.
	\end{enumerate}
\end{defn}
We have the following technical lemma.
\begin{lem}\label{lem1-comparison}
	Let $S_k$ be an ATI with bounded support and $D_k = S_k-S_{k-1}$. Then for any $M>0$ there exists $C>0$ such that for every $k\in \Z$ and every  $p$-atom $ a_Q$ associated to some dyadic cube $Q \in \mathscr{D}_\nu, \nu \in \mathbb{Z}$,
	\begin{enumerate}[{\rm (i)}]
		\item $\displaystyle | D_k{a_Q}(\x) | \le C w( Q)^{ - 1/p}\delta^{k-\nu}{\left( {1 + \frac{\|\x-\x_Q\|}{\delta^{ \nu }}} \right)^{ - M}}$ for all $ \nu \leq k$;
		\item $\displaystyle  | D_k{a_Q}(\x) | \le Cw( Q)^{ - 1/p}\delta^{-k+\nu} \frac{{w\left( Q \right)}}{{w\left( B({{\x_Q},\delta^k}) \right)}}{\left( {1 + \frac{\|\x-\x_Q\|}{\delta^{  k }}} \right)^{ - M}}$ for all $ k \leq \nu$.
	\end{enumerate}
\end{lem}

\begin{proof}
	Due to the bounded support condition of $D_k(\x,\y)$, it suffices to prove the lemma with $\x\in 3\Lambda Q$. In this case,
	\[
	\left( {1 + \frac{\|\x-\x_Q\|}{\delta^{  k\wedge \nu }}} \right)^{ - M}\simeq 1, \ \ \ \ \x\in 3\Lambda Q.
	\]
	(i) Since $\displaystyle \int D_k(\x,\y)dw(\y)=0$, we have
	\begin{align*}
		\left| D_ka_Q(x)\right|&=  \Big|\int_{\RN} D_k(\x,\y) a_Q(\y) dw(\y)\Big|\\
		&=	 \Big|\int_{\RN} D_k(\x,\y) [a_Q(\y)-a_Q(\x)] dw(\y)\Big|\\
		&\lesi w(Q)^{-1/p}\int_{\RN} |D_k(\x,\y)| \f{\|\x-\y\|}{\ell(Q)} dw(\y).
	\end{align*}
	Due to the bounded support condition (ii) in Definition \ref{defn- Sk}, we further obtain
	\begin{align*}
		\left| D_ka_Q(x)\right|&=  \Big|\int_{\RN} D_k(\x,\y) a_Q(\y) dw(\y)\Big|\\
		&\lesi w(Q)^{-1/p}\f{\delta^k}{\ell(Q)} \int_{\RN} |D_k(\x,\y)| dw(\y)\\
		&\lesi w(Q)^{-1/p}\f{\delta^k}{\ell(Q)}\simeq w(Q)^{-1/p}\delta^{k-\nu}.
	\end{align*}
	
	This completes the proof of (i).
	
	\medskip
	
	(ii) Since $\displaystyle \int a_Q(\x)dw(\x)=0$, we have
	\begin{align*}
		\left| D_ka_Q(x)\right|&=  \Big|\int_{\Lambda Q} [D_k(\x,\y)-D_k(\x,\x_Q)] a_Q(\y) dw(\y)\Big|\\
		&\lesi	 \int_{\Lambda Q}  \f{\|\y-\x_Q\|}{\delta^k} \f{1}{w(B(\y,\delta^k))}|a_Q(\y)| dw(\y) \\
		&\lesi w(Q)^{-1/p} \f{\ell(Q)}{\delta^k}\f{w(Q)}{w(B(x_Q,\delta^k))}\\
		&\lesi w(Q)^{-1/p} \delta^{\nu-k}\f{w(Q)}{w(B(x_Q,\delta^k))}.
	\end{align*}

	This completes the proof of (ii).
\end{proof}
We have the following atomic decomposition for Besov and Triebel-Lizorkin spaces.
\begin{thm}\label{thm1-Besov atomic}
	Let $s\in (-1,1)$ and $p(s,1)<p<\vc$, $0<q<\vc$ and let $\epsilon\in (0,1)$ and $\beta, \gamma \in (0,\epsilon)$ satisfying  $|s|<\beta\wedge \gamma$ and $p>p(s,\epsilon)$. If $f\in \dot{B}^{s}_{p,q}(dw)$, then there exist a sequence of $p$-atoms $\{a_{Q}: Q\in \mathscr D\}$ and a sequence of numbers $\{s_{Q}: Q\in \mathscr D\}$ such that
	$$f= \sum\limits_{k  \in \mathbb{Z}} {\sum\limits_{Q \in \mathscr D_k }} {{s_{Q}{a_{Q}}} } \text{ in }  (\mathring{\mathcal G}^\epsilon_0(\beta,\gamma))',$$ 
	and
	$$
	{\Big[ {\sum\limits_{k  \in \mathbb{Z}} {{\delta^{-k s q}}{{\Big( {\sum\limits_{Q\in \mathscr D_k}} {{{| {{s_{Q}}|}^p}} } \Big)}^{q/p}}} } \Big]^{1/q}}\lesi \|f\|_{\dot{B}^{s}_{p,q}(dw)}.$$
	
	Conversely, if there exist a sequence of $p$-atoms $\{a_{Q}: Q\in \mathscr D\}$ and a sequence of numbers $\{s_{Q}: Q \in \mathscr D\}$ such that
	$$f= \sum\limits_{k  \in \mathbb{Z}} {\sum\limits_{Q \in \mathscr D_k }} {{s_{Q}{a_{Q}}} } \text{ in }  (\mathring{\mathcal G}^\epsilon_0(\beta,\gamma))',$$ 
	and
	$$
	{\Big[ {\sum\limits_{k  \in \mathbb{Z}} {{\delta^{-k s q}}{{\Big( {\sum\limits_{Q\in \mathscr D_k}} {{{| {{s_{Q}}|}^p}} } \Big)}^{q/p}}} } \Big]^{1/q}}<\vc,$$
	then $f\in \dot{B}^{s}_{p,q}(dw)$ and
	\[
	\|f\|_{\dot{B}^{s}_{p,q}(dw)}\lesi {\Big[ {\sum\limits_{k  \in \mathbb{Z}} {{\delta^{-k s q}}{{\Big( {\sum\limits_{Q\in \mathscr D_k}} {{{| {{s_{Q}}|}^p}} } \Big)}^{q/p}}} } \Big]^{1/q}}.
	\]
\end{thm}
\begin{thm}\label{thm1-TL atomic}
	Let $s\in (-1,1)$ and $p(s,1)<p,q<\vc$ and let $\epsilon\in (0,1)$ and $\beta, \gamma \in (0,\epsilon)$ satisfying  $|s|<\beta\wedge \gamma$ and $p\wedge q>p(s,\epsilon)$.  If $f\in \dot{F}^{s}_{p,q}(dw)$, then there exist a sequence of $p$-atoms $\{a_{Q}: Q\in \mathscr D\}$ and a sequence of numbers $\{s_{Q}: Q\in \mathscr D_k\}$ such that
	$$f= \sum\limits_{k  \in \mathbb{Z}} {\sum\limits_{Q \in \mathscr D_k }} {{s_{Q}{a_{Q}}} } \text{ in }  (\mathring{\mathcal G}^\epsilon_0(\beta,\gamma))',$$ 
	and
	$${\Big\| {{{\Big[ {\sum\limits_{k  \in \mathbb{Z}} {{\delta^{-k s q}}{{  {\sum\limits_{Q \in {\mathscr{D}_k }} {w{{(Q)}^{ - q/p}}| {{s_Q}}|^q{\chi _Q}} }  }}} } \Big]}^{1/q}}} \Big\|_{L^p(dw)}}\lesi \|f\|_{\dot{F}^{s}_{p,q}(dw)}.$$
	
	Conversely, if there exist a sequence of $p$-atoms $\{a_{Q}: Q\in \mathscr D\}$ and a sequence of numbers $\{s_{Q}: Q\in \mathscr D\}$ such that
	$$f= \sum\limits_{k  \in \mathbb{Z}} {\sum\limits_{Q \in \mathscr D_k }} {{s_{Q}{a_{Q}}} } \text{ in }  (\mathring{\mathcal G}^\epsilon_0(\beta,\gamma))',$$ 
	and
	$$
	{\Big\| {{{\Big[ {\sum\limits_{k  \in \mathbb{Z}} {{\delta^{-k s q}}{{  {\sum\limits_{Q \in {\mathscr{D}_k }} {w{{(Q)}^{ - q/p}}| {{s_Q}}|^q{\chi _Q}} }  }}} } \Big]}^{1/q}}} \Big\|_{L^p(dw)}}<\vc,$$
	then $f\in \dot{F}^{s}_{p,q}(dw)$ and
	\[
	\|f\|_{\dot{F}^{s}_{p,q}(dw)}\lesi {\Big\| {{{\Big[ {\sum\limits_{k  \in \mathbb{Z}} {{\delta^{-k s q}}{{  {\sum\limits_{Q \in {\mathscr{D}_k }} {w{{(Q)}^{ - q/p}}| {{s_Q}}|^q{\chi _Q}} }  }}} } \Big]}^{1/q}}} \Big\|_{L^p(dw)}}.
	\]
\end{thm}

Although the proofs of Theorem \ref{thm1-Besov atomic} and Theorem \ref{thm1-TL atomic} are standard and could have been addressed elsewhere, we have not come across them in the existing literature. Therefore, we aim to present the proofs here for the sake of completeness.  Since the proofs of Theorem \ref{thm1-Besov atomic} and Theorem \ref{thm1-TL atomic} are similar, we just give the proof of Theorem \ref{thm1-TL atomic} as an illustrative example.
\begin{proof}[Proof of Theorem \ref{thm1-TL atomic}:]
	Let $S_k$ be an ATI with bounded support and $D_k = S_k-S_{k-1}$. From Lemma \ref{lem-Calderon reproducing lemma},
	\[
	\begin{aligned}
		f(\x) &=\sum_{k\in \Z} \sum_{\alpha\in I_k}\sum_{m=1}^{N(k,\alpha)} w(Q_{\alpha,m}^k)D_k(\x,\y_{\alpha,m}^k) \widetilde D_k f(\y_{\alpha,m}^k)
	\end{aligned}
	\]
	in $(\mathring{\mathcal G}^\epsilon_0(\beta,\gamma))'$, where $\widetilde D_k$ satisfies (i)-(iii) in Lemma \ref{lem-Calderon reproducing lemma}.

	Define
	\[
	a_{Q_{\alpha,m}^k}(\x)= \f{w({Q_{\alpha,m}^k})}{s_{Q_{\alpha,m}^k}}D_k(\x,\y_{\alpha,m}^k) \widetilde D_k f(\y_{\alpha,m}^k)
	\]
	and
	$$
	s_{Q_{\alpha,m}^k}=w({Q_{\alpha,m}^k})^{1/p}|\widetilde D_k f(\y_{\alpha,m}^k)|.
	$$
	From the properties (ii), (iii), (iv) and (v) in Definition \ref{defn- Sk}, it is easy to see that $a_{{Q_{\alpha,m}^k}}$ is an $p$-atom associated to the cube $Q_{\alpha,m}^k$. 
	
	This, together with \eqref{eq-dyadic cube 2}, implies
	\[
	\begin{aligned}
		\Big\| \Big[\sum_{k\in \Z}\delta^{-ksq} \sum_{\alpha\in I_k}&w(Q_{\alpha,m}^k)^{-q/p}|s_{Q_{\alpha,m}^k}|^q\chi_{Q_{\alpha,m}^k}\Big]^{1/q}\Big\|_{L^p(dw)}\\
		&= \Big\| \Big[\sum_{k\in \Z}\delta^{-(k+j_0)sq} \sum_{\alpha\in I_{k+j_0}}w(Q_\alpha^{k+j_0})^{-q/p}|s_{Q_\alpha^{k+j_0}}|^q\chi_{Q_\alpha^{k+j_0}}\Big]^{1/q}\Big\|_{L^p(dw)}\\
		&\simeq \Big\| \Big[\sum_{k\in \Z}\delta^{-ksq} \sum_{\alpha\in I_{k+j_0}}w(Q_\alpha^{k+j_0})^{-q/p}|s_{Q_\alpha^{k+j_0}}|^q\chi_{Q_\alpha^{k+j_0}}\Big]^{1/q}\Big\|_{L^p(dw)}\\
		&=\Big\| \Big[\sum_{k\in \Z}\delta^{-ksq} \sum_{\alpha\in I_{k}}\sum_{m=1}^{N(\alpha,k)}w(Q_{\alpha,m}^{k})^{-q/p}|s_{Q_{\alpha,m}^{k}}|^q\chi_{Q_{\alpha,m}^{k}}\Big]^{1/q}\Big\|_{L^p(dw)}.
	\end{aligned}
	\]
	On the other hand,
	\[
	\begin{aligned}
		\Big\| \Big[\sum_{k\in \Z}\delta^{-ksq} \sum_{\alpha\in I_{k}}&\sum_{m=1}^{N(\alpha,k)}w(Q_{\alpha,m}^{k})^{-q/p}|s_{Q_{\alpha,m}^{k}}|^q\chi_{Q_{\alpha,m}^{k}}\Big]^{1/q}\Big\|_{L^p(dw)}\\
		&\lesi \Big\| \Big[\sum_{k\in \Z}\delta^{-ksq} \sum_{\alpha\in I_{k}}\sum_{m=1}^{N(\alpha,k)} |\widetilde D_k f(\y_{\alpha,m}^k)|^q\chi_{Q_{\alpha,m}^{k}}\Big]^{1/q}\Big\|_{L^p(dw)}.
	\end{aligned}
	\]
	It follows that 
	\[
	\begin{aligned}
		\Big\| \Big[\sum_{k\in \Z}\delta^{-ksq} \sum_{\alpha\in I_k}&w(Q_\alpha^k)^{-q/p}|s_{Q_\alpha^k}|^q\chi_{Q_\alpha^k}\Big]^{1/q}\Big\|_{L^p(dw)}\\
		&\lesi \Big\| \Big[\sum_{k\in \Z}\delta^{-ksq} \sum_{\alpha\in I_{k}}\sum_{m=1}^{N(\alpha,k)} |\widetilde D_k f(\y_{\alpha,m}^k)|^q\chi_{Q_{\alpha,m}^{k}}\Big]^{1/q}\Big\|_{L^p(dw)}.
	\end{aligned}
	\]
	By Remark 5.5 in \cite{HMY},
	\[
	\begin{aligned}
		\Big\| \Big[\sum_{k\in \Z}\delta^{-ksq} \sum_{\alpha\in I_{k}}\sum_{m=1}^{N(\alpha,k)} &|\widetilde D_k f(\y_{\alpha,m}^k)|^q\chi_{Q_{\alpha,m}^{k}}\Big]^{1/q}\Big\|_{L^p(dw)}\\
		&\lesi \Big\| \Big[\sum_{k\in \Z}\delta^{-ksq} \sum_{\alpha\in I_{k}}\sum_{m=1}^{N(\alpha,k)} \inf_{z\in Q_{\alpha,m}^{k}}| D_k f(z)|^q\chi_{Q_{\alpha,m}^{k}}\Big]^{1/q}\Big\|_{L^p(dw)}\\
		&\lesi \Big\| \Big[\sum_{k\in \Z}\delta^{-ksq} \sum_{\alpha\in I_{k}}\sum_{m=1}^{N(\alpha,k)} | D_k f|^q\chi_{Q_{\alpha,m}^{k}}\Big]^{1/q}\Big\|_{L^p(dw)}\\
		&\lesi \Big\| \Big[\sum_{k\in \Z}\delta^{-ksq} \sum_{\alpha\in I_{k}}  | D_k f|^q\chi_{Q_{\alpha}^{k}}\Big]^{1/q}\Big\|_{L^p(dw)}\\
		&\lesi \|f\|_{\dot{F}^s_{p,q}(dw)}.
	\end{aligned}
	\]
	Consequently,
	\[
	\Big\| \Big[\sum_{k\in \Z}\delta^{-ksq} \sum_{\alpha\in I_k}w(Q_\alpha^k)^{-q/p}|s_{Q_\alpha^k}|^q\chi_{Q_\alpha^k}\Big]^{1/q}\Big\|_{L^p(dw)}\lesi \|f\|_{\dot{F}^s_{p,q}(dw)}.
	\]
	This proves the first direction.

	The reverse direction is quite standard and follows directly from Lemma \ref{lem1-comparison}. For the sake of completeness, we provide the proof.
	
	Fix  $r<\min\{1,p,q\}$ so that $r>p(s,1)$ we also fix $M>\f{\N}{r}$.  Since 
	$$
	f=\sum_{\nu\in\mathbb{Z}}\sum_{Q\in \mathscr{D}_\nu}s_Qa_Q \ \ \text{in $(\mathring{\mathcal G}^\epsilon_0(\beta,\gamma))'$},
	$$
	we have, for each $k\in \mathbb{Z}$,
	\[
	\begin{aligned}
		D_kf&=\sum_{\nu\in\mathbb{Z}}\sum_{Q\in \mathscr{D}_\nu}s_QD_ka_Q\\
		&=\sum_{\nu: \nu\geq k}\sum_{Q\in \mathscr{D}_\nu}s_QD_ka_Q+\sum_{\nu: \nu< k}\sum_{Q\in \mathscr{D}_\nu}s_QD_ka_Q.
	\end{aligned}
	\]
	By Lemma \ref{lem1-comparison},
	\[
	\begin{aligned}
		\delta^{-sk}|D_kf|
		&\lesi \sum_{\nu: \nu\geq k}\sum_{Q\in \mathscr{D}_\nu}\delta^{-sk}\delta^{\nu-k}\f{w(Q)}{w(B(\x_{Q},\delta^k))} \Big(1+\f{d(\x,\x_{Q})}{\delta^{k}}\Big)^{-M}w(Q)^{-1/p}|s_Q|\\
		& \ \ +\sum_{\nu: \nu< k}\sum_{Q\in \mathscr{D}_\nu}\delta^{-sk}\delta^{k-\nu} \Big(1+\f{d(\x,\x_{Q})}{\delta^{\nu}}\Big)^{-M}w(Q)^{-1/p}|s_Q|\\
		&\lesi \sum_{\nu: \nu\geq k}\sum_{Q\in \mathscr{D}_\nu}\delta^{(\nu-k)(1+s)}\f{w(Q)}{w(B(\x_{Q},\delta^k))} \Big(1+\f{d(\x,\x_{Q})}{\delta^{k}}\Big)^{-M}\delta^{-s\nu}w(Q)^{-1/p}|s_Q|\\
		& \ \ +\sum_{\nu: \nu< k}\sum_{Q\in \mathscr{D}_\nu}\delta^{(k-\nu)(1-s)} \Big(1+\f{d(\x,\x_{Q})}{\delta^{\nu}}\Big)^{-M}\delta^{-s\nu}w(Q)^{-1/p}|s_Q|.
	\end{aligned}
	\]
	
	Using Lemma \ref{lem1- thm2 atom Besov}, 
	\begin{equation}\label{eq-proof of reverse Besov}
		\begin{aligned}
			\delta^{-sk}|D_kf|&\lesi \sum_{\nu: \nu\geq k}\delta^{(\nu-k)(1+s)}\delta^{-\N(\nu-k)(1/r-1)} \mathcal{M}_{r}\Big(\sum_{Q\in \mathscr{D}_\nu}\delta^{-s\nu}|s_Q|w(Q)^{-1/p}\chi_Q\Big)\\
			& \ \ \ \ +\sum_{\nu: \nu< k} \delta^{(k-\nu)(1-s)} \mathcal{M}_{r}\Big(\sum_{Q\in \mathscr{D}_\nu}\delta^{-s\nu}|s_Q|w(Q)^{-1/p}\chi_Q\Big).
		\end{aligned}
	\end{equation}
	
	By using \eqref{YFSIn} we conclude that
	\[
	\begin{aligned}
		\|f\|_{\FF^{s,L}_{p,q}(dw)}&=\Big\|\Big[\sum_{k\in \mathbb{Z}}(\delta^{-ks}|D_kf|)^{q}\Big]^{1/q}\Big\|_{L^p(dw)}\\
		&\lesi \Big\|\Big[\sum_{\nu\in\mathbb{Z}}\delta^{-\nu s q}\Big(\sum_{Q\in \mathscr{D}_\nu}w(Q)^{-1/p}|s_Q|\chi_Q\Big)^q\Big]^{1/q}\Big\|_{L^p(w)},
	\end{aligned}
	\]
	provided that $p,q >r>\f{\N}{\N+1+s}$.

	This completes our proof.
\end{proof}

The atomic decomposition results above imply that $\B^{s}_{p,q}(dw)\cap L^2(dw)$ and $\FF^{s}_{p,q}(dw)\cap L^2(dw)$ are  dense in $\B^{s}_{p,q}(dw)$ and $\FF^{s}_{p,q}(dw)$. Consequently, we have the following corollary.
\begin{cor}\label{cor2}
	For $s\in (-1,1)$, $p(s,1)<p< \vc$ and $0<q< \vc$, the Besov space $\dot{B}^s_{p,q}(dw)$ defined in Definition \ref{defn-Besov and TL spaces} is independent of the choices of  $\epsilon \in (0,1)$ and $0<\beta,\gamma<\epsilon$ satisfying $|s|<\beta \wedge \gamma$ and $p>p(s,\epsilon)$.
	
	Similarly, for $s\in (-1,1)$ and $p(s,1)<p, q< \vc$, the Triebel-Lizorkin space $\dot{F}^s_{p,q}(dw)$ defined in Definition \ref{defn-Besov and TL spaces} is independent of the choices of $\epsilon \in (0,1)$ and $0<\beta,\gamma<\epsilon$ satisfying $|s|<\beta \wedge \gamma$ and $p\wedge q >p(s,\epsilon)$.
\end{cor}

\subsection{Besov and Triebel-Lizorkin spaces associated to the Dunkl operator }

In this section, we will study the theory of Besov and Triebel-Lizorkin spaces associated to the Dunkl operator. It is interesting to note that we will show that these Besov and Triebel-Lizorkin spaces coincide with the Besov and Triebel-Lizorkin spaces defined in Definition \ref{defn-Besov and TL spaces}. 
\begin{defn}\label{defn 1}
	Let $\psi$ be  a partition of unity. Let $\epsilon\in (0,1]$ and $\beta,\gamma\in (0,\epsilon)$.  For $s\in \mathbb R$ and $p, q\in (0,\vc)$, we define the homogeneous Besov space $\B^{s,  L}_{p,q}(dw)$ as the completion of the set
	\[
	\Big\{f\in (\mathring{\mathcal G}^\epsilon(\beta,\gamma))': \Big\{\sum_{j\in \mathbb{Z}}\left(\delta^{-js}\|\psi_j(\sqrt{L})f\|_{L^p(dw)}\right)^q\Big\}^{1/q}<\vc\Big\}
	\]
	under the norm 
	\[
	\|f\|_{\B^{s, L}_{p,q}(dw)}:= \Big\{\sum_{j\in \mathbb{Z}}\left(\delta^{-js}\|\psi_j(\sqrt{L})f\|_{L^p(dw)}\right)^q\Big\}^{1/q}<\vc.
	\]
	
	Similarly, $s\in \mathbb R$ and $p, q\in (0,\vc)$, we define the homogeneous Triebel-Lizorkin space $\FF^{s, L}_{p,q}(dw)$ as the completion of the set
	\[
	\Big\{f\in (\mathring{\mathcal G}^\epsilon(\beta,\gamma))': \Big\|\Big[\sum_{j\in \mathbb{Z}}(\delta^{-js}|\psi_j(\sqrt{L})f|)^q\Big]^{1/q}\Big\|_{L^p(dw)}<\vc\Big\}
	\]
	under the norm 
	\[
	\|f\|_{\FF^{s, L}_{p,q}(dw)}:= \Big\|\Big[\sum_{j\in \mathbb{Z}}(\delta^{-js}|\psi_j(\sqrt{L})f|)^q\Big]^{1/q}\Big\|_{L^p(dw)}.
	\]
\end{defn}
We have some relevant comments on the definitions of these functions spaces. 
\begin{enumerate}[{\rm (a)}]
	\item As shown in \cite{BBD}, the function spaces in Definition \ref{defn 1}  are independence of the partition of unity function $\psi$. Later on, we will show that these function spaces in Definition \ref{defn 1} are also independent of $\epsilon,\beta,\gamma$. See Corollary \ref{cor1}. This serves as a rationale for omitting these indices from the notations of the function spaces without any potential confusion.
	
	\item We have to defined our function spaces via the completion approach since the space of distributions $(\mathring{\mathcal G}^\epsilon(\beta,\gamma))'$ might not be large enough for the function spaces as $p,q$ close to $0$ and $|s|\ge 1$. Note that in \cite{BBD}, the Besov and Triebel-Lizorkin spaces are defined directly on a new space of distributions associated with the operator $L$ rather than by the completion as in Definition \ref{defn 1}. However, the abstract nature and dependency on the main operator of the space of distributions in \cite{BBD} render it undesirable. 
	
	In Definition \ref{defn 1}, we specially focus on the space of distributions $(\mathring{\mathcal G}^\epsilon(\beta,\gamma))'$ for several reasons.  Firstly, the distributions $(\mathring{\mathcal G}^\epsilon(\beta,\gamma))'$ are more natural and have been extensively utilized in the examination of function spaces in homogeneous spaces. See for example \cite{HMY, HWY, WHH}. Secondly, the ultimate goal in this section is to establish the equivalence between the Besov and Triebel-Lizorkin spaces associated with the Dunkl operator and their counterparts on the space homogeneous type, as defined in Definition \ref{defn-Besov and TL spaces}.  With the ranges of the indices specified in Definition \ref{defn-Besov and TL spaces}, the space of distributions $(\mathring{\mathcal G}^\epsilon(\beta,\gamma))'$ proves to be sufficient spacious as demonstrated in Theorem \ref{thm 1}).
\end{enumerate}
We note that since $L$ is a non-negative self-adjoint operator satisfying the Gaussian upper bound in $(\RN, d, dw)$,  properties of Besov and Triebel-Lizorkin spaces established in \cite{BBD} can be transferred to the Dunkl setting.  We now recall some properties in \cite{BBD}.

\bigskip
For $\lambda>0, j\in \mathbb{Z}$ and $\varphi\in \mathscr{S}(\mathbb{R})$ the Peetre's type maximal function is defined, for $f\in \mathcal (\mathring{\mathcal G}^\epsilon(\beta,\gamma))'$ with $\epsilon\in (0,1]$ and $\beta,\gamma\in (0,\epsilon)$, by
\begin{equation}
	\label{eq-PetreeFunction}
	\varphi_{j,\lambda}^*(\sqrt{L})f(\x)=\sup_{\y\in \RN}\f{|\varphi_j(\sqrt{L})f(\y)|}{(1+2^jd(\x,\y))^\lambda} \ \  , \x\in \RN,
\end{equation}
where $\varphi_j(\lambda)=\varphi(2^{-j}\lambda)$.

Obviously, we have
\[
\varphi_{j,\lambda}^*(\sqrt{L})f(\x)\geq |\varphi_j(\sqrt{L})f(\x)|, \ \ \ \ \x\in \RN.
\]
Similarly, for $s, \lambda>0$ and $f\in \mathcal (\mathring{\mathcal G}^\epsilon(\beta,\gamma))'$ with $\epsilon\in (0,1]$ and $\beta,\gamma\in (0,\epsilon)$ we set
\begin{equation}
	\label{eq2-PetreeFunction}
	\varphi_{\lambda}^*(s\sqrt{L})f(\x)=\sup_{\y\in \RN}\f{|\varphi(s\sqrt{L})f(\y)|}{(1+d(\x,\y)/s)^\lambda}.
\end{equation}
Then from \cite[Proposition 3.3]{BBD}, we have
\begin{prop}
	\label{prop - equivalence of varphi*}
	Let $\psi$ be a partition of unity and $\varphi\in \mathscr{S}(\mathbb R)$ be an even function with ${\rm supp}\, \varphi\subset [-c,-b]\cup[b,c]$ for some $c>b>0$. Then we have:
	\begin{enumerate}[{\rm (a)}]
		\item For $0< p, q< \vc$, $s\in \mathbb{R}$, $\lambda>\N/p$ and $f\in \mathcal (\mathring{\mathcal G}^\epsilon(\beta,\gamma))'$ with $\epsilon\in (0,1]$ and $\beta,\gamma\in (0,\epsilon)$,
		$$\displaystyle \Big\{\sum_{j\in \mathbb{Z}}\left(\delta^{-js}\|\varphi^*_{j,\lambda}(\sqrt{L})f\|_{L^p(dw)}\right)^q\Big\}^{1/q}\lesi \Big\{\sum_{j\in \mathbb{Z}}\left(\delta^{-js}\|\psi_{j}(\sqrt{L})f\|_{L^p(dw)}\right)^q\Big\}^{1/q}.
		$$
		
		\item For $0< p, q<\vc$,  $s\in \mathbb{R}$,  $\lambda>\max\{\N/q, \N/p\}$, and $f\in \mathcal (\mathring{\mathcal G}^\epsilon(\beta,\gamma))'$ with $\epsilon\in (0,1]$ and $\beta,\gamma\in (0,\epsilon)$,
		$$\displaystyle \Big\|\Big[\sum_{j\in \mathbb{Z}}(\delta^{-js}|\varphi^*_{j,\lambda}(\sqrt{L})f|)^q\Big]^{1/q}\Big\|_{L^p(dw)}\lesi \Big\|\Big[\sum_{j\in \mathbb{Z}}(\delta^{-js}|\psi_{j}(\sqrt{L})f|)^q\Big]^{1/q}\Big\|_{L^p(dw)}.$$
	\end{enumerate}
\end{prop}

Propositions \ref{prop-duality}, \ref{prop-comple interpolation} and  \ref{mainthm-Interpolation} below are taken from \cite{BD}.
\begin{prop}
	\label{prop-duality}
	Let $s\in \mathbb{R}$ and $1<p,q<\vc$. The dual space $[\FF^{s,L}_{p,q}(dw)]^*$ of the Triebel-Lizorkin space $\FF^{s,L}_{p,q}(dw)$ is $\FF^{-s,L}_{p',q'}(dw)$.  {More precisely, if $f\in \FF^{-s,L}_{p',q'}(dw)$, then 
		$$\mathcal{L}_f(g)=\langle f,g \rangle:= \int_{\RN} f(x)g(x)dw(x)
		$$ 
		defines a linear functional on $\FF^{s,L}_{p,q}(dw)$  which satisfies 
		\begin{equation*}\label{eq1-duality}
			\left|\langle f,g \rangle \right|\lesi \|f\|_{\FF^{-s,L}_{p',q'}(dw)}\|g\|_{\FF^{s,L}_{p,q}(dw)}\quad \text{for all 
				$g\in \FF^{s,L}_{p,q}(dw)$}.
	\end{equation*}}
	
	{Conversely, if $\mathcal{L}$ is a bounded linear functional on $\FF^{s,L}_{p,q}(dw)$, then there exists $f\in \FF^{-s,L}_{p',q'}(dw)$ so that $\mathcal{L}_f(g)=\langle f,g \rangle$  for every $g\in \FF^{s,L}_{p,q}(dw)$ and $\|f\|_{\FF^{-s,L}_{p',q'}(dw)}\lesi \|\mathcal{L}\|$.}
\end{prop}

\begin{prop}
	\label{prop-comple interpolation}
	We have
	\begin{equation}
		\label{complex interpolation}
		\left(\FF^{s_0,L}_{p_0,q_0}(dw),\FF^{s_1,L}_{p_1,q_1}(dw)\right)_\theta = \FF^{s,L}_{p,q}(dw)
	\end{equation}
	for all $s_0, s_1 \in \mathbb{R}$, $0<p_0,p_1, q_0, q_1<\vc$, $\theta\in (0,1)$ and
	\[
	s=(1-\theta)s_0 +\theta s_1, \ \ \ \f{1}{p} =\f{1-\theta}{p_0}+\f{\theta}{p_1}, \ \ \ \f{1}{q} =\f{1-\theta}{q_0}+\f{\theta}{q_1}
	\] 	
	where $(\cdot, \cdot)_\theta$ stands for the complex interpolation brackets.
\end{prop}

Our main result of this section is the following theorem.

\begin{thm}
	\label{mainthm-Interpolation}
	Let $\theta\in (0,1), s_1, s_2\in \mathbb{R}, s_1\neq s_2$ and $s=(1-\theta)s_1+\theta s_2$.
	\begin{enumerate}[{\rm (i)}]
		\item If $0< p,  q_1, q_2, q\leq \vc$ then
		\begin{equation}\label{eq-BesovInter}
			\left(\B^{s_1,L}_{p,q_1}(dw),\B^{s_2,L}_{p,q_2}(dw)\right)_{\theta,q}=\B^{s,L}_{p,q}(dw).
		\end{equation}
		
		\item If $0<p, q_1, q_2, q< \vc$ then
		\begin{equation}\label{eq-TLInter}
			\left(\FF^{s_1,L}_{p,q_1}(dw),\FF^{s_2,L}_{p,q_2}(dw)\right)_{\theta,q}=\B^{s,L}_{p,q}(dw).
		\end{equation}
		
	\end{enumerate}
	Here  $(\cdot, \cdot)_{\theta,q}$ stands for the real interpolation brackets.
\end{thm}

We now recall atomic decomposition results for the Besov and Triebel-Lizorkin spaces associated to the Dunkl Laplacian $L$ in \cite{BBD}. Recall that $\mathscr D^d =\{\mathscr D^d_k: k\in \Z\}$ is the dyadic system in $(\RN, d,dw)$.
\begin{defn}\label{defLmol}
	Let $0< p\leq \vc$ and $M\in \mathbb{N}_+$. A function $a$ is said to be an $(L, M, p)$ atom if there exists a dyadic cube $Q^d\in \mathscr{D}^d$ so that
	\begin{enumerate}[{\rm (i)}]
		\item $a=L^{M} b$;
		
		\item ${\rm supp} \,L^{k} b\subset 6 Q^d$, $k=0,\ldots , 2M$;
		
		\item $\displaystyle |L^{k} b(x)|\leq \ell(Q^d)^{2(M-k)}w(Q^d)^{-1/p}$, $k=0,\ldots , 2M$.
	\end{enumerate}
	
\end{defn}

The following results on the atomic decompositions for the Besov and Triebel--Lizorkin are taken from Theorem 4.2, Theorem 4.3, Theorem 4.6 and Theorem 4.7 in \cite{BBD}. Note that the convergence of the atomic decomposition is due to Theorem \ref{Calderon reproducing}.
\begin{thm}\label{thm1- atom Besov - associated to L}
	Let $s\in \mathbb{R}$ and  $0<p,q< \vc$. Let $\epsilon\in (0,1]$ and $\beta,\gamma\in (0,\epsilon)$. For  $M\in \mathbb{N}_+$, if $f\in \B^{s,L}_{p,q}(dw)\cap \mathcal (\mathring{\mathcal G}^\epsilon(\beta,\gamma))'$ then there exist a sequence of $(L,M,p)$ atoms $\{a_{Q^d}\}_{Q^d\in \mathscr{D}_k, k\in \mathbb{Z}}$ and a sequence of coefficients  $\{s_{Q^d}\}_{Q^d\in \mathscr{D}_k, k\in\mathbb{Z}}$ so that
	$$
	f=\sum_{k\in\mathbb{Z}}\sum_{Q^d\in \mathscr{D}_k}s_{Q^d}a_{Q^d} \ \ \text{in $\mathcal (\mathring{\mathcal G}^\epsilon(\beta,\gamma))'$}.
	$$
	Moreover,
	$$
	\Big[\sum_{k\in\mathbb{Z}}\delta^{- sk q}\Big(\sum_{Q^d\in \mathscr{D}^d_k}|s_{Q^d}|^p\Big)^{q/p}\Big]^{1/q}\lesi \|f\|_{\B^{s,L}_{p,q}(dw)}.
	$$ 
	Conversely, for $M>\f{\N}{2}+\f{1}{2}\max\{s,\f{\N}{1\wedge p\wedge q}-s\}$,	if there exist a sequence of $(L,M,p)$ atoms $\{a_{Q^d}\}_{Q^d\in \mathscr{D}^d_k, k\in \mathbb{Z}}$ and a sequence of coefficients  $\{s_{Q^d}\}_{Q^d\in \mathscr{D}^d_k, k\in\mathbb{Z}}$
	$$
	f=\sum_{k\in\mathbb{Z}}\sum_{Q^d\in \mathscr{D}^d_k}s_{Q^d}a_{Q^d} \ \ \text{in $\mathcal (\mathring{\mathcal G}^\epsilon(\beta,\gamma))'$}
	$$
	and
	$$
	\Big[\sum_{k\in\mathbb{Z}}\delta^{-ks q}\Big(\sum_{Q^d\in \mathscr{D}_k}|s_{Q^d}|^p\Big)^{q/p}\Big]^{1/q}<\vc,
	$$
	then $f\in \B^{s,L}_{p,q}(dw)$ and
	$$
	\|f\|_{\B^{s,L}_{p,q}(dw)} \lesi \Big[\sum_{k\in\mathbb{Z}}\delta^{-ks q}\Big(\sum_{Q^d\in \mathscr{D}^d_k}|s_{Q^d}|^p\Big)^{q/p}\Big]^{1/q}.
	$$
	
\end{thm}	

Similar results also hold for the Triebel--Lizorkin spaces $\FF^{\alpha,L}_{p,q}(dw)$.
\begin{thm}\label{thm1- atom TL spaces- associated to L}
	Let $s\in \mathbb{R}$, $0<p,q<\vc$. Let $\epsilon\in (0,1]$ and $\beta,\gamma\in (0,\epsilon)$. If $f\in \FF^{s,L}_{p,q}(dw)$ then there exist a sequence of $(L,M,p)$ atoms $\{a_{Q^d}\}_{Q^d\in \mathscr{D}^d_k, k\in \mathbb{Z}}$ and 
	a sequence of coefficients  $\{s_{Q^d}\}_{Q^d\in \mathscr{D}^d_k, k\in\mathbb{Z}}$ so that
	$$
	f=\sum_{k\in\mathbb{Z}}\sum_{Q^d\in \mathscr{D}^d_k}s_{Q^d}a_{Q^d} \ \ \text{in $\mathcal (\mathring{\mathcal G}^\epsilon(\beta,\gamma))'$}.
	$$
	Moreover,
	\begin{equation}\label{eq1-thm1 atom TL space}
		\Big\|\Big[\sum_{k\in\mathbb{Z}}\delta^{-ks q}\Big(\sum_{Q^d\in \mathscr{D}^d_k}w(Q^d)^{-1/p}|s_{Q^d}|\chi_{Q^d}\Big)^q\Big]^{1/q}\Big\|_{L^p(dw)}\lesi \|f\|_{\F^{s,L}_{p,q}(dw)}.
	\end{equation}
	Conversely, if there exist a sequence of $(L,M,p)$ atoms $\{a_{Q^d}\}_{Q^d\in \mathscr{D}^d_k, k\in \mathbb{Z}}$ with $M>\f{\N}{2}+\f{1}{2}\max\{s,\f{\N}{1\wedge p\wedge q}-s\}$ and 
	a sequence of coefficients  $\{s_{Q^d}\}_{Q^d\in \mathscr{D}^d_k, k\in\mathbb{Z}}$ such  that
	$$
	f=\sum_{\nu\in\mathbb{Z}}\sum_{Q^d\in \mathscr{D}^d_\nu}s_{Q^d}a_{Q^d} \ \ \text{in $\mathcal (\mathring{\mathcal G}^\epsilon(\beta,\gamma))'$}
	$$
	and
	$$
	\Big\|\Big[\sum_{k\in\mathbb{Z}}\delta^{-ks q}\Big(\sum_{Q^d\in \mathscr{D}^d_k}w(Q^d)^{-1/p}|s_{Q^d}|\chi_{Q^d}\Big)^q\Big]^{1/q}\Big\|_{L^p(dw)}<\vc,
	$$
	then $f\in \FF^{s,L}_{p,q}(dw)$ and
	$$
	\|f\|_{\FF^{s,L}_{p,q}(dw)} \lesi \Big\|\Big[\sum_{k\in\mathbb{Z}}\delta^{-ks q}\Big(\sum_{Q^d\in \mathscr{D}^d_\nu}w(Q^d)^{-1/p}|s_{Q^d}|\chi_{Q^d}\Big)^q\Big]^{1/q}\Big\|_{L^p(dw)}.
	$$
	
\end{thm}

From the atomic decomposition results above, it is easy to see that $\B^{s,L}_{p,q}(dw)\cap L^2(dw)$ and $\FF^{s,L}_{p,q}(dw)\cap L^2(dw)$ are  dense in $\B^{\alpha,L}_{p,q}(dw)$ and $\FF^{s,L}_{p,q}(dw)$ for all $s\in \mathbb{R}$ and $0<p,q<\vc$, respectively. Consequently, we have the following corollary.
\begin{cor}\label{cor1}
	The Besov and Triebel-Lizorkin spaces defined in Definition \ref{defn 1} are independent of the choices of $\epsilon \in (0,1]$ and $ \beta, \gamma\in (0,\epsilon)$.
\end{cor}

\bigskip

Similarly to the classical case, the new Triebel-Lizorkin spaces are identical to some known function spaces in the particular cases of the indices. 

Let $0<p\le 1$. The Hardy space $H^p_{L}(\RN)$ is defined as the completion of the set
\begin{equation*}
	\left\{f\in L^2: \mathcal{S}_Lf\in L^p(dw) \right\}
\end{equation*}
under the norm $\|f\|_{H^p_{L}(dw)}=\|\mathcal{S}_Lf\|_{L^p(dw)}$ where
\[
\mathcal{S}_Lf(x)=\Big[\int_0^\vc\int_{d(\x,\y)<t}|t^2Le^{-tL^2}f(\y)|^2\f{dw(\y) dt}{tw(B(\x,t))}\Big]^{1/2}.
\]
The Hardy spaces $H^1_L$ was initiated in \cite{ADM}. See also \cite{DY}. The theory of Hardy spaces associated to operators satisfying Davies--Gaffney estimates  $H^1_L$ was established in \cite{HLMMY}.

\begin{defn}[\cite{HLMMY, JY, BCKYY1}]\label{def: L-atom} 
	Let $0<p\le 1$ and $M\in \mathbb{N}$ and let $r\in (1,\vc)$. A function $a(\x)$ supported in a ball $B \subset \RN$ of radius $r_{B}$ is called a  $(p,r,M,L)$-atom if there exists a 	function $b\in D(L^{M})$ such that
	\begin{enumerate}[{\rm (i)}]
		\item  $a=L^M b$;
		\item $\supp L ^{k}b\subset \mathcal O(B), \ k=0, 1, \dots, M$;
		\item $\|L^{k}b\|_{L^r(dw)}\leq
		r_B^{2(M-k)}w(B)^{1/r-1/p},\ k=0,1,\dots,M$.
	\end{enumerate}
\end{defn}

\begin{defn}[Atomic Hardy spaces for $L$]
	
	Given  $0<p\le 1$, $r\in (1,\vc)$ and $M\in \mathbb{N}$, we  say that $f=\sum
	\lambda_ja_j$ is an atomic $(p,2,M,L)$-representation if
	$\{\lambda_j\}_{j=0}^\infty\in \ell^p$, each $a_j$ is a $(p,r,M,L)$-atom,
	and the sum converges in $L^2(dw)$. The space $H^{p,r}_{L,at,M}(dw)$ is then defined as the completion of
	\[
	\left\{f\in L^2(dw):f \ \text{has an atomic
		$(p,r,M,L)$-representation}\right\},
	\]
	with the norm given by
	$$
	\|f\|^p_{H^{p,r}_{L,at,M}(dw)}=\inf\left\{ \sum|\lambda_j|^p :
	f=\sum \lambda_ja_j \ \text{is an atomic $(p,r,M,L)$-representation}\right\}.
	$$
\end{defn}

\begin{thm}[\cite{BCKYY1}]\label{thm-Hardy space}
	Let $p\in (0,1]$, $r\in (1,\vc)$ and $M>\f{\N}{2}(\f{1}{p}-\f{2}{\N r})$. Then the Hardy spaces $H^{p,r}_{L,at,M}(dw)$ and $H^{p}_{L}(dw)$ coincide and have equivalent norms.
\end{thm}

The following identifications are just direct consequences of Theorem 5.2 and Theorem 5.7 in \cite{BBD}.
\begin{thm}
	\label{equiv-Hardy}
	\begin{enumerate}[{\rm (a)}]
		\item Let $0<p\le 1$. Then we have
		\[
		H^p_{L}(dw)\equiv \FF^{0,L}_{p,2}(dw).
		\]
		\item  Let $p\in (1,\vc)$ and $s\in \mathbb R$. Then we have
		\[
		\|L^{s/2}f\|_{L^p(dw)}\simeq \|f\|_{\FF^{s,L}_{p,2}(dw)}
		\]
	\end{enumerate}
\end{thm}

\subsubsection{\textbf{A case study}} In this section, we will particularly consider the Besov and Triebel-Lizorkin spaces associated to the operator $L$ with a suitable range of  indices containing those in Definition \ref{defn-Besov and TL spaces}.

%In what follows, for $s\in (-1,1)$ and $\epsilon\in (0,1]$, we set
%\[
%p(s,\epsilon)=\max\Big\{\f{\N}{\N +\epsilon},\f{\N}{\N +s+\epsilon}\Big\}.
%\]
\begin{thm}\label{thm 1}
	Let $s\in (-1,1)$ and $p\in (\f{\N}{\N+s+1},\vc)$, $q\in (0,\vc)$. For any $\epsilon \in (0,1]$ and $\beta,\gamma\in (0,\epsilon)$ such that $|s|<\beta,\gamma$ and $p>\f{\N}{\N+s+\beta}$, we have
	\[
	\B^{s,  L}_{p,q}(dw)=\Big\{f\in (\mathring{\mathcal G}^\epsilon(\beta,\gamma))': \|f\|_{\B^{s, L}_{p,q}(dw)}<\vc\Big\}.
	\]
	
	Similarly, let $s\in (-1,1)$ and $p,q\in (\f{\N}{\N+s+1},\vc)$. For any $\epsilon \in (0,1]$ and $\beta,\gamma\in (0,\epsilon)$ such that $|s|<\beta,\gamma$ and $p, q>\f{\N}{\N+s+\beta}$, we have
	\[
	\FF^{s,  L}_{p,q}(dw)=\Big\{f\in (\mathring{\mathcal G}^\epsilon(\beta,\gamma))': \|f\|_{\FF^{s, L}_{p,q}(dw)}<\vc\Big\}.
	\]
	
\end{thm}
\begin{proof}
	We only prove the statement for the Besov spaces. The case of Triebel-Lizorkin spaces can be done similarly.
	
	With the same indices as in the statement of the theorem, it suffices to prove that the set 
	\[
	\B^{s,  L}_{p,q}(dw)=\Big\{f\in (\mathring{\mathcal G}^\epsilon(\beta,\gamma))': \|f\|_{\B^{s, L}_{p,q}(dw)}<\vc\Big\}.
	\]
	is complete.
	
	To do this, we first prove that if $f\in \B^{s,  L}_{p,q}(dw)\cap (\mathring{\mathcal G}^\epsilon(\beta,\gamma))'$, then for any $g\in \mathring{\mathcal G}^\epsilon(\beta,\gamma)$,
	\begin{equation}\label{eq1-thm 3.11}
		|\langle f,g \rangle| \lesi \|g\|_{\mathring{\mathcal G}( \beta, \gamma)}\|f\|_{\B^{s, L}_{p,q}(dw)}.
	\end{equation}
	Indeed, let $\psi$ be a partition of unity.  It is well-known that there exists  an even function $\varphi \in \mathscr{S}(\mathbb{R})$ with  ${\rm supp}\, \varphi\subset [-c,-b]\cup[b,c]$ for some $c>b>0$ so that 
	\[
	\sum_{j=-\infty}^\infty\psi_j(\lambda)\varphi_j(\lambda)=1, \ \ \forall \lambda\neq 0.
	\]
	See for example \cite{ST}.
	
	Let $g\in \mathring{\mathcal G}(\epsilon,\epsilon)$. By Theorem \ref{thm - Calderon discrete version},
	\[
	f=\sum_{j\in \Z} \psi_j(\sqrt L)\varphi_j(\sqrt L)f 
	\]
	\text{in $(\mathring{\mathcal G}^\epsilon(\beta,\gamma))'$}.
	
	Consequently,
	\[
	\langle f,g \rangle = \Big\langle \sum_{j\in \Z} \psi_j(\sqrt L)\varphi_j(\sqrt L)f, g \Big\rangle,
	\]
	which follows that 
	\[
	\begin{aligned}
		|\langle f,g \rangle|&\le  \sum_{j\in \Z} \Big|\int \psi_j(\sqrt L)f(\x) \varphi_j(\sqrt L)g(\x)dw(\x)\Big|\\
		&\le  \sum_{j\in \Z} \int |\psi_j(\sqrt L)f(\x)| |\varphi_j(\sqrt L)g(\x)|dw(\x).		
	\end{aligned}
	\]
	Applying Lemma \ref{lem-estimate for varphi t f x},
	\[
	\begin{aligned}
		|\langle f,g \rangle| &\lesi \|g\|_{\mathring{\mathcal G}( \beta, \gamma)}\sum_{j\ge 0} \delta^{j\beta}\int |\psi_j(\sqrt L)f(\x)|P_\gamma(0,\x;1)  dw(\x)\\
		& \ \ \ +\|g\|_{\mathring{\mathcal G}( \beta, \gamma)}\sum_{j< 0} \delta^{-j\gamma}\int |\psi_j(\sqrt L)f(\x)|P_{\beta\wedge \gamma}(0,\x;\delta^{j})  dw(\x).
	\end{aligned}
	\]

	\textbf{Case 1: $p\ge 1$}

	Since $w(B(0, \|\x\|+1)), w(B(0, \|\x\|+\delta^j)) \gtrsim 1$  (due to \eqref{eq-volume formula}) for $j<0$,  by H\"older's inequality we have
	\[
	\begin{aligned}
		\int |\psi_j(\sqrt L)f(\x)|P_\gamma(0,\x;1)  dw(\x)&\le \|\psi_j(\sqrt L)f\|_{L^p(dw)}\Big(\int_{\RN}P_\gamma(0,\x;1)^{p'}  dw(\x) \Big)^{1/p'}\\
		&\lesi  \|\psi_j(\sqrt L)f\|_{L^p(dw)}\Big(\int_{\RN}P_{\gamma p'}(0,\x;1)  dw(\x) \Big)^{1/p'}\\
		&\lesi  \|\psi_j(\sqrt L)f\|_{L^p(dw)}
	\end{aligned}
	\]
	and for $j<0$,
	\[
	\begin{aligned}
		\int |\psi_j(\sqrt L)f(\x)|P_\gamma(0,\x;1)  dw(\x)&\le \|\psi_j(\sqrt L)f\|_{L^p(dw)}\Big(\int_{\RN}P_\gamma(0,\x;\delta^j)^{p'}  dw(\x) \Big)^{1/p'}\\
		&\lesi  \|\psi_j(\sqrt L)f\|_{L^p(dw)}\Big(\int_{\RN}P_{\gamma p'}(0,\x;\delta^j)  dw(\x) \Big)^{1/p'}\\
		&\lesi  \|\psi_j(\sqrt L)f\|_{L^p(dw)}.
	\end{aligned}
	\]
	Consequently,
	\[
	\begin{aligned}
		|\langle f,g \rangle| &\lesi \|g\|_{\mathring{\mathcal G}( \beta, \gamma)}\sum_{j\ge 0} \delta^{j\beta}\|\psi_j(\sqrt L)f\|_{L^p(dw)} +\|g\|_{\mathring{\mathcal G}( \beta, \gamma)}\sum_{j< 0} \delta^{-j\gamma}\|\psi_j(\sqrt L)f\|_{L^p(dw)}\\
		&\lesi \|g\|_{\mathring{\mathcal G}( \beta, \gamma)}\Big\{\sum_{j\in \mathbb{Z}}\left(\delta^{-js}\|\psi_{j}(\sqrt{L})f\|_{L^p(dw)}\right)^q\Big\}^{1/q}\\
		&\lesi \|g\|_{\mathring{\mathcal G}( \beta, \gamma)}\|f\|_{\B^{s, L}_{p,q}(dw)}
	\end{aligned}
	\]
	as along as $|s|<\min\{\beta, \gamma\}$.

	\textbf{Case 2: $0<p<1$}
	
	In this situation, we have
	\[
	\begin{aligned}
		|\langle f,g \rangle|		&\lesi \|g\|_{\mathring{\mathcal G}( \beta, \gamma)}\sum_{j\ge 0} \delta^{j\beta}\sum_{Q^d\in\mathscr{D}^d_j}\int_{Q^d} |\psi_j(\sqrt L)f(\x)|P_\gamma(0,\x;1)  dw(\x)\\
		& \ \ \ +\|g\|_{\mathring{\mathcal G}( \beta, \gamma)}\sum_{j< 0} \delta^{-j\gamma}\sum_{Q^d\in\mathscr{D}^d_j}\int_{Q^d} |\psi_j(\sqrt L)f(\x)|P_{\beta\wedge \gamma}(0,\x;\delta^{j})  dw(\x).
	\end{aligned}
	\]
	It is easy to see that for $Q^d\in \mathscr{D}_j^d$ and $\lambda>\N/p$, 
	\[
	|\psi_j(\sqrt L)f(\x)|\lesi \inf_{\z\in Q^d}\psi^*_{j,\lambda}(\sqrt L)f(\z), \ \ \ \text{for all $\x\in Q^d$}. 
	\]
	Therefore,
	\[
	\begin{aligned}
		|\langle f,g \rangle|&\lesi \|g\|_{\mathring{\mathcal G}( \beta, \gamma)}\sum_{j\ge 0} \delta^{j\beta}\sum_{Q^d\in\mathscr{D}^d_j}\inf_{\z\in Q^d}\psi^*_{j,\lambda}(\sqrt L)f(\z)\int_{Q^d} P_\gamma(0,\x;1)  dw(\x)\\
		& \ \ \ +\|g\|_{\mathring{\mathcal G}( \beta, \gamma)}\sum_{j< 0} \delta^{-j\gamma}\sum_{Q^d\in\mathscr{D}^d_j}\inf_{\z\in Q^d}\psi^*_{j,\lambda}(\sqrt L)f(\z)\int_{Q^d}  P_{\beta\wedge \gamma}(0,\x;\delta^{j})  dw(\x).
	\end{aligned}
	\]
	Since $P_\gamma(0,\x;1)\lesi 1$, we have, for $Q^d\in \mathscr{D}^d_j$ with $j\ge 0$,
	\[
	\begin{aligned}
		\int_{Q^d} P_\gamma(0,\x;1)  dw(\x)&\lesi w(Q^d)\\
		&\lesi w(Q^d)^{1/p}w(Q^d)^{1-1/p}\\
		&\lesi \delta^{j \N(1-1/p)}w(Q^d)^{1/p},
	\end{aligned}
	\]
	where in the last inequality we used \eqref{eq-volume formula}.
	
	For $Q^d\in \mathscr{D}^d_j$ with $j< 0$, by using \eqref{eq-volume formula},
	\[
	\begin{aligned}
		\int_{Q^d}  P_{\beta\wedge \gamma}(0,\x;\delta^{j})  dw(\x)&\lesi 1\\
		&\lesi w(Q^d)^{1/p}.
	\end{aligned}
	\]
	Therefore, 
	\[
	\begin{aligned}
		|\langle f,g \rangle|&\lesi \|g\|_{\mathring{\mathcal G}( \beta, \gamma)}\sum_{j\ge 0} \delta^{j(\beta+\N(1-1/p))}\sum_{Q^d\in\mathscr{D}^d_j}w(Q^d)^{1/p}\inf_{\z\in Q^d} \psi^*_{j,\lambda}(\sqrt L)f(\z)  \\
		& \ \ \ +\|g\|_{\mathring{\mathcal G}( \beta, \gamma)}\sum_{j< 0} \delta^{-j\gamma}\sum_{Q^d\in\mathscr{D}^d_j}w(Q^d)^{1/p}\inf_{\z\in Q^d}\psi^*_{j,\lambda}(\sqrt L)f(\z)\\
		&\lesi \|g\|_{\mathring{\mathcal G}( \beta, \gamma)}\sum_{j\ge 0} \delta^{j(\beta+\N(1-1/p))}\Big[\sum_{Q^d\in\mathscr{D}^d_j}w(Q^d)\inf_{\z\in Q^d} |\psi^*_{j,\lambda}(\sqrt L)f(\z)|^p\Big]^{1/p}  \\
		& \ \ \ +\|g\|_{\mathring{\mathcal G}( \beta, \gamma)}\sum_{j< 0} \delta^{-j\gamma}\Big[\sum_{Q^d\in\mathscr{D}^d_j}w(Q^d)\inf_{\z\in Q^d} |\psi^*_{j,\lambda}(\sqrt L)f(\z)|^p\Big]^{1/p}\\
		&\lesi \|g\|_{\mathring{\mathcal G}( \beta, \gamma)}\sum_{j\ge 0} \delta^{j(\beta+\N(1-1/p))}\Big[\sum_{Q^d\in\mathscr{D}^d_j} \int_{Q^d}|\psi^*_{j,\lambda}(\sqrt L)f(\z)|^pdw(\z)\Big]^{1/p}  \\
		& \ \ \ +\|g\|_{\mathring{\mathcal G}( \beta, \gamma)}\sum_{j< 0} \delta^{-j\gamma}\Big[\sum_{Q^d\in\mathscr{D}^d_j}\int_{Q^d} |\psi^*_{j,\lambda}(\sqrt L)f(\z)|^pdw(\z)\Big]^{1/p}\\
		&\lesi \|g\|_{\mathring{\mathcal G}( \beta, \gamma)}\Big(\sum_{j\ge 0} \delta^{j(\beta+\N(1-1/p))}\|\psi^*_{j,\lambda}(\sqrt L)f\|_{L^p(dw)} +\sum_{j\ge 0} \delta^{-j\gamma}\|\psi^*_{j,\lambda}(\sqrt L)f\|_{L^p(dw)} \Big).
	\end{aligned}
	\]
	Since $p>\f{\N}{\N+\beta+s}$, $\beta+s+\N(1-1/p)>0$. This, along with $|s|<\gamma$ and Proposition \ref{prop - equivalence of varphi*}, yields
	\[
	\begin{aligned}
		|\langle f,g \rangle|& \lesi  \|g\|_{\mathring{\mathcal G}( \beta, \gamma)}\Big\{\sum_{j\in \mathbb{Z}}\left(\delta^{-js}\|\psi_{j}(\sqrt{L})f\|_{L^p(dw)}\right)^q\Big\}^{1/q}\\
		&\lesi \|g\|_{\mathring{\mathcal G}( \beta, \gamma)}\|f\|_{\B^{s, L}_{p,q}(dw)}.
	\end{aligned}
	\]
	We have proved that for any $g\in \mathring{\mathcal G}( \epsilon, \epsilon)$,
	\begin{equation}\label{eq- f g besov and test functions}
		|\langle f,g \rangle|\lesi \|g\|_{\mathring{\mathcal G}( \beta, \gamma)}\|f\|_{\B^{s, L}_{p,q}(dw)}.
	\end{equation}
	For $g\in \mathring{\mathcal G}^\epsilon( \beta, \gamma)$, there exists $\{g_n\}\subset \mathring{\mathcal G}( \epsilon, \epsilon)$ such that $\lim_{n}\|g_n-g\|_{\mathring{\mathcal G}( \beta, \gamma)}=0$. Then from \eqref{eq- f g besov and test functions}, 
	\[
	|\langle f,g_m-g_n \rangle|\lesi \|g_m-g_n\|_{\mathring{\mathcal G}( \beta, \gamma)}\|f\|_{\B^{s, L}_{p,q}(dw)}.
	\]
	It follows that the limit $\lim \langle f,g_n \rangle$ exists and hence we can define
	\[
	\langle f,g \rangle = \lim\langle f,g_n \rangle.
	\]
	Therefore,
	\[
	\begin{aligned}
		|\langle f,g \rangle|&\lesi \lim \|g_n\|_{\mathring{\mathcal G}( \beta, \gamma)}\|f\|_{\B^{s, L}_{p,q}(dw)}\\&\lesi \|g\|_{\mathring{\mathcal G}( \beta, \gamma)}\|f\|_{\B^{s, L}_{p,q}(dw)}.
	\end{aligned}	
	\]
	This ensures \eqref{eq1-thm 3.11}.
	
	We now turn to the proof. For  any Cauchy sequence $\{f_n\}\subset \B^{s, L}_{p,q}(dw)\cap (\mathring{\mathcal G}^\epsilon(\beta,\gamma))'$ in $\B^{s, L}_{p,q}(dw)$, by  \eqref{eq1-thm 3.11},
	\[
	|\langle f_m-f_n,g \rangle|\lesi \|g\|_{\mathring{\mathcal G}( \beta, \gamma)}\|f_m-f_n\|_{\B^{s, L}_{p,q}(dw)}.
	\]	
	for any $g\in \mathring{\mathcal G}^\epsilon( \beta, \gamma)$.
	
	It follows that $\{f_n\}$ is also a Cauchy sequence in $(\mathring{\mathcal G}^\epsilon( \beta, \gamma))'$ and hence there exists $f\in (\mathring{\mathcal G}^\epsilon( \beta, \gamma))'$ such that $\lim_{n} f_n = f$ in $(\mathring{\mathcal G}^\epsilon( \beta, \gamma))'$. This, along with Proposition \ref{prop-prop1 kernel is a test function}, yields
	\[
	\lim_{n\to \vc }\psi_j(\sqrt L) f_n(\x)=\psi_j(\sqrt L) f(\x), \ \ \ \x\in \RN,
	\]
	or equivalently,
	\begin{equation}\label{eq2-thm311}
		\lim_{n\to \vc }2^{-js}\psi_j(\sqrt L) f_n(\x)=2^{-js}\psi_j(\sqrt L) f(\x), \ \ \ \x\in \RN.
	\end{equation}
	On the other hand, since $\{f_n\}\subset \B^{s, L}_{p,q}(dw)\cap (\mathring{\mathcal G}^\epsilon(\beta,\gamma))'$ is a Cauchy sequence, $2^{-js}\psi_j(\sqrt L) f_n(\x)$ can be viewed as a Cauchy sequence in the space $\ell^p(L^p(\RN))$. It is well-known that $\ell^p(L^p(\RN))$ is complete. Hence, there exists $F(j,\x)\in \ell^p(L^p(\RN))$ such that 
	\[
	\lim_{n\to \vc}\| 2^{-js}\psi_j(\sqrt L) f_n(\x) - F(j,\x)\|_{\ell^p(L^p(\RN))}=0,
	\]
	which implies 
	\[
	\lim_{n\to \vc }2^{-js}\psi_j(\sqrt L) f_n(\x) = F(j,\x) \ \ \text{for a.e. $\x\in \RN$}.
	\]
	From this and \eqref{eq2-thm311},
	\[
	2^{-js}\psi_j(\sqrt L) f(\x) = F(j,\x) \ \ \text{for a.e. $\x\in \RN$},
	\]
	and hence,
	\[
	\lim_{n\to \vc}\| 2^{-js}\psi_j(\sqrt L) f_n(\x) - 2^{-js}\psi_j(\sqrt L) f(\x)\|_{\ell^p(L^p(\RN))}=0,
	\]
	i.e., $f_n \to f$ in $\B^{s, L}_{p,q}(dw)$.
	
	This completes our proof.
\end{proof}

\subsubsection{\textbf{Atomic decomposition}}

In this section we will establish a new atomic decomposition for the Besov spaces $\B^{s,  L}_{p,q}(dw)$ and Triebel-Lizorkin spaces $\F^{s,  L}_{p,q}(dw)$. The atomic decomposition is different from those in Theorem \ref{thm1- atom Besov - associated to L} and Theorem \ref{thm1- atom TL spaces- associated to L}, since the new atoms in this section are independent of the operator $L$.
\begin{defn}\label{def:7.1}
	Let $ 0<p \leq \infty$. A function $a$ is said to be a $(p,d)$-atom if there exists a dyadic cube $Q \in \mathscr{D}_{k}$ such that \\
	\begin{enumerate}[\rm (i)]
		\item $\operatorname{supp}\,a \subset  \mathcal O(6Q)$;
		\item For any $\x \in \RN$,
		$$\displaystyle |a(\x)| \leq  \Big(1+\f{\|\x-\x_Q\|}{\ell_Q}\Big)^{-2} w(Q)^{-1/p};$$
		
		\item For any $\x,\x'\in\RN$,
		$$|a(\x)-a(\x')| \leq w(Q)^{-1/p} \dfrac{\|\x-\x'\|}{\ell_{Q}} ;
		$$
		\item $\displaystyle \int a(\x)d w(\x)=0$.
	\end{enumerate}
\end{defn}

%\begin{defn}\label{def:7.1}
%	Let $ 0<p \leq \infty$ and $\epsilon >0$. A function $a$ is said to be a %$(\beta,\epsilon)$ atom if there exists a dyadic cube $Q \in \mathscr{D}_{k}$ %such that \\
%	\begin{enumerate}[\rm (i)]
	%		%\item $\operatorname{supp}\,a \subset 3Q$;
	%		\item For any $\x \in \RN$,
	%		$$\displaystyle |b(\x)| \leq %w(Q)^{1/2}\f{1}{\mu(Q)+V(\x_Q,\x)}\Big[\f{\ell_Q}{\ell_Q+ %\|\x-\x_Q\|}\Big]^\epsilon;$$
	
	%		\item For any $\x,\x'\in\RN$ with $\|\x-\x'\|\le %\f{1}{2}(\ell_Q+\|\x-\x_Q\|)$,
	%		$$|b(\x)-b(\x')| \leq w(Q)^{1/2}\Big(\dfrac{\|\x-\x'\|}{\ell_Q}\Big)^\beta\f{1}{\mu(Q)+V(\x_Q,\x)}\Big[\f{\ell_Q}{\ell_Q+ \|\x-\x_Q\|}\Big]^\epsilon;
	%		$$
	%		\item $\displaystyle \int b(\x)d\mu w(\x)=0$.
	%	\end{enumerate}
%\end{defn}
\begin{thm}\label{thm1- atom Besov-new atom}
	Let $s\in (-1,1)$ and $p\in (p(s,1),\vc]$, $q\in (0,\vc]$. If $f\in \B^{s,  L}_{p,q}(dw)$, then   there exist a sequence of $(p,d)$-atoms $\{a_{Q}\}_{Q\in \mathscr{D}_\nu, \nu\in \mathbb{Z}}$ and a sequence of coefficients  $\{s_{Q}\}_{Q\in \mathscr{D}_\nu, \nu\in\mathbb{Z}}$ such that for every $\epsilon, \beta, \gamma$ satisfying  $|s|<\beta,\gamma<\epsilon\le 1$ and $p>p(s,\beta)$ 
	$$
	f=\sum_{k\in\mathbb{Z}}\sum_{Q\in \mathscr{D}_k}s_{Q}a_{Q} \ \ \text{in $(\mathring{\mathcal G}^\epsilon(\beta,\gamma))'$}.
	$$
	Moreover,
	$$
	\Big[\sum_{k\in\mathbb{Z}}\delta^{-s k q}\Big(\sum_{Q\in \mathscr{D}_k}|s_{Q}|^p\Big)^{q/p}\Big]^{1/q}\lesi \|f\|_{\B^{s,L}_{p,q}(dw)}.
	$$
	%	Conversely, if there exist a sequence of $(p,d)$-atoms $\{a_{Q^d}\}_{Q^d\in \mathscr{D}^d_\nu, \nu\in \mathbb{Z}}$ and a sequence of coefficients  $\{s_{Q^d}\}_{Q^d\in \mathscr{D}^d_\nu, \nu\in\mathbb{Z}}$ such that 
	%	$$
	%	f=\sum_{k\in\mathbb{Z}}\sum_{Q^d\in \mathscr{D}_k^d}s_{Q^d}a_{Q^d} \ \ \text{in $(\mathring{\mathcal G}^\epsilon(\beta,\gamma))'$}
	%	$$
	%	for some $\beta, \gamma$ satisfying  $|s|<\beta,\gamma<1$ and $p>p(s,\beta)$, then
	%	$$
	%	\|f\|_{\B^{s,L}_{p,q}(dw)}\lesi \Big[\sum_{k\in\mathbb{Z}}\delta^{-s k q}\Big(\sum_{Q^d\in \mathscr{D}^d_\nu}|s_{Q^d}|^p\Big)^{q/p}\Big]^{1/q}.
	%	$$
\end{thm}

\begin{proof}
	Let  $\Phi$ be a function as in Lemma \ref{lem:finite propagation}. Then there exists an even function $\psi \in \mathscr{S}(\mathbb{R})$ with  ${\rm supp}\, \psi\subset [-c,-b]\cup[b,c]$ for some $c>b>0$ so that 
	\[
	\sum_{k=-\infty}^\infty\lambda^2\Phi_k(\lambda)\psi_k(\lambda)=1, \ \ \forall \lambda\neq 0.
	\]
	See for example \cite{ST}. Due to Theorem \ref{thm - Calderon discrete version}, for $f\in (\mathring{\mathcal G}^\epsilon(\beta,\gamma))'$ we have
	\[
	f=\sum_{k=-\infty}^\infty(\delta^k\sqrt L)^2\Phi_k(\sqrt L)\psi_k(\sqrt L)f
	\]
	in $(\mathring{\mathcal G}^\epsilon(\beta,\gamma))'$.
	
	As a consequence, we have
	\begin{equation}
		\label{eq1-atom}
		\begin{aligned}
			f =\sum_{k=-\infty}^\infty\sum_{Q\in \mathscr{D}_k}(\delta^k\sqrt L)^2\Phi_k(\sqrt L)[\psi_k(\sqrt L)f\cdot \chi_{Q}].
		\end{aligned}
	\end{equation}

	For each $k\in \mathbb{Z}$ and $Q\in \mathscr{D}_k$, we set
	$$
	s_{Q}=  w(Q)^{1/p} \sup_{\y\in  Q}|\psi_k(\sqrt L)f(\y)|,
	$$
	and  
	\begin{equation}\label{eq-bQ}
		a_{Q}=\f{1}{s_{Q}}  (\delta^k\sqrt L)^2\Phi_k(\sqrt L)[\psi_k(\sqrt L)f\cdot \chi_{Q}].
	\end{equation}
	Obviously,  we deduce from \eqref{eq1-atom} that
	$$
	f=\sum_{k\in\mathbb{Z}}\sum_{Q\in \mathscr{D}_k}s_{Q}a_{Q} \ \ \text{in $(\mathring{\mathcal G}^\epsilon(\beta,\gamma))'$}.
	$$
	
	Using  Lemma \ref{lem:finite propagation} we can see that
	\[
	{\rm supp}\,a_{Q}\subset 6 \mathcal O(Q), \ \ \int a_{Q}(\x) dw(\x) = 0.
	\]
	and
	$$
	\begin{aligned}
		|a_{Q}(\x)|&\lesi \f{1}{s_{Q}}  \int_{Q} \Big(1+\f{\|\x-\x_Q\|}{\ell_Q}\Big)^{-2}\f{1}{w(B(\y,\delta^k))}|\psi_k(\sqrt L)f(\y)|dw(\y)\\
		&\lesi w(Q)^{-1/p} \Big(1+\f{\|\x- \x_Q\|}{\ell_Q}\Big)^{-2};
	\end{aligned}
	$$
	moreover, for $\x,\x'\in \RN$,
	\[
	\begin{aligned}
		|a_{Q^d}(\x)-a_{Q^d}(\x')|&\lesi \f{1}{s_{Q^d}}  \f{\|\x-\x'\|}{ \delta^k}\int_{Q^d} \f{1}{w(B(\y,\delta^k))}|\psi_k(\sqrt L)f(\y)|dw(\y)\\
		&\lesi \f{\|\x-\x'\|}{ \delta^k}w(Q^d)^{-1/p}.
	\end{aligned}
	\]
	It follows that $a_{Q^d}$ is (a multiple of) a $(p,d)$-atom.
	
	It remains to show that 
	$$
	\Big[\sum_{k\in\mathbb{Z}}\delta^{-s k q}\Big(\sum_{Q\in \mathscr{D}_k}|s_{Q}|^p\Big)^{q/p}\Big]^{1/q}\lesi \|f\|_{\B^{s,L}_{p,q}(dw)}.
	$$
	
	Indeed, for $\lambda>\N/p$ is is easy to see that 
	\[
	\begin{aligned}
		s_{Q}&=  w(Q)^{1/p} \sup_{\y\in  Q}|\psi_k(\sqrt L)f(\y)|\\
		&\lesi w(Q)^{1/p} \inf_{\z\in Q}\psi^*_{k, \lambda}(\sqrt L)f(\z)
	\end{aligned}
	\]
	which implies
	\[
	\begin{aligned}
		\sum_{Q\in \mathscr{D}_k}|s_{Q}|^p&\lesi \sum_{Q\in \mathscr{D}_k} w(Q) \inf_{\z\in Q}\psi^*_{k, \lambda}(\sqrt L)f(\z)^p\\
		&\lesi \sum_{Q\in \mathscr{D}_k} \int_{Q}|\psi^*_{k, \lambda}(\sqrt L)f(\z)|^p dw(\z)\\
		&\simeq \int_{\RN}|\psi^*_{k, \lambda}(\sqrt L)f(\z)|^p dw(\z).
	\end{aligned}
	\]
	This, along with Proposition \ref{prop - equivalence of varphi*}, implies that 
	\[
	\begin{aligned}
		\Big[\sum_{k\in\mathbb{Z}}\delta^{-s k q}\Big(\sum_{Q\in \mathscr{D}_k}|s_{Q}|^p\Big)^{q/p}\Big]^{1/q}&\lesi \Big[\sum_{k\in\mathbb{Z}}\Big(\sum_{Q\in \mathscr{D}_k}\delta^{-s k }\|\psi^*_{k, \lambda}(\sqrt L)f\|_{L^p(w)} \Big)^{q/p}\Big]^{1/q}\\
		&\lesi \|f\|_{\B^{s,L}_{p,q}(dw)}.
	\end{aligned}
	\]
	
	This completes our proof.

\end{proof}

As a counter part of Theorem \ref{thm1- atom Besov-new atom}, we have the following atomic decomposition for the Triebel-Lizorkin spaces $\FF^{s,  L}_{p,q}(dw)$.  
\begin{thm}\label{thm1- atom TL-new atom}
	Let $s\in (-1,1)$ and $p\in (p(s,1),\vc)$, $q\in (p(s,1),\vc]$. If $f\in \B^{s,  L}_{p,q}(dw)$, then   there exist a sequence of $(p,d)$-atoms $\{a_{Q}\}_{Q\in \mathscr{D}_\nu, \nu\in \mathbb{Z}}$ and a sequence of coefficients  $\{s_{Q}\}_{Q\in \mathscr{D}_\nu, \nu\in\mathbb{Z}}$ such that for every $\epsilon,\beta, \gamma$ satisfying  $|s|<\beta,\gamma<\epsilon\le 1$ and $p\wedge q>p(s,\beta)$ 
	$$
	f=\sum_{k\in\mathbb{Z}}\sum_{Q\in \mathscr{D}_k}s_{Q}a_{Q} \ \ \text{in $(\mathring{\mathcal G}^\epsilon(\beta,\gamma))'$}.
	$$
	Moreover,
	$$
	\Big\|\Big[\sum_{k\in\mathbb{Z}}2^{-\delta s q}\Big(\sum_{Q\in \mathscr{D}_k}w(Q)^{-1/p}|s_{Q}|\chi_{Q}\Big)^q\Big]^{1/q}\Big\|_{L^p(dw)}\lesi \|f\|_{\B^{s,L}_{p,q}(dw)}.
	$$
	%Conversely, if there exist a sequence of $(p,d)$-atoms $\{a_{Q^d}\}_{Q^d\in \mathscr{D}^d_\nu, \nu\in \mathbb{Z}}$ and a sequence of coefficients  $\{s_{Q^d}\}_{Q^d\in \mathscr{D}^d_\nu, \nu\in\mathbb{Z}}$ such that 
	%$$
	%f=\sum_{k\in\mathbb{Z}}\sum_{Q^d\in \mathscr{D}_k^d}s_{Q^d}a_{Q^d} \ \ \text{in $(\mathring{\mathcal G}^\epsilon(\beta,\gamma))'$}
	%$$
	%for some $\beta, \gamma$ satisfying  $|s|<\beta,\gamma<1$, $p\in (p(s,\beta),\vc)$ and $p\in (p(s,\beta),\vc]$, then
	%$$
	%\|f\|_{\FF^{s,L}_{p,q}(dw)}\lesi \Big[\sum_{k\in\mathbb{Z}}\delta^{-s k q}\Big(\sum_{Q^d\in \mathscr{D}^d_\nu}|s_{Q^d}|^p\Big)^{q/p}\Big]^{1/q}.
	%$$
\end{thm}
\begin{proof}
	The proof of the theorem is similar to that of Theorem \ref{thm1- atom Besov-new atom} and hence we omit the details.
\end{proof}

\begin{rem}
	The reverse directions in Theorems \ref{thm1- atom Besov-new atom} and \ref{thm1- atom TL-new atom} can be established similarly to the proof of Theorem \ref{thm-coincidence Besov and TL spaces} below. I do not pursue this here and leave it to the interested reader. 
\end{rem}

\subsubsection{\textbf{Proof of Theorem \ref{thm-coincidence Besov and TL spaces}}} This section is devoted to proving Theorem \ref{thm-coincidence Besov and TL spaces}. It is worth noting that Theorem \ref{thm-coincidence Besov and TL spaces} will follow from the following technical lemmas.

\begin{lem}\label{lem2-thm2 atom Besov-Dk}
	Let $S_k$ be an ATI with bounded support and $D_k = S_k-S_{k-1}$. Then for any $M>0$ there exits $C>0$ such that for every $k\in \mathbb Z$ and every $(p,d)$-atom $a_{Q}$ associated with some $Q\in \mathscr{D}_\nu, \nu \in \Z$,
	\begin{equation}\label{eq- Dk atom k greater than nu}
		|D_k a_{Q}(\x)|\le C\delta^{k-\nu}w(Q)^{-1/p} \Big(1+\f{d(\x,\x_Q)}{\delta^{\nu}}\Big)^{-M}, \ \  k\ge \nu,
	\end{equation}
	and
	\begin{equation}\label{eq- Dk atom k less than nu}
		|D_k a_{Q}(\x)|\le C\delta^{\nu-k}\f{w(Q)}{w(B(\x_{Q},\delta^k))}w(Q)^{-1/p} \Big(1+\f{d(\x,\x_Q)}{\delta^{k}}\Big)^{-M}, \ \  k< \nu.
	\end{equation}
\end{lem}
\begin{proof}
	%Due to the bounded support condition of $D_k(\x,\y)$, it suffices to prove the lemma with $M=0$.
	
	We now consider two cases: $k\ge \nu$ and $k<\nu$.

	\noindent{\bf Case 1: $k\ge \nu$, i.e., $\delta^{k}< \delta^{\nu}$.}

	Due to the bounded support condition of $D_k$, we have $D_k a_{Q}(\x)=0$, whenever $d(\x,\x_Q) \ge \Lambda \delta^\nu$. Hence, we need only to prove the estimate for $\x$ with $d(\x,\x_Q) < \Lambda \delta^\nu$, i.e., for $M=0$. Since $\int D_k(\x,\y)dw(\y)=0$, we have
	\[
	\begin{aligned}
		D_k a_{Q}(\x)&=\int_{\mathcal O(Q)}D_k(\x,\y)a_{Q}(\y)dw(\y)\\
		&=\int_{\RN}D_k(\x,\y)(a_{Q}(\y)-a_{Q}(\x))dw(\y).
	\end{aligned}
	\]

	On the other hand, 
	\[
	|a_{Q}(\y)-a_{Q}(\x)|\lesi w({Q})^{-1/p}\f{\|\y-\x\|}{\ell({Q})}.
	\]
	
	This, along with the bound of $D_k(\x,\y)$, yields
	$$
	\begin{aligned}
		|D_k a_{Q}(\x)|&\lesi w(Q)^{-1/p}\int_{\{\y: \|\x-\y\|<A_1\delta^k\}} \f{1}{w(B(\x,\delta^k))}\f{\|\y-\x\|}{\ell({Q})} dw(\y)\\
		&\lesi \delta^{k-\nu}w({Q})^{-1/p}.
	\end{aligned}
	$$
	This proves \eqref{eq- Dk atom k greater than nu}.	
	\medskip
	
	\noindent{\bf Case 2: $k< \nu$, i.e., $\delta^{k}> \delta^{\nu}$.}

	Due to the bounded support condition of $D_k$, we have $D_k a_{Q}(\x)=0$, whenever $d(\x,\x_{Q})\ge \Lambda \delta^{k}$. Hence, we need only to prove the estimate for $\x$ with $d(\x,\x_{Q})< \Lambda \delta^{k}$, i.e., for $M=0$.
	
	Using the cancellation property $\int a_Q(\x)dw(\x)=0$, property (iv) in Definition \ref{defn- Sk} and the property (ii) in Definition \ref{def:7.1}, we have	
	\[
	\begin{aligned}
		D_k a_{Q}(\x)&=\int_{\mathcal O(Q)}D_k(\x,\y)a_{Q}(\y)dw(\y)\\
		&=\int_{\mathcal O(Q)}[D_k(\x,\y)-D_k(\x,\x_Q)]a_{Q}(\y)dw(\y)\\
		%&=\sum_{\sigma\in G} \int_{\sigma(Q)}D_k(\x,\y)a_{Q,\sigma}(\y)dw(\y)\\
		%&=\sum_{\sigma\in G}  \int_{\sigma(Q)}[D_k(\x,\y)-D_k(\x,\sigma(\x_Q))] a_{Q,\sigma}(\y) dw(\y)\\
		&\lesi  \int_{\mathcal O(Q)} \f{\|y- \x_Q\|}{\delta^{k}} \f{\ell_Q}{\|\y-\x_Q\|} \f{1}{w(B(\x,\delta^{k}))}w(Q)^{-1/p}dw(\y)\\
		&\lesi  \delta^{\nu-k}\f{w(Q)}{w(B(\x_{Q},\delta^k))}w(Q)^{-1/p}.
	\end{aligned}
	\]

	Hence \eqref{eq- Dk atom k less than nu} follows.
	
	This completes our proof.
\end{proof}

\begin{lem}\label{lem3-atom Besov}
	Let $\psi$ be a partition of unity. Then for any $M>0$ there exists $C>0$ such that for any $k\in \Z$ and any  $p$-atom  $a_Q$ associated with some $Q\in \mathscr{D}_\nu, \nu \in Z$,
	\begin{equation*}
		|\psi_k(\sqrt{L}) a_Q(\x)|\le C\delta^{k-\nu}w(Q)^{-1/p} \Big(1+\f{d(\x,\x_{Q})}{\delta^{\nu}}\Big)^{-M}, \ \  k\ge \nu,
	\end{equation*}
	and
	\begin{equation*}
		|\psi_k(\sqrt{L}) a_Q(\x)|\le C\delta^{\nu-k}\f{w(Q)}{w(B(\x_{Q},\delta^k))}w(Q)^{-1/p} \Big(1+\f{d(\x,\x_{Q})}{\delta^{k}}\Big)^{-M}, \ \  k< \nu.
	\end{equation*}
	
\end{lem}
\begin{proof}
	
	We now consider two cases: $k\ge \nu$ and $k<\nu$.

	\noindent{\bf Case 1: $k\ge \nu$, i.e., $\delta^{k}<\delta^\nu$}

	For $M>0$ and $K>\N$, using Lemma \ref{lem-heat kernel estimate} and the definition of the atoms we have
	$$
	\begin{aligned}
		|\psi_k(\sqrt{L}) a_Q(\x)|&= \Big|\int_{\Lambda Q} \psi_k(\sqrt{L})(\x,\y)(a_Q(\y)-a_Q(\x)) d\mu(\y)\Big|\\
		&\lesi w(Q)^{-1/p}\int_{\Lambda Q} \Big(1+\f{\|\x-\y\|}{\delta^k}\Big)^{-1} \f{\|\x-\y\|}{\ell_Q} \f{1}{w(B(\x,\delta^k))}\Big(1+\f{d(\x,\y)}{\delta^k}\Big)^{-M-K} dw(\y)\\
		&\lesi w(Q)^{-1/p}\f{\delta^k}{\ell_Q}\int_{\Lambda Q} \f{1}{w(B(\x,\delta^k))}\Big(1+\f{d(\x,\y)}{\delta^k}\Big)^{-M-K} dw(\y)
	\end{aligned}
	$$
	For $\y\in \Lambda Q$ and $k\ge \nu$, we have
	\[
	\Big(1+\f{d(\x,\y)}{\delta^k}\Big)^{-M}\simeq \Big(1+\f{d(\x,\x_{Q^d})}{\delta^k}\Big)^{-M}\le \Big(1+\f{d(\x,\x_{Q^d})}{\delta^\nu}\Big)^{-M}.
	\]
	Therefore,
	$$
	\begin{aligned}
		|\psi_k(\sqrt{L}) a_Q(\x)|
		&\lesi  \f{\delta^k}{\ell_Q} w(Q)^{-1/p}\Big(1+\f{d(\x,\x_{Q^d})}{\delta^\nu}\Big)^{-M}\int_{\Lambda Q}  \f{1}{w(B(\x,\delta^k))}\Big(1+\f{d(\x,\y)}{\delta^k}\Big)^{-K} dw(\y)\\
		&\lesi  \f{\delta^k}{\ell_Q} w(Q)^{-1/p}\Big(1+\f{d(\x,\x_{Q^d})}{\delta^\nu}\Big)^{-M}.
	\end{aligned}
	$$
	
	\medskip

	\noindent{\bf Case 2: $k< \nu$, i.e., $\delta^{k}>\delta^\nu$}

	We have
	$$
	\begin{aligned}
		\psi_k(\sqrt{L}) a_Q(\x)&= \int_{\Lambda Q} \big[\psi_k(\sqrt{L})(\x,\y)-\psi_k(\sqrt{L})(\x,\x_Q)\big]b_Q(\y)  dw(\y).
	\end{aligned}
	$$
	This, along with Lemma \ref{lem:finite propagation}, implies that
	
	\[
	\begin{aligned}
		\psi_k(\sqrt{L}) a_Q(\x)&\lesi w(Q)^{-1/p}\int_{\Lambda Q} \f{\|\y-\x_Q\|}{\delta^k} \f{1}{w(B(\y,\delta^k))}\Big(1+\f{d(\x,\y)}{\delta^k}\Big)^{-M} dw(\y)\\
		&\lesi  \f{\ell_Q}{\delta^k} \f{w(Q)}{w(B(\x_{Q},\delta^k))} w(Q)^{-1/p}\Big(1+\f{d(\x,\x_{Q})}{\delta^k}\Big)^{-M}.
	\end{aligned}
	\]
	This completes our proof.	
\end{proof}

\begin{proof}
	[Proof of Theorem \ref{thm-coincidence Besov and TL spaces}:] The proof follows directly from Lemma \ref{lem2-thm2 atom Besov-Dk}, Lemma \ref{lem3-atom Besov} and atomic decomposition theorem. Hence, we just sketch the main ideas.
	
	With notations and indices as in Theorem \ref{thm-coincidence Besov and TL spaces}, for  $f\in \B^s_{p,q}(dw)$, by Theorem \ref{thm1-Besov atomic}, we can write 
	\[
	f=\sum_{\nu\in \Z} \sum_{Q\in \mathscr D_\nu} s_Qa_Q 
	\]
	where $\{a_Q\}$ is a sequence of $p$-atoms and $\{s_Q\}$ is a sequence of numbers as in Theorem \ref{thm1-Besov atomic}.
	
	Let $\psi$ be a partition of unity. Then we have    
	\[
	\psi_k(\sqrt L)f=\sum_{\nu\in \Z} \sum_{Q\in \mathscr D_\nu} s_Q\psi_k(\sqrt L) a_Q. 
	\]
	At this stage, using Lemma \ref{lem3-atom Besov} and arguing similarly to the proof of the reverse direction of Theorem \ref{thm1-TL atomic} we come up with $f\in \B^{s,L}_{p,q}(dw)$, which implies $\B^s_{p,q}(dw)\hookrightarrow \B^{s,L}_{p,q}(dw)$. Similarly, by using Theorem \ref{thm1- atom Besov-new atom} and Lemma \ref{lem2-thm2 atom Besov-Dk}, we obtain $\B^{s,L}_{p,q}(dw)\hookrightarrow \B^{s}_{p,q}(dw)$ and hence 
	\[
	\B^{s,L}_{p,q}(dw)\equiv \B^{s}_{p,q}(dw).
	\]
	
	Arguing similarly,
	\[
	\FF^{s,L}_{p,q}(dw)\equiv \FF^{s}_{p,q}(dw).
	\]
	This completes our proof.
	
\end{proof}
From Theorem \ref{thm-coincidence Besov and TL spaces}, Theorem \ref{thm 1} and Corollary 3.8 in \cite{BBD}, we have the following characterizations of the Besov and Triebel-Lizorkin spaces via heat semigroups of the Dunkle Laplacian operator. 

\begin{cor}\label{cor - heat kernel characterization}
	Let $s\in (-1,1)$ and $p\in (p(s,1),\vc)$, $q\in (0,\vc)$. For any $\epsilon \in (0,1]$ and $\beta,\gamma\in (0,\epsilon)$ such that $|s|<\beta,\gamma$ and $p>p(s,\epsilon)$, we have
	\[
	\B^{s}_{p,q}(dw)=\Big\{f\in (\mathring{\mathcal G}^\epsilon(\beta,\gamma))': \Big[\int_0^\vc \big[t^{-s/2}\|tLe^{-tL} f\|_{L^p(dw)}\big]^q\f{dt}{t}\Big]^{1/q}<\vc\Big\};
	\]
	moreover,
	\[
	\|f\|_{\B^{s}_{p,q}(dw)}\simeq \Big[\int_0^\vc \big[t^{-s/2}\|tLe^{-tL} f\|_{L^p(dw)}\big]^q\f{dt}{t}\Big]^{1/q}.
	\]
	
	Similarly, let $s\in (-1,1)$ and $p,q\in (p(s,1),\vc)$. For any $\epsilon \in (0,1]$ and $\beta,\gamma\in (0,\epsilon)$ such that $|s|<\beta,\gamma$ and $p\wedge q>p(s,\epsilon)$, we have
	\[
	\FF^{s}_{p,q}(dw)=\Big\{f\in (\mathring{\mathcal G}^\epsilon(\beta,\gamma))': \Big\|\Big[\int_0^\vc \big[t^{-s/2} tLe^{-tL} f\big]^q\f{dt}{t}\Big]^{1/q}\Big\|_{L^p(dw)}<\vc\Big\};
	\]
	moreover,
	\[
	\|f\|_{\FF^{s}_{p,q}(dw)}\simeq \Big\|\Big[\int_0^\vc \big[t^{-s/2} tLe^{-tL} f\big]^q\f{dt}{t}\Big]^{1/q}\Big\|_{L^p(dw)}.
	\]
\end{cor}

\section{Multiplier theorem for the Dunkl transform}
Recall that for a bounded function $m$ on $\RN$,  define
\[
T_mf = \F^{-1}(m\F f).
\]
From \eqref{eq-convolution defn}, $T_m$ has the associated kernel $T_m(\x,\y)$ defined by
\begin{equation}\label{eq- kernel formula Tm}
	T_m(\x,\y):=\tau_{-\y}\F^{-1}m(\x).
\end{equation}

Recall that  the operator $L$ is a non-negative and  self-adjoint operator on $L^2(dw)$ and generates a semigroup $e^{-tL}$ whose kernel $h_t(\x,\y)$ is defined by
\[
h_t(\x,\y) = c_k^{-1}(2t)^{-\N/2} \exp\Big(-\f{\|\x\|^2+\|\y\|^2}{4t}\Big)E\Big(\f{\x}{\sqrt{2t}},\f{\y}{\sqrt{2t}}\Big).
\]
On the other hand, 
\begin{equation}\label{eq- heat kernel via Fourier transform}
	h_t(\x,\y)=\tau_\x h_t(-\y)=\tau_{-\y}h_t(\x),
\end{equation}
where $h_t(\x)=\F^{-1}(e^{-t\|\xi\|^2})(\x)=c_k^{-1}(2t)^{-\N/2}e^{-\|\x\|^2/(4t)}$. See for example \cite{AH, R}.

We now give a relationship between the Dunkl multiplier and the spectral multiplier of the Dunkl Laplacian operator $L$. Since $\F(Lf)(\xi) =\|\xi\|^2 \F f(\xi)$, i.e., $Lf(\x)=\F^{-1}(\|\xi\|^2 \F f)(\x)$, by the spectral theory, for any Borel bounded function $F: [0,\vc)\to \mathbb C$,
\begin{equation}\label{eq-spectral multiplier and Tm}
	F(L)f = T_{\widetilde F} f,
\end{equation}
where $\widetilde F(\xi) = F(\|\xi\|^2)$ for $\xi \in \RN$.

\medskip
The main aim of this section is to prove the following result.

\begin{thm}\label{mainthm-spectralmultipliers-space adapted to L} 	Let $\alpha>\f{\N}{2}$ and let   $s\in \mathbb{R}$ and $0<p,q<\vc$. Let $\psi$ be a smooth radial function defined on $\RN$ such that supp $\psi\subset \{\xi: 1/4\le \|\xi\|\le 4\}$ and $\psi\equiv 1$ on $\{\xi: 1/2\le \|\xi\|\le 2\}$. If $m$  satisfies \eqref{smoothness condition}, then we have:
	\begin{enumerate}[\rm (a)]
		\item the  multiplier $T_m$ is bounded on $\FF^{s, L}_{p,q}(dw)$ provided that $s\in \mathbb{R}$, $0<p,q <\vc$ and $\alpha> \N(\f{1}{1\wedge p\wedge  q}-\f{1}{2}) $, i.e.,
		\[
		\|T_m\|_{\FF^{s,L}_{p,q}(dw)\to \FF^{s,L}_{p,q}(dw)}\lesi |m(0)|+\sup_{t>0}\|\psi(\cdot)m(t\cdot)\|_{W^{\vc}_\alpha(\RN)};
		\]
		\item the multiplier $T_m$ is bounded on $\B^{s,L}_{p,q}(dw)$ provided that $s\in \mathbb{R}$, $0<p<\vc, 0<q<\vc$ and $\alpha> \N(\f{1}{1\wedge p\wedge  q}-\f{1}{2}) $, i.e.,
		$$
		\|T_m\|_{\B^{s,L}_{p,q}(dw)\to \B^{s,L}_{p,q}(dw)}\lesi |m(0)| +\sup_{t>0}\|\psi(\cdot)m(t\cdot)\|_{W^{\vc}_\alpha(\RN)}.
		$$		
	\end{enumerate}
	Moreover, if  \eqref{eq-L1 uniform} holds true,  the estimates (a) and (b)  hold true with $W^{2}_\alpha(\RN)$ taking place of $W^{\vc}_\alpha(\RN)$ in \eqref{smoothness condition}.
\end{thm}
%\textbf{Note.} We will show that 
%\[
%LT_m f = T_m(Lf) = T_{\|\xi\|^2m(\xi)}f.
%\]
%Indeed, recall that 
%\[
%T_mf = \F^{-1}(m\F f).
%\]
%Since $\F(Lf)(\xi)= \|\xi\|^2 \F f(\xi)$, we have
%\[
%\F(LT_m f)(\xi)=\|\xi\|^2 \F(T_m f)(\xi)= \|\xi\|^2m(\xi)\F f(\xi)= %T_{\|\xi\|^2m(\xi)}f.
%\]
Once Theorem \ref{mainthm-spectralmultipliers-space adapted to L} has been proved, Theorem \ref{mainthm-spectralmultipliers-homogeneous spaces} follows immediately from Theorem \ref{mainthm-spectralmultipliers-space adapted to L} and Theorem \ref{thm-coincidence Besov and TL spaces}. In order to prove Theorem \ref{mainthm-spectralmultipliers-space adapted to L}, we need to establish a number of technical results.

\begin{prop}\label{prop-L2 norm via Linfty}
	If $m$ is a bounded function supported in $B(0,1/t)\subset \RN$, then we have
	\[
	\|T_m(\cdot, \y)\|_{L^2(dw)}\lesi \|m\|_\vc w(B(\y,t))^{-1/2}.
	\]
\end{prop}
\begin{proof}
	We write
	\[
	m(\xi) = m(\xi) e^{t^2\|\xi\|^2}e^{-t^2\|\xi\|^2} =:\widetilde{ m}(\xi)e^{-t^2\|\xi\|^2}.  
	\]
	It follows that 
	\[
	\begin{aligned}
		\F^{-1}m(\x) &=  \F^{-1}\widetilde{ m}\ast \F^{-1}(e^{-t^2\|\xi\|^2})(\x)\\
		&=\int_{\RN} \F^{-1}\widetilde{ m}(\z)\tau_{-\z}\F^{-1}(e^{-t^2\|\xi\|^2})(\x)dw(\z).
	\end{aligned}
	\]
	Therefore,
	\[
	\begin{aligned}
		\tau_{-\y}\F^{-1}m(\x) 
		&=\int_{\RN} \F^{-1}\widetilde{ m}(\z)\tau_{-\y}\tau_{-\z}\F^{-1}(e^{-t^2\|\xi\|^2})(\x)dw(\z).
	\end{aligned}
	\]
	On the other hand,
	\[
	\begin{aligned}
		\tau_{-\y}\tau_{-\z}\F^{-1}(e^{-t^2\|\xi\|^2})(\x)
		&=c_k^{-1}\int_{\RN} E(i\xi, -\z)E(i\xi, -\y)E(i\xi, \x) e^{-t^2\|\xi\|^2}dw(\xi)\\
		&=\tau_{-\z}\tau_{-\y}\F^{-1}(e^{-t^2\|\xi\|^2})(\x).
	\end{aligned}
	\]
	Consequently,
	\[
	\begin{aligned}
		\tau_{-\y}\F^{-1}m(\x)&=\int_{\RN} \F^{-1}\widetilde{ m}(\z)\tau_{-\z}\tau_{-\y}\F^{-1}(e^{-t^2\|\xi\|^2})(\x)dw(\z)\\
		&=T_{\widetilde m}\Big(\tau_{-\y}\F^{-1}(e^{-t^2\|\xi\|^2})\Big)(\x).
	\end{aligned}
	\]
	Hence,
	\[
	\begin{aligned}
		\|T_m(\cdot,\y)\|_{L^2(dw)}:=\|\tau_{-\y}\F^{-1}m\|_{L^2(dw)}&\le \|\widetilde m\|_\vc \Big\|\tau_{-\y}\F^{-1}(e^{-t^2\|\xi\|^2})\Big\|_{L^2(dw)}\\
		&\le \|\widetilde m\|_\vc \|h_t(\cdot, \y)\|_{L^2(dw)}\\
		&\lesi \|m\|_\vc w(B(\y,t))^{-1/2},
	\end{aligned}
	\]
	where in the second line we used \eqref{eq- heat kernel via Fourier transform}, and the last inequality is just a direct consequent of the Gaussian upper bound of $h_t(\x,\y)$ in Lemma \ref{lem-heat kernel estimate for semigroups}.
	
	This completes our proof.
\end{proof}

The following estimate is taken from Proposition 4.4 in \cite{DH}.
\begin{prop}\label{prop - L2 estimate for the translation}
	Let $\phi\in C_c(B(0,r))\subset \RN$ be a radial function with $r>0$.  Then, for any $\y\in \RN$, we have
	\[
	\|\tau_{-\y}\phi\|_{L^2(dw)} \lesi \f{r^{\mathfrak{N}}}{w(B(\y,r))^{1/2}}\|\phi\|_{L^\vc}.
	\]
\end{prop}

The following result regarding supports of translations is taken from Proposition 1.5 in \cite{DH}.
\begin{lem}
	\label{lem - support condition DH}
	Let $f\in L^2(dw)$, $\supp f\subset B(0,r)\subset \RN$ with $r>0$ and $\x\in \RN$. Then
	\[
	\supp \tau_\x f(-\cdot)\subset \mathcal O(B(\x,r)).
	\]
\end{lem}

In Proposition \ref{prop-kernel of spectral multiplier} , Proposition \ref{prop- sharp W2 base Lr norm} and Propositin \ref{prop-pointwise estimate kernel of spectral multiplier} below, we will assume that $m$ is a bounded function on $\RN$ and  $\psi$ is a smooth radial function defined on $\RN$ such that supp $\psi\subset \{\xi: 1/4\le \|\xi\|\le 4\}$ and $\psi\equiv 1$ on $\{\xi: 1/2\le \|\xi\|\le 2\}$. For each $\ell\in \mathbb Z$, define
\[
\widetilde m_{\ell}(\xi) = m(2^{\ell}\xi)\psi(\xi),
\]
and
\[
m_{\ell}(\xi) = m(\xi)\psi(2^{-\ell}\xi).
\]
We now have the following estimates on the kernel of $T_{m_\ell}$.

\begin{prop}\label{prop-kernel of spectral multiplier} For each $\ell$ and $\alpha, \epsilon>0$, there exists $C=C(\alpha,\epsilon)$ such that 
	\begin{equation}\label{eq- L2 norm Tm x y W infrty} 
		\int_{\RN} |T_{m_{\ell}}(\x,\y)|^2(1+2^\ell d(\x,\y))^{2\alpha} dw(\x) \le  \f{C}{w(B(\y,2^{-\ell}))}\|m(2^{\ell}\cdot)\psi\|_{W^\vc_{\alpha+\epsilon}(\RN)}
	\end{equation}
	for any $\y\in \RN$.
\end{prop}	
\begin{proof}

	We have
	\[
	\widetilde m_{\ell}(\xi) =\widetilde m_{\ell}(\xi)e^{\|\xi\|^2} e^{-\|\xi\|^2}=:m^{\ell}(\xi)e^{\|\xi\|^2}.
	\]	
	Recall that $h_t(\x)=\F^{-1}(e^{-t\|\xi\|^2})(\x)$. Then,
	\[
	\begin{aligned}
		T_{\widetilde m_{\ell}}(\x,\y)&=\tau_{-\y}\F^{-1}\widetilde m_{\ell}(\x)\\
		&= \tau_{-\y}\big[\F^{-1}(m^{\ell})\ast h_{1}\big](\x).
	\end{aligned}
	\]
	Since $\tau_{-\y}\tau_{\x} h_t(\z) =\tau_{\x}\tau_{-\y} h_t(\z)$ for $\x,\y,\z\in \RN$, we further obtain
	\begin{equation}\label{eq1- proof of kernel}
		\begin{aligned}
			T_{\widetilde m_{\ell}}(\x,\y)&= \big[\F^{-1}(m^{\ell})\ast \tau_{-\y}h_{1}\big](\x).
		\end{aligned}
	\end{equation}
	Let $\Psi_0\in C^\vc((-1,1))$ and $\Psi\in C^\vc(1/4,4)$ such that 
	\begin{equation*}
		\Psi_0(\|\x\|) + \sum_{j\ge 1}\Psi_k(\|\x\|) =1, \ \ \x\in \RN.
	\end{equation*}
	For $j\ge 0$, denote
	\[
	{M}_{\ell, j} (\xi) =  \F^{-1}(m^{\ell})(\xi)\Psi_j(\|\xi\|) \ \ \ \text{and} \ \ h_{1, j}(\xi)=h_{1}(\xi)\Psi_j(\|\xi\|).
	\]
	This, along with \eqref{eq1- proof of kernel}, implies that 
	\begin{equation*}
		\begin{aligned}
			T_{\widetilde m_{\ell}}(\x,\y)&=\tau_{-\y}\F^{-1}\widetilde m_{\ell}(\x)\\
			&=\sum_{j, k \in \mathbb N} \big[ M_{\ell,j}\ast \tau_{-\y}h_{1,k}\big](\x).
		\end{aligned}
	\end{equation*}
	Since 
	$$
	\big[ M_{\ell,j}\ast \tau_{-\y}h_{1,k}\big](\x)= \tau_{-\y}\big[ M_{\ell,j}\ast h_{1,k}\big](\x),
	$$
	using Lemma \ref{lem - support condition DH} we imply that $\big[ M_{\ell,j}\ast \tau_{-\y}h_{1,k}\big](\x)=0$ whenever $d(\x,\y)>2^j+2^k$.
	
	Consequently,
	\begin{equation}\label{eq2-proof of L2 Tm(x,y)}
		\begin{aligned}
			\Big(\int_{\RN} |T_{\widetilde m_{\ell}}(\x,\y)|^2(1+ d(\x,\y))^{2\alpha} dw(\x)\Big)^{1/2}&\le \sum_{j, k \in \mathbb N} \Big(\int_{\RN}\Big|\big[ M_{\ell,j}\ast \tau_{-\y}h_{1,k}\big](\x)\Big|^2(1+ d(\x,\y))^{2\alpha}dw(\x)\Big)^{1/2}\\
			&\lesi \sum_{j,k\in \mathbb N}(2^j+2^k)^\alpha \|M_{\ell,j}\ast \tau_{-\y}h_{1,k}\|_{L^2(dw)}\\
			&\lesi \sum_{j,k\in \mathbb N}(2^j+2^k)^\alpha \|M_{\ell,j}\|_{L^1(dw)}\|\tau_{-\y}h_{1,k}\|_{L^2(dw)}\\
			&\lesi \sum_{j,k\in \mathbb N}2^{j\alpha}2^{k\alpha} \|M_{\ell,j}\|_{L^1(dw)}\|\tau_{-\y}h_{1,k}\|_{L^2(dw)}.
		\end{aligned}
	\end{equation}
	Applying Proposition \ref{prop - L2 estimate for the translation} and using the fact that $h_{t}(\x)=c_k^{-1}(2t)^{-\N/2}e^{-\|\x\|^2/(4t)}$, we have, for $k\in \mathbb N$ and $\y\in \RN$,
	\[
	\begin{aligned}
		\|\tau_{-\y}h_{1,k}\|_{L^2(dw)}&\lesi \f{2^{k\mathfrak{N}}}{w(B(\y, 2^k))^{1/2}}\|h_{1,k}\|_{L^\vc}\\
		&\lesi \f{2^{k\mathfrak{N}}}{w(B(\y, 1))^{1/2}}\|h_{1,k}\|_{L^\vc}\\
		&\lesi \f{2^{k\mathfrak{N}}}{w(B(\y, 1))^{1/2}}e^{-c2^{2k}}.
	\end{aligned}
	\]
	This, along with \eqref{eq2-proof of L2 Tm(x,y)}, yields that
	\[
	\begin{aligned}
		\Big(\int_{\RN} |T_{\widetilde m_{\ell}}(\x,\y)|^2(1+ d(\x,\y))^{2\alpha} dw(\x)\Big)^{1/2}
		&\lesi \f{1}{w(B(\y, 1))^{1/2}}\sum_{j\in \mathbb N}2^{j\alpha}  \|M_{\ell,j}\|_{L^1(dw)}
	\end{aligned}
	\]
	for all $\y\in \RN$.

	On the other hand, for any $\epsilon>0$,
	\[
	\begin{aligned}
		\sum_{j\in \mathbb N}2^{j\alpha}  \|M_{\ell,j}\|_{L^1(dw)}&\simeq \sum_{j\in \mathbb N}   \|\F^{-1}(m^{\ell})(\xi)(1+\|\x\|)^s\Psi_j(\|\xi\|) \|_{L^1(dw)}\\
		&\simeq  \|\F^{-1}(m^{\ell})(\xi)(1+\|\x\|)^\alpha \|_{L^1(dw)}\\
		&\lesi  \|\F^{-1}(m^{\ell})(\xi)(1+\|\x\|)^{\alpha+\N/2+\epsilon} \|_{L^2(dw)}\|(1+\|\x\|)^{-\N/2-\epsilon}\|_{L^2(dw)}\\
		&\lesi  \|\F^{-1}(m^{\ell})(\xi)(1+\|\x\|)^{\alpha+\N/2+\epsilon} \|_{L^2(dw)}\\
		&\lesi \|m^{\ell}\|_{W^2_{\alpha+\mathfrak{N}/2+\epsilon}(\RN)}\simeq \|\widetilde m_{\ell}\|_{W^2_{\alpha+\mathfrak{N}/2+\epsilon}(\RN)},
	\end{aligned}
	\]
	where in the last inequality we used Proposition 5.3 in \cite{A etal}.
	
	It follows that for $\y\in \RN$,
	\begin{equation}
		\label{eq- T widetilde m}
		\Big(\int_{\RN} |T_{\widetilde m_{\ell}}(\x,\y)|^2(1+ d(\x,\y))^{2\alpha} dw(\x)\Big)^{1/2}
		\lesi \f{1}{w(B(\y, 1))^{1/2}}\|\widetilde m_{\ell}\|_{W^2_{\alpha+\mathfrak{N}/2+\epsilon}(\RN)}.
	\end{equation}

	By homogeneity,
	\[
	T_{m_\ell}(\x,\y)=2^{\mathfrak{N}\ell}T_{\widetilde m_\ell}(2^\ell\x,2^\ell\y).
	\]
	Therefore, for $\y\in \RN$,
	\[
	\begin{aligned}
		\int_{\RN} |T_{m_{\ell}}(\x,\y)|^2(1+2^\ell d(\x,\y))^{2\alpha} dw(\x) &=\int_{\RN} |T_{\widetilde m_\ell}(2^\ell\x,2^\ell\y)|^2(1+2^\ell d(\x,\y))^{2\alpha} 2^{2\mathfrak{N}\ell} dw(\x)\\
		& =2^{\mathfrak{N}\ell} \int_{\RN} |T_{\widetilde m_\ell}(2^\ell\x,2^\ell\y)|^2(1+ d(2^\ell\x,2^\ell\y))^{2\alpha} dw(2^{\ell}\x).
	\end{aligned}
	\]
	Applying \eqref{eq- T widetilde m} and \eqref{eq-homgoenous measure}, for $\y\in \RN$,
	\begin{equation}\label{eq-rough W2 base}
		\begin{aligned}
			\int_{\RN} |T_{m_{\ell}}(\x,\y)|^2(1+2^\ell d(\x,\y))^{2\alpha} dw(\x)&\lesi \f{2^{\mathfrak{N}\ell} }{w(B(2^{\ell}\y, 1))}\|\widetilde m_{\ell}\|^2_{W^2_{\alpha+\mathfrak{N}/2+\epsilon}(\RN)}\\
			&\simeq \f{1}{w(B(\y, 2^{-\ell}))}\|\widetilde m_{\ell}\|^2_{W^2_{\alpha+\mathfrak{N}/2+\epsilon}(\RN)}.
		\end{aligned}
	\end{equation}
	This, together with Proposition \ref{prop-L2 norm via Linfty} and the interpolation trick used in \cite{MM, DOS}, implies the desired estimate.

	This completes our proof.
\end{proof}

\begin{prop}
	\label{prop- sharp W2 base Lr norm}
	Assume the condition \eqref{eq-L1 uniform} holds true. Then for each $\ell \in \mathbb Z$, $r\in (1,2)$ and $\alpha> \alpha'>\mathfrak{N}/2$, there exists $C=C(r, \alpha,\alpha')$ such that  
	\begin{equation}\label{eq- Lr estimate on the kernel}
		\Big[\int_{\RN} |T_{m_{\ell}}(\x,\y)|^r(1+2^{\ell}d(\x,\y))^{r(\alpha'-\mathfrak{N}/2)} dw(\x)\Big]^{1/r} \le  \f{C}{w(B(\y, 2^{-\ell}))^{1-1/r}}\|\widetilde m_{\ell}\|_{W^2_{\alpha}(\RN)}
	\end{equation}
	for $\y\in \RN$.
\end{prop}

\begin{proof}
	In \cite{DH}, it was proved that under the condition \eqref{eq-L1 uniform}, we have, for $0<\alpha'<\alpha$ with $\alpha>\mathfrak{N}/2$,
	\begin{equation*}
		\int_{\RN} |T_{\widetilde m_{\ell}}(\x,\y)|(1+ d(\x,\y))^{\alpha'-\mathfrak{N}/2} dw(\x) 
		\lesi \|\widetilde m_{\ell}\|_{W^2_{\alpha}(\RN)}.
	\end{equation*}
	This, along with the fact homogeneity
	\[
	T_{m_\ell}(\x,\y)=2^{\mathfrak{N}\ell}T_{\widetilde m_\ell}(2^\ell\x,2^\ell\y),
	\]
	implies that 
	\begin{equation}\label{eq-L2 base for Tm x y}
		\int_{\RN} |T_{m_{\ell}}(\x,\y)|(1+ 2^{\ell}d(\x,\y))^{\alpha'-\mathfrak{N}/2} dw(\x) 
		\lesi \|\widetilde m_{\ell}\|_{W^2_{\alpha}(\RN)}
	\end{equation}
	for all $\y\in \RN$.

	On the other hand, we have proved in \eqref{eq-rough W2 base} that for $0<\alpha'<\alpha$ and $\alpha>\mathfrak{N}/2$,
	\begin{equation*} 
		\begin{aligned}
			\int_{\RN} |T_{m_{\ell}}(\x,\y)|^2(1+2^{\ell}d(\x,\y))^{2(\alpha'-\mathfrak{N}/2)} dw(\x)			&\simeq \f{1}{w(B(\y, 2^{-\ell}))}\|\widetilde m_{\ell}\|^2_{W^2_{\alpha}(\RN)}
		\end{aligned}
	\end{equation*}
	for all $\y\in \RN$.
	
	Then interpolating these two estimates above, we derive the estimate \eqref{eq- Lr estimate on the kernel}.
	
	This completes our proof.
\end{proof}

\begin{prop}\label{prop-pointwise estimate kernel of spectral multiplier} 
	For every $\ell\in \Z$ and $\alpha>\alpha'>0$, there exists $C=C(\alpha,\alpha')$
	\begin{equation}\label{eq1- pointwise Tm}
		|T_{m_{\ell}}(\x,\y)|  \le C \f{(2^{\ell} d(\x,\y))^{-\alpha'}}{ w(B(\y,2^{-\ell}))}\|m(2^{\ell}\cdot)\psi\|_{W^\vc_{\alpha}(\RN)}, \ \ \ \x,\y\in \RN.
	\end{equation}
	Moreover, if \eqref{eq-L1 uniform} holds, then for every $\ell\in \Z$ and $\alpha>\alpha'>\N/2$, there exists $C=C(\alpha,\alpha')$ such that 
	\begin{equation}\label{eq2- pointwise Tm}
		|T_{m_{\ell}}(\x,\y)|  \le  C\f{(2^\ell d(\x,\y))^{-(\alpha'-\N/2)}}{ w(B(\y,2^{-\ell}))}\|m(2^{\ell}\cdot)\psi\|_{W^2_{\alpha}(\RN)}, \ \ \ \x,\y\in \RN.
	\end{equation}
\end{prop}	
\begin{proof}
	We prove \eqref{eq2- pointwise Tm} first. We first write
	\[
	m_\ell(\xi) = m_\ell(\xi) e^{2^{-2\ell}\|\xi\|^2}e^{-2^{-2\ell}\|\xi\|^2},
	\]
	which implies 
	\[
	\begin{aligned}
		T_{m_\ell}(\x,\y) = \int_{\RN} T_{m_\ell(\xi) e^{2^{-2\ell}\|\xi\|^2}}(\x,\z)h_{2^{-2\ell}}(\z,\y)dw(\z).
	\end{aligned}
	\]
	Hence,
	\[
	\begin{aligned}
		(2^\ell d(\x,\y))^{\alpha'-\N/2}|T_{m_\ell}(\x,\y)| &\le \int_{\RN} |T_{m_\ell(\xi) e^{2^{-2\ell}\|\xi\|^2}}(\x,\z)|(1+2^\ell d(\x,\z))^{\alpha'-\N/2}dw(z)\\
		& \ \ \ \times  \sup_{z} h_{2^{-2\ell}}(\z,\y)(1+2^\ell d(\z,\y))^{\alpha'-\N/2}.
	\end{aligned}
	\]
	By \eqref{eq-L2 base for Tm x y} and Lemma \ref{lem-heat kernel estimate for semigroups},
	\[
	\begin{aligned}
		(2^\ell d(\x,\y))^{\alpha'-\N/2}|T_{m_\ell}(\x,\y)| &\lesi \f{1}{w(B(\y, 2^\ell))}\|\widetilde m_\ell\|_{W^2_{\alpha}(\RN)}.
	\end{aligned}
	\]
	It follows that
	\[
	T_{m_\ell}(\x,\y)| \lesi \f{(2^\ell d(\x,\y))^{-(\alpha'-\N/2)}}{w(B(\y, 2^\ell))}\|\widetilde m_\ell\|_{W^2_{\alpha}(\RN)}.
	\]
	The proof of \eqref{eq1- pointwise Tm} is similar by using \eqref{eq- L2 norm Tm x y W infrty} instead of \eqref{eq-L2 base for Tm x y} and hence we omit the details.
	
	This completes our proof.
\end{proof}

We now prove the following result.
\begin{thm}
	\label{thm-bounded on Hardy spaces}
	Let $p\in (0,1]$ and $\alpha>\mathfrak{N}(1/p-1/2)$. If \eqref{smoothness condition} holds true, then $T_m$ is bounded on $H^p_L(dw)$.
	
	Moreover, if \eqref{eq-L1 uniform} holds true, then $T_m$ is bounded on $H^p_L(dw)$ provided that \eqref{smoothness condition} holds true with $W^2_\alpha(\RN)$ taking place of $W^\vc_\alpha(\RN)$.
\end{thm}
\begin{proof}
	We will only prove the theorem under the condition 	\eqref{smoothness condition}  with respect to the norm on  $W^2_\alpha(\RN)$ instead of  $W^\vc_\alpha(\RN)$. The proof under the norm $W^\vc_\alpha(\RN)$ is similar and even easier.
	
	Let $0< p \le 1$ and $\alpha>\alpha'>\mathfrak{N}(1/p-1/2)$. Then $\alpha'-\mathfrak{N}/2 > \mathfrak{N}(1/p-1)$. Fix $r\in (0,1)$ such that $\alpha'-\mathfrak{N}/2 > \mathfrak{N}(1/p-1/r)$.
	
	Let $\varphi \in C_0^\infty(0, \infty)$
	such that $\supp \varphi \subset [1/4, 4]$ and
	\begin{align*}
		\sum_{\ell \in \mathbb{Z}}\varphi(2^{-\ell}\lambda) =1 \quad \mbox{for all }  \lambda \in (0, \infty).
	\end{align*}
	By Theorem \ref{equiv-Hardy}, it suffices to show that there exists a constant $C$ such that for any $(p, r, M, L)$-atom $a$,
	\begin{align*}
		\big\|S_{L, \varphi}( T_ma)\big\|_{L^p(dw)} \leq C,
	\end{align*}
	where $S_{L,\varphi}$ is the Littlewood--Paley operator defined by
	\begin{align*}
		S_{L, \varphi}f:= \left(\sum_{\ell \in \mathbb{Z}} |\varphi(2^{-\ell}\sqrt L)f|^2 \right)^{1/2}.
	\end{align*}

	Fix $M>\max\Big\{\f{\N}{2}(\f{1}{p}-\f{2}{\N r}), (\alpha'-\N/2)/2 \Big\}$. Let $a$ be an arbitrary $(p,r,M,L)$-atom associated to some ball $\mathcal O(B)$ with  $B=  B(x_B, r_B)$, and let $b$ be the corresponding function such that
	$a = L^M b$. By the spectral theory, we write
	\begin{align} \label{identity}
		I = (I -e^{-r_B^2 L})^M +\sum_{k =1}^{M} (-1)^{k+1}C^M_k e^{-kr_B^2 L} =: (I -e^{-r_B^2 L})^M + P(r_B^2 L),
	\end{align}
	where $P(\cdot)$ is a function on $[0,\infty)$ given by
	\begin{align*}
		P(\lambda) = \sum_{k =1}^{M} (-1)^{k+1}C^M_k e^{-k\lambda}, \quad \lambda \in [0, \infty).
	\end{align*}
	Using the identity \eqref{identity},
	\begin{align*}
		\big\|S_{L, \varphi}\big(T_m a \big)\|_{L^p(\RN)}^p &=
		\left\| S_{L, \varphi} \big[ (I -e^{-r_B^2 L})^M T_ma\big]\right\|_{L^p(dw)}^p
		+ \left\| S_{L, \varphi} \big[ (r_B^2 L)^M P(r_B^2 L)T_m  r_B^{-2M}b \big]\right\|_{L^p(dw)}^p\\
		& = :  H_{1} +   H_2.
	\end{align*}
	Hence it suffices to prove
	\begin{align} \label{e1k}
		H_{1} + H_2 \lesssim 1.
	\end{align}

	\medskip
	Since the estimates of $H_1$ and $H_2$ share some similarities, it suffices to estimate $H_1$. To do this, we write
	\begin{align*}
		H_1 &= \left\|S_{L, \varphi} \big[ (I -e^{-r_B^2 L})^M T_ma\big] \right\|_{L^p(4\mathcal O(B), dw)}^p +
		\left\|S_{L, \varphi} \big[ (I -e^{-r_B^2 L})^M T_ma\big] \right\|_{L^p(\RN\backslash 4\mathcal O(B), dw)}^p \\
		& =: H_{11} + H_{12}.
	\end{align*}
	By H\"{o}lder's inequality, the $L^r$ boundedness of $S_{L, \varphi}$ (see Theorem \ref{equiv-Hardy}), $(I -e^{-r_B^2 L})^M$ and $T_m$, the properties of atoms, 
	we have
	\begin{equation} \label{e11}
		\begin{split}
			H_{11} & \lesssim w(4 \mathcal O(B))^{1-p/2} \big\|S_{L,\varphi}(I -e^{-r_B^2L})^M T_ma\big\|_{L^r(4\mathcal O(B),dw)}^p \\
			&\lesssim  w(4\mathcal O(B))^{1-p/r}  \|a\|_{L^r(dw)}^p \\
			&  \lesssim w(\mathcal O(B))^{1-p/r} w(\mathcal O(B))^{p/r -1} \\
			&\simeq 1.
		\end{split}
	\end{equation}
	
	Now we estimate $H_{12}$. Define
	\begin{align*}
		m_{\ell, r_B} (\xi) := \widetilde\varphi(2^{-\ell}\xi) (1-e^{-r_B^2 \|\xi\|^2})^M m(\xi),
	\end{align*}
	where $\widetilde\varphi(\xi)=\varphi(\|\xi\|)$.
	Then, by \eqref{eq-spectral multiplier and Tm}, we have
	\begin{equation} \label{decom}
		\begin{split}
			H_{12}&\lesssim \sum_{ \ell \in \Z}  \big\|T_{m_{\ell, r_B}}a\big\|_{L^p(\RN\backslash  4\mathcal O(B), dw)}^p\\
			& \lesi  \sum_{\ell \in \Z}  \sum_{j \geq 2} \big\|T_{m_{\ell, r_B}}a\big\|_{L^p (S_j(\mathcal O(B)), dw)}^p,
		\end{split}
	\end{equation}
	where $S_j(\mathcal O(B))=2^{j+1}\mathcal O(B)\backslash 2^{j}\mathcal O(B)$ for $j\ge 2$.

	Using H\"{o}lder's inequality, the doubling property \eqref{eq-ratios on volumes of balls} and the properties of atoms,
	\begin{equation} \label{msn}
		\begin{aligned}
			\big\|T_{m_{\ell, r_B}}a\big\|_{L^p (S_j(\mathcal O(B)), dw)}^p & \lesssim
			\sum_{\ell \in \Z} w(\mathcal O(2^jB))^{1-p/r}\big\|T_{m_{\ell,r_B}}a\big\|_{L^r (S_j(\mathcal O(B)), dw) }^p\\
			& \lesssim
			\sum_{\ell \in \Z} 2^{j\mathfrak{N}(1-p/r)}w(\mathcal O(B))^{1-p/r}\big\|T_{m_{\ell,r_B}}a\big\|_{L^r (S_j(\mathcal O(B)), dw) }^p
		\end{aligned}
	\end{equation}	
	We have 
	\begin{equation*}
		\begin{aligned}
			\big\|T_{m_{\ell,r_B}}a\big\|_{L^r (S_j(\mathcal O(B)), dw) } &=\Big(\int_{S_j(\mathcal O(B))}\Big|\int_{\mathcal O(B)} T_{m_{\ell,r_B}}(\x,\y)a(\y)dw(\y) \Big|dw(\x)\Big)^{1/r}\\
			&\lesi  \int_{\mathcal O(B)} \Big(\int_{S_j(\mathcal O(B))}|T_{m_{\ell,r_B}}(\x,\y)|^r dw(\x)\Big)^{1/r}|a(\y)|dw(\y). 
		\end{aligned}
	\end{equation*}
	For $\x \in \mathcal O(B)$, by Proposition \ref{prop- sharp W2 base Lr norm},
	\begin{equation*}
		\begin{aligned}
			\Big(\int_{S_j(\mathcal O(B))}&|T_{m_{\ell,r_B}}(\x,\y)|^r dw(\x)\Big)^{1/r}\\
			&\simeq (1+2^{\ell+j}r_B)^{-(\alpha'-\mathfrak{N}/2)}\Big(\int_{S_j(\mathcal O(B))}|T_{m_{\ell,r_B}}(\x,\y)|^r (1+2^\ell d(\x,\y))^{r(\alpha'-\mathfrak{N}/2)} dw(\x)\Big)^{1/r}\\
			& \lesi \sup_{\y\in \mathcal O(B)}\f{(1+2^{\ell+j}r_B)^{-(\alpha'-\mathfrak{N}/2)}}{w(B(\y,2^{-\ell}))^{1-1/r}}\|\widetilde\varphi \delta_{2^\ell}(m(\cdot)(1-e^{r_B^2\|\cdot\|^2}))^M\|_{W^2_\alpha(\RN)}.
		\end{aligned}
	\end{equation*}
	It can be verified that 
	\[
	\|\widetilde\varphi \delta_{2^\ell}(m(\cdot)(1-e^{r_B^2\|\cdot\|^2}))^M\|_{W^2_\alpha(\RN)} \lesi \min\{1, (2^\ell r_B)^{2M}\}.
	\]
	Therefore,
	\begin{equation*}
		\begin{aligned}
			\big\|T_{m_{\ell,r_B}}a\big\|_{L^r (S_j(\mathcal O(B)), dw) } &\lesi \|a\|_{L^1(dw)} \sup_{\y\in \mathcal O(B)}\f{(1+2^{\ell+j}r_B)^{-(\alpha'-\mathfrak{N}/2)}}{w(B(\y,2^{-\ell}))^{1-1/r}} \min\{1, (2^\ell r_B)^{2M}\}\\
			&\lesi w(\mathcal O(B))^{1-1/p} \sup_{\y\in \mathcal O(B)}\f{(1+2^{\ell+j}r_B)^{-(\alpha'-\mathfrak{N}/2)}}{w(B(\y,2^{-\ell}))^{1-1/r}} \min\{1, (2^\ell r_B)^{2M}\}
		\end{aligned}
	\end{equation*}
	Inserting this into \eqref{msn},
	\[
	\begin{aligned}
		\big\|T_{m_{\ell, r_B}}a\big\|_{L^p (S_j(\mathcal O(B)), dw)}^p\lesi \sum_{\ell\in \Z} 2^{j\mathfrak{N}(1-p/r)}w(\mathcal O(B))^{p-p/r}\sup_{\y\in \mathcal O(B)}\f{(1+2^{\ell+j}r_B)^{-p(\alpha'-\mathfrak{N}/2)}}{w(B(\y,2^{-\ell}))^{p-p/r}} \min\{1, (2^\ell r_B)^{2Mp}\}.
	\end{aligned}
	\]
	On the other hand, by the doubling property \eqref{eq-ratios on volumes of balls},
	\[
	\sup_{\y\in \mathcal O(B)}\f{w(\mathcal O(B))}{w(B(\y,2^{-\ell}))}\lesi (1 + 2^\ell r_B)^{\mathfrak{N}}
	\]
	Consequently,
	\[
	\begin{aligned}
		\big\|T_{m_{\ell, r_B}}a\big\|_{L^p (S_j(\mathcal O(B)), dw)}^p&\lesi \sum_{\ell\in \Z}   2^{j\mathfrak{N}(1-p/r)}(1+2^{\ell+j}r_B)^{-p(\alpha'-\mathfrak{N}/2)}(1 + 2^\ell r_B)^{p\mathfrak{N}(1-1/r)}\min\{1, (2^\ell r_B)^{2Mp}\}\\
		&\lesi \sum_{\ell\in \Z}   2^{-jp[(\alpha'-\mathfrak{N}/2)-\N(1/p-1/r)]}( 2^{\ell}r_B)^{-p(\alpha'-\mathfrak{N}/2)}(1 + 2^\ell r_B)^{p\mathfrak{N}(1-1/r)}\min\{1, (2^\ell r_B)^{2Mp}\}\\
		&\lesi 2^{-jp[(\alpha'-\mathfrak{N}/2)-\N(1/p-1/r)]},
	\end{aligned}
	\]
	provided that $2M > (\alpha'-\mathfrak{N}/2)$.	
	
	As a result,
	\[
	H_{12}\lesi \sum_{j\ge 2}2^{-jp[(\alpha'-\mathfrak{N}/2)-\N(1/p-1/r)]}\lesi 1,
	\]
	as long as $\alpha'-\mathfrak{N}/2 > \mathfrak{N}(1/p-1/r)$.
	
	Taking this and \eqref{e11} into account, we come up with 
	\[
	H_1\lesi 1.
	\]

	Hence, this completes our proof.
	
\end{proof}

We are ready to give the proof for Theorem \ref{mainthm-spectralmultipliers-space adapted to L} .
\begin{proof}
	The proof of (b) is similar to that of (a). Hence, it suffices to prove (a). To do this, we first note that from Theorem \ref{equiv-Hardy} and Theorem \ref{thm-bounded on Hardy spaces}, for $p\in (0,1]$,
	\begin{equation}\label{eq- Hardy bounded Tm}
		\|T_m\|_{\FF^{0, L}_{p,2}(dw)\to \FF^{0, L}_{p,2}(dw)}\lesi |m(0)|+\sup_{t>0}\|\psi(\cdot)m(t\cdot)\|_{W^{2}_\alpha(\RN)},
	\end{equation}
	provided that $\alpha >\N(\f{1}{p}-\f{1}{2})$.
	
	On the other hand, it was proved in \cite{DH} that for $p\in (1,\vc)$,
	\begin{equation}\label{eq- Lp bounded Tm}
		\|T_m\|_{L^p(dw)\to L^p(dw)}\lesi |m(0)|+\sup_{t>0}\|\psi(\cdot)m(t\cdot)\|_{W^{\vc}_\alpha(\RN)},
	\end{equation}
	provided that $\alpha >\f{\N}{2}$, which, along with Theorem \ref{equiv-Hardy}, implies
	\begin{equation*}%\label{eq- Hardy bounded Tm}
		\|T_m\|_{\FF^{0, L}_{p,2}(dw)\to \FF^{0, L}_{p,2}(dw)}\lesi |m(0)|+\sup_{t>0}\|\psi(\cdot)m(t\cdot)\|_{W^{2}_\alpha(\RN)},
	\end{equation*}
	provided that $\alpha >\f{\N}{2}$ and $p\in (1,\vc)$.
	
	Let $\psi$ be a partition of unity. By Proposition \ref{prop - equivalence of varphi*}, Theorem \ref{equiv-Hardy} and the fact $L^{s/2}\circ T_m = T_m \circ L^{s/2}$, we have, for $s\in \mathbb R$ and $p\in (0,1)$,
	\[
	\begin{aligned}
		\|T_mf\|_{\FF^{s, L}_{p,2}(dw)} &\simeq \Big\|\Big[\sum_{j\in \mathbb{Z}}(\delta^{-js}|(\delta^j \sqrt L)^s\psi_{j}(\sqrt{L})T_mf|)^q\Big]^{1/q}\Big\|_{L^p(dw)}\\
		& \simeq \Big\|\Big[\sum_{j\in \mathbb{Z}}(|\psi_{j}(\sqrt{L})T_m(L^{s/2}f)|)^q\Big]^{1/q}\Big\|_{L^p(dw)}\\
		&\simeq \|T_m(L^{s/2}f)\|_{L^p(dw)}
	\end{aligned}
	\]
	Applying \eqref{eq- Lp bounded Tm} and Theorem \ref{equiv-Hardy},
	\[
	\begin{aligned}
		\|T_mf\|_{\FF^{s, L}_{p,2}(dw)} &\lesi  \| L^{s/2}f\|_{L^p(dw)}\\
		&\simeq \|f\|_{\FF^{s, L}_{p,2}(dw)},	
	\end{aligned}
	\]
	provided that $s\in \mathbb R$, $\alpha >\f{\N}{2}$ and $p\in (1,\vc)$.
	
	This, along with 	\eqref{eq- Hardy bounded Tm} and the interpolation in Proposition \ref{prop-comple interpolation},  yields
	\begin{equation}\label{eq- Fsp2 bounded Tm}
		\|T_m\|_{\FF^{s, L}_{p,2}(dw)\to \FF^{s, L}_{p,2}(dw)}\lesi |m(0)|+\sup_{t>0}\|\psi(\cdot)m(t\cdot)\|_{W^{2}_\alpha(\RN)},
	\end{equation}
	provided that $s\in \mathbb R$, $p\in (0,\vc)$ and $\alpha >\N(\f{1}{1\wedge p }-\f{1}{2})$.
	
	From Proposition \ref{prop-pointwise estimate kernel of spectral multiplier}, using the similar argument used in the proof of Theorem 4.1 (b) in \cite{BD} we can show that
	\[
	\|T_m\|_{\FF^{s, L}_{p,q}(dw)\to \FF^{s, L}_{p,q}(dw)}\lesi |m(0)|+\sup_{t>0}\|\psi(\cdot)m(t\cdot)\|_{W^{2}_\alpha(\RN)},
	\] 
	provided that $s\in \mathbb R, 0<p,q<\vc$ with $\alpha> \f{\N}{1\wedge p\wedge q}+\f{\N}{2}$.
	
	At this stage, using this, \eqref{eq- Fsp2 bounded Tm}, Proposition \ref{prop-duality}, Proposition \ref{prop-comple interpolation} and Proposition \ref{mainthm-Interpolation}, and by the argument used in the proof of Theorem 1.2 in \cite{BD}, we will come up with
	\[
	\|T_m\|_{\FF^{s, L}_{p,q}(dw)\to \FF^{s, L}_{p,q}(dw)}\lesi |m(0)|+\sup_{t>0}\|\psi(\cdot)m(t\cdot)\|_{W^{2}_\alpha(\RN)},
	\] 
	provided that $s\in \mathbb R, 0<p,q<\vc$ with $\alpha> \N(\f{1}{1\wedge p\wedge q}-\f{1}{2})$.
	
	This completes our proof.
\end{proof}

\vskip1cm

\noindent \textbf{Funding}

The author was supported by the Australian Research Council via the grant ARC DP220100285. 

\medskip
\noindent \textbf{Declaration of competing interest}

The author has no any conflicts of interest/competing interests in the manuscript.

\medskip

\noindent{\bf Acknowledgment}  

The author was supported by the Australian Research Council via the grant ARC DP220100285. The author would like to thank the referee for his/her useful comments which helped to improve the paper. He would also like to thank Ji Li and Dachun Yang for useful discussions on Besov and Triebel-Lizorkin spaces on the space of homogeneous type.
%\bibliographystyle{plain}
%\bibliography{coulomb}
%\end{document}

\bigskip
%\centerline{\textbf{References}}

\end{document}